\documentclass[a4paper,10pt]{amsart} 
\usepackage{amsmath,amsxtra,amssymb,latexsym, amscd,amsthm}
\usepackage[utf8]{inputenc}
\usepackage{geometry}
\usepackage{color}
\usepackage{graphicx}
\usepackage{amssymb}
\usepackage{epstopdf}
\usepackage{mathrsfs}
\usepackage{amsmath}
\usepackage{amsthm}
\usepackage{bbm}
\usepackage{enumerate}
\usepackage[dvipsnames]{xcolor}
\usepackage[colorlinks=true,%
            linkcolor=red!75!black,%
            citecolor=blue!75!black,%
            urlcolor=darkgray]{hyperref}  
\numberwithin{equation}{section}
\makeatletter
\newcommand*\bigcdot{\mathpalette\bigcdot@{.5}}
\newcommand*\bigcdot@[2]{\mathbin{\vcenter{\hbox{\scalebox{#2}{$\m@th#1\bullet$}}}}}
\makeatother
\newcommand\R{\mathbb{R}}
\newcommand\Z{\mathbb{Z}}
\newcommand\N{\mathbb{N}}
\newcommand\C{\mathbb{C}}
\newcommand\T{\mathbb{T}}
\newcommand\Tr{\mathrm{Tr}}
\newcommand\id{\mathrm{Id}}
\newcommand{\cA}{\mathcal{A}}
\newcommand{\cQ}{\mathcal{Q}}
\newcommand{\cB}{\mathcal{B}}

\newcommand{\cP}{\mathcal{P}}
\newcommand{\cL}{\mathcal{L}}

\newcommand{\cG}{\mathcal{G}}
\newcommand{\cT}{\mathcal{T}}
\newcommand{\cU}{\mathcal{U}}
\newcommand{\cS}{\mathcal{S}}
\newcommand{\cC}{\mathcal{C}}
\newcommand{\cE}{\mathcal{E}}

\newcommand{\To}{\longrightarrow}
\DeclareMathOperator{\supp}{supp}

\DeclareMathOperator{\Lip}{Lip}

\begin{document}

\newcommand{\dd}{\mathrm{d}}
\newcommand{\ee}{\mathrm{e}}
\newcommand{\ii}{\mathrm{i}}
\newcommand{\rme}{\ee}  %
\newcommand{\rmi}{\ii}  
\renewcommand{\Re}{\mathop{\mathfrak{Re}}}
\renewcommand{\Im}{\mathop{\mathfrak{Im}}}
\newtheorem{theoreme}{Theorem}[section]
\newtheorem{definition}[theoreme]{Definition}
\newtheorem{lemme}[theoreme]{Lemma}
\newtheorem{remarque}[theoreme]{Remark}
\newtheorem{exemple}[theoreme]{Example}
\newtheorem{proposition}[theoreme]{Proposition}
\newtheorem{corollaire}[theoreme]{Corollary}
\newtheorem{hyp}{Hypothesis}
\newtheorem*{theo}{Theorem}
\newtheorem*{prop}{Proposition}
\newtheorem*{main*}{Main results}

\newcommand{\mar}[1]{{\marginpar{\sffamily{\scriptsize
        #1}}}}
\newcommand{\mi}[1]{{\mar{MI:#1}}}

\title{Empirical spectral measures of quantum graphs in the Benjamini-Schramm limit}
\author{Nalini Anantharaman, Maxime Ingremeau, Mostafa Sabri, Brian Winn}
\date{}
\address{Universit\'e de Strasbourg, CNRS, IRMA UMR 7501, F-67000 Strasbourg, France.}
\email{anantharaman@math.unistra.fr}
\address{Laboratoire J.A.Dieudonn\'e, UMR CNRS-UNS 7351, Universit\'e C\^ote
  d'Azur, 06000 Nice, France}
\email{maxime.ingremeau@univ-cotedazur.fr}
\address{Department of Mathematics, Faculty of Science, Cairo University, Cairo 12613, Egypt.}
\email{mmsabri@sci.cu.edu.eg}
\address{Department of Mathematical Sciences, Loughborough University, Leicestershire, LE11 3TU, United Kingdom.}
\email{b.winn@lboro.ac.uk}

\begin{abstract}
We introduce the notion of Benjamini-Schramm convergence for quantum graphs. This notion of convergence, intended to play the role of the already existing notion for discrete graphs, means that the restriction of the quantum graph to a randomly chosen ball has a limiting distribution. We prove that any sequence of quantum graphs with uniformly bounded data has a convergent subsequence in this sense. We then consider the empirical spectral measure of a convergent sequence (with general boundary conditions and edge potentials) and show that it converges to the expected spectral measure of the limiting random rooted quantum graph. These results are similar to the discrete case, but the proofs are significantly different.
\end{abstract}

\maketitle

\section{Introduction}

The study of empirical spectral measures (ESM) is a subject of rich
history, going back at least to the works of Wigner in the 50s
\cite{wig:cvo, wig:otd}. The ESM is the normalized count of
eigenvalues in a set. Wigner showed that the ESM of certain random
matrices arising from the study of heavy nuclei converges in
probability to the \emph{semicircular distribution}
$\mu_{\mathrm{sc}} = \frac{1}{2\pi}(4-|x|^2)_+^{1/2}\,\dd x$. This gives a
general idea of how the spectrum of these matrices looks like when the
size of the matrix gets large. Since then, a large body of literature
has been devoted to generalizing this result to other classes of
random matrices, considering spectral windows that shrink with $N$,
and also studying non-Hermitian matrices, so that the spectrum now lies
in $\C$---the ``universal'' analog in this case is the \emph{circular
  law} $\mu_{\mathrm c} = \frac{1}{\pi} \mathbf{1}_{|z|\le 1}$
\cite{gir:cl, bai:cl, tao:rmu}.

In the framework of graphs, a well-known result called the \emph{law of Kesten-McKay} \cite{Ke59,Mc81} says the following: if $G_N=(V_N, E_N)$ is a sequence of $(q+1)$-regular graphs with a negligible number of cycles (which is generically true if the $G_N$ are chosen at random), and if $(\lambda_j^{(N)})$ are the eigenvalues of the adjacency matrix $\cA_{G_N}$, then for any bounded continuous
function $\chi$,
\begin{equation}\label{e:kestenmckay}
  \frac1{|V_N|}   \sum_{i=1}^{|V_N|}\chi(\lambda_i^{(N)})
  \underset{N\To+\infty}{\To} \frac{1}{\pi}\int_{-\infty}^\infty
  \chi(\lambda)\Im \big\{G_{\lambda+\ii0}^{\T_q}(o,o)\big\}\,\dd \lambda
  = \int_{-\infty}^\infty \chi(\lambda)\, \dd\mu_o^{\T_q}(\lambda)\,,
\end{equation}
where $G_z^{\T_q}(o,o)$ is the Green's kernel of the adjacency matrix $\cA_{\T_q}$ of the $(q+1)$-regular tree $\T_q$ at a vertex $o\in \T_q$. More precisely, $G_z^{\T_q}(o,o) = \langle \delta_o, (\cA_{\T_q}-z)^{-1}\delta_o\rangle$ for $z\in \C\setminus \R$, and $G_{\lambda+\ii0}^{\T_q}(o,o)$ is its limit when $z$ approaches $\lambda$ from the upper half-plane. Here $\mu_o^{\T_q}(I) = \langle \delta_o, \mathbf{1}_I(\cA_{\T_q})\delta_o\rangle$ is the \emph{spectral measure} of $\cA_{\T_q}$ at the vertex $o$, and it is in fact independent of the vertex (as $\T_q$ is homogeneous). The density $\Im G_{\lambda+\ii0}^{\T_q}(o,o)$ is given explicitly by
$\frac{(q+1)\sqrt{4q-\lambda^2}}{2\pi[(q+1)^2-\lambda^2]}
\mathbf{1}_{|\lambda|\le 2\sqrt{q}}$.  

A good way to interpret \eqref{e:kestenmckay} is to say that such regular graphs ``converge'' to $\T_q$ due to the assumption of few-cycles, and consequently the ESM of $\cA_{G_N}$ converges to the spectral measure of $\cA_{\T_q}$ at some point $o$. The pertinent notion of convergence of graphs here is the \emph{Benjamini-Schramm} or \emph{local-weak} convergence \cite{BS,AL}, which studies how $k$-balls in $G_N$ look like as $N$ gets large. More precisely, we say that a sequence of finite graphs $(G_N)$ converges if for any finite rooted graph $(F,o)$ and $k\in\N$, the fraction of points $x\in G_N$ such that $B_{G_N}(x,k) \cong (F,o)$ has a limit. In the case of regular graphs with few short cycles, the limit is either zero or one, depending if $(F,o) \cong B_{\T_q}(o,k)$ or not, and we call $\delta_{[\T_q,o]}$ the limit of $(G_N)$. In general one obtains a probability distribution $\rho$ on the set of rooted graphs. The limit is then interpreted as a random rooted graph $\{(G,o)\}$, distributed according to $\rho$.

It turns out that \eqref{e:kestenmckay} is then a special case of a general fact: if $(G_N)$ converges to $\rho$ in the sense of Benjamini-Schramm, then the ESM of $G_N$ converges to the mean spectral measure $\mathbb{E}_{\rho}(\mu_o(\cdot)) = \int \mu_o(\cdot)\,\dd\rho([G,o])$. Here $\mu_o(I) = \langle \delta_o, \mathbf{1}_I(\cA_G)\delta_o\rangle$, and the integral runs over (isomorphism classes of) rooted graphs $(G,o)$. This fact is well-known, see e.g.\ \cite{ATV}. This can be generalized to the case of weighted graphs, with weights on the edges and/or vertices (see e.g. \cite[Appendix A]{AS2}). The adjacency matrix is then replaced by some weighted Schr\"odinger operator.

Hence, if one knows that $(G_N)$ converges to $\{(G,o)\}$, and if one has a good knowledge of the limiting operator $\cA_G$ (as in $\T_q$), then one gets spectral information about $G_N$ for large $N$. But things can also go the other way around. For example, in the theory of random Schr\"odinger operators, the model of interest is usually the infinite one, such as $H_{\Z^d}(\omega) = \cA_{\Z^d} + W^{\omega}$ on $\Z^d$, where $W^{\omega}$ is a random i.i.d.\ potential on the vertices. If $\Lambda_N = \{v\in\Z^d:\|v\|_{\infty}\le N\}$ is a set of cubes, it is known that
\begin{equation}\label{e:densityofstates}
\frac{\# \{\lambda_i^{(N)}(H_{\Lambda_N}(\omega))\in I\}}{|\Lambda_N|} \underset{N\To+\infty}{\To} \mathbf{E}\big[\mu_0^{H_{\Z^d}(\omega)}(I)\big]
\end{equation}
both almost-surely and in expectation, see e.g.\ \cite[Section 5]{Kirsch}. The boldface $\mathbf{E}$ denotes expectation w.r.t.\ the random potential. Though generally not presented this way, this is actually a consequence of the a.s.\ Benjamini-Schramm convergence of such cubes to $(\Z^d,W^{\omega},0)$. The quantity in the RHS is called the \emph{integrated density of states} in this context. Usually, to prove regularity
 of the integrated density of states, one first proves a \emph{Wegner estimate} valid for any finite $N$, then passes to the limit.
Thus in this case, a fact on the infinite graph was deduced by constructing a sequence of graphs $(\Lambda_N,W^{\omega}_N)$ which converges to $(\Z^d,W^{\omega},0)$. 

\medskip

Our aim in this paper is to introduce the notion of Benjamini-Schramm convergence for \emph{quantum graphs}. A quantum graph is a graph $G$ in which each edge is endowed a length $L_e$, and one studies a second-order (Sturm-Liouville type) differential operator on the edges, satisfying appropriate boundary conditions at the vertices \cite{Kuch}. Here we consider Schr\"odinger operators with edge-potentials and general boundary conditions. We give the precise definition of quantum graphs and the corresponding Benjamini-Schramm topology in Sections~\ref{sec:QG} and \ref{sec:BS}, respectively. We summarize our main results as follows and refer the reader to these sections for details:

\begin{main*}
Let $(\cQ_N)$ be a sequence of quantum graphs with uniformly bounded data (degree, length, edge potential, boundary conditions). Denote by $H_{\cQ_N}$ the corresponding Schr\"odinger operator. Then:-
\begin{enumerate}[\rm 1.]
\item Up to passing to a subsequence, the graphs $(\cQ_N)$ converge in the sense of Benjamini-Schramm to some probability measure $\mathbb{P}$ on the set of rooted quantum graphs (modulo isomorphism);
\item The empirical spectral measure $\mu_{\cQ_N} = \frac{1}{\sum_{e\in E_N} L_e} \sum_{\lambda_k^{(N)}\in \sigma(H_{\cQ_N})}\delta_{\lambda_k^{(N)}}$ converges vaguely: for any $\chi\in C_c(\R)$, we have
\begin{equation*}
\lim\limits_{N\rightarrow \infty} \int_\R \chi(\lambda)\, \mathrm{d}\mu_{\cQ_N}(\lambda)=\int_{\mathbf{Q}_*} \chi\big{(}H_\cQ\big{)}(\mathbf{x_0},\mathbf{x_0}) \dd \mathbb{P}(\cQ,\mathbf{x_0})\,
\end{equation*}
where $H_\cQ$ is the Schr\"odinger operator on the limiting random quantum graph $\cQ$, $\chi\big{(}H_\cQ\big{)}(\mathbf{x_0},\mathbf{x_0})$ is the value of the Schwartz kernel of the operator $\chi\big{(}H_\cQ\big{)}$ at the root $\mathbf{x_0}$ and $\mathbf{Q}_*$ is a set of equivalence classes of rooted
quantum graphs (the equivalence relation is specified in Section
\ref{sec:BS}).
\end{enumerate}
\end{main*}

In the framework of discrete graphs with uniformly bounded degrees, the convergence of ESM is an elementary consequence of Benjamini-Schramm convergence \cite{ATV}. In fact, as the measures have compact support, it suffices to establish the convergence of moments. But the $k$-moment of the ESM is just $\frac{1}{N}\sum_{x\in V_N}\langle \delta_x, \cA_{G_N}^k \delta_x\rangle$. The term $\langle \delta_x, \cA_{G_N}^k \delta_x\rangle$ is the number of $k$-paths in $G_N$ from $x$ to $x$, and only depends on a $k$-ball around $x$. The convergence of the $k$-th moment then follows from the convergence of the distribution of $k$-balls.

The proof for quantum graphs is more involved and takes the major part of this paper. We start by noting that since the Green's kernel $y\mapsto G_z(x,y)$ solves the eigenvalue equation, we may expand it in a basis of solutions of $H_{\cQ}f=zf$ on each edge. The coefficients of the expansion are eigenvectors of an \emph{evolution operator}, which is a matrix indexed by the oriented edges of the graphs. Thus, instead of working with the unbounded operator $H_{\cQ}$ (which poses obvious problems if one wants to take powers as in the discrete case), we work with finite matrices indexed by the edges. This allows us to prove the continuity of the Green's kernel in the Benjamini-Schramm topology. From this, we proceed to establish the convergence of the ESM. We do this in two steps, finally obtaining a statement resembling \eqref{e:kestenmckay}. In fact, while the last equality in \eqref{e:kestenmckay} is immediate, making sense of a punctual spectral measure $\mu_{\mathbf{x_0}}^{H_{\cQ}}$ for quantum graphs takes some effort. 

Evolution operators are well-known in the ``scattering approach'' to quantum graphs when there is no edge potential \cite{KoSm99,BeKu13}. The construction has been adapted to include edge potentials with Kirchhoff conditions in \cite{RuSmi12}, then general boundary conditions in \cite{BoEgRu}. We construct our operator a bit differently; for technical reasons, we want it to be sub-unitary in (some large region of) the upper-half plane $\Im z>0$. While this is already true in the framework of \cite{RuSmi12} when $\Re z$ is large enough, it seems some more work has to be done in general, and our construction works only when $\Im z$ is large enough. In particular, we do not obtain a characterization of the spectrum of $H_{\cQ}$ in terms of the evolution operator; for this see \cite{RuSmi12,BoEgRu} instead, when the graph is finite.

The paper is organized as follows. In Section~\ref{sec:QG}, we recall the definition of quantum graphs with edge potentials and general boundary conditions at the vertices. In Section~\ref{sec:BS}, we define the Benjamini-Schramm convergence of quantum graphs and state our main results (Theorems \ref{th:GreenContinuous2}, \ref{th:IntegralKernelsAreCool}, \ref{th:Empirical}). In Section~\ref{sec:Kirch}, we present the main ideas of the proof of Theorem \ref{th:GreenContinuous2} in the special case where there is no potential, and where each vertex has Kirchhoff boundary conditions. In Section~\ref{sec:Evolution}, we construct our evolution operator, and we use it in Section~\ref{sec:proof} to prove continuity of Green's kernel. The convergence of the ESM is then established in Section~\ref{sec:ConvEmp}. We conclude the paper with several examples of Benjamini-Schramm convergence in Section~\ref{sec:exam}. The reader who wants to develop more intuition can skip the proofs and go directly from Section~\ref{sec:BS} to \ref{sec:exam}.

\section{Quantum graphs}\label{sec:QG}
\subsection{Metric graphs and quantum graphs}

Let $G=(V,E)$ be a combinatorial graph with countable vertex set $V$ and edge set $E$. For each vertex $v\in V$, we denote by $d(v)$ the degree of $v$. For simplicity, we assume there is at most one edge between any two vertices, and that there is no edge from a vertex to itself.
If $v_1, v_2\in V$, we write $v_1\sim v_2$ if $\{v_1,v_2\}\in E$.
We let $\mathcal{B}= \mathcal{B}(G)$ be the set of oriented edges (or bonds), so that $|\mathcal{B}|=2|E|$ (an edge $\{v_1,v_2\}\in E$ gives rise to two oriented edges
$(v_1,v_2), (v_2,v_1)\in \mathcal{B}$). If $b\in \mathcal{B}$, we shall denote by $\hat{b}$ the reverse bond. We write $o(b)$ for the origin of $b$ and $t(b)$ for the terminus of $b$. We will also write $e(b)\in E$ for the edge obtained by forgetting the orientation of $b$.

\begin{definition}\label{def:MG}
A \emph{length graph} $(V,E,L)$ is a connected combinatorial graph $(V,E)$ endowed with a map $L: E\rightarrow (0,\infty)$. We denote $L_e:=L(e)$. If $b\in \mathcal{B}$, we also denote $L_b:= L(e(b))$.

The underlying \emph{metric graph} is the set\footnote{Note that we do not include the vertices in $\cG$.}
\[
\cG:= \{ (b, x); b\in \mathcal{B}, x\in (0,L_b)\} \slash \simeq \,,
\]
where the equivalence relation $\simeq$ is described by $(b,x)\simeq (b',x')$ 
if $b'= \hat{b}$ and $x'= L_b-x$.

Points of $\cG$ will be denoted by the shorter boldface $\mathbf{x}$ when there is no confusion.
\end{definition}
Compared to more usual definitions, this one avoids identifying each edge with a segment $(0, L)$, which involves a non-canonical choice of an origin.

The set $\cG$ can be equipped with a measure $\mathrm{d}\mathbf{x}$, as follows:
 if $f: \cG\To \C$, then we can lift it to a map $\tilde{f} : \{ (b, x); b\in \mathcal{B}, x\in (0,L_b)\} $ such that $\tilde{f}(b,x)= \tilde{f}(\hat{b}, L_b-x)$. If $\tilde{f}$ is measurable, we define
\[
\int_{\cG} f(\mathbf{x}) \,\mathrm{d}\mathbf{x}:= \frac{1}{2} \sum_{b\in \mathcal{B}} \int_0^{L_b} \tilde{f}(b,x)\, \mathrm{d}x \,.
\]
The $\frac{1}{2}$ factor comes from the fact that each non-oriented edge is counted twice, once for each orientation. In the sequel, we will denote by $\langle \cdot, \cdot \rangle$ the $L^2$-scalar product associated to this measure.

A function $f:\cG \To \C$ can be written as $f=(f_b)_{b\in \mathcal{B}}$, where $f_b(x) := f(b,x)$ satisfies the condition $f_{\widehat{b}}(L_b-x) = f_b(x)$. This means that the value of $f$ at the point at distance $x$ from $o(b)$ is equal to the value of $f$ at the point at distance $L_b-x$ from $o(\hat{b}) = t(b)$.

A quantum graph is a length graph endowed with a Schr\"odinger operator satisfying boundary conditions at each vertex. More precisely
\begin{definition}\label{def:QG}
A \emph{quantum graph} $\cQ=(V,E,L,W,\beta,U)$ is the data of:
\begin{itemize}
\item A length graph $(V,E,L)$,
\item A potential $W=(W_b)_{b\in \mathcal{B}} \in \bigoplus_{b\in \mathcal{B}} C^0([0,L_b]; \R)$ satisfying for $x\in [0,L_b]$,
\begin{equation}\label{eq:ReversePot}
W_{\widehat{b}}(L_b-x) = W_b(x)\,.
\end{equation}
\item For each $v\in V$, a labelling of the oriented edges starting at $v$, that is, a one-to-one map
\[
\beta^v : \{1,\ldots,d(v)\}\To \left\{b\in \mathcal{B} \,|\, o(b)=v\right\}.
\]
\item For each $v\in V$, a unitary matrix $U_v\in \mathcal{M}_{d(v)}(\C)$.
\end{itemize}
\end{definition}
Condition \eqref{eq:ReversePot} simply means that $W$ is well-defined on the quotient $\cG$. It does not impose that $W_b$ is symmetric (a symmetric potential would satisfy moreover $W_b(L_b-x)=W_b(x)$).

The unitary matrix $U_v$ will encode the boundary condition at $v$.

In the sequel, we will always make the following assumption
\begin{hyp}
\[
d(G):= \sup_{v\in V} d(v)<\infty
\]
\[
\underline{L}(\cQ):= \inf_{e\in E} L(e)>0
\]
\[
\overline{L}(\cQ):= \sup_{e\in E} L(e)<\infty
\]
\[
\|W\|_\infty:= \sup_{b\in \mathcal{B}} \|W_b\|_{C^0} < \infty
\]
\begin{equation}\label{e:robinpart}
\sup_{v\in V} \|\Lambda_v\| <\infty\,,
\end{equation}
where $\Lambda_v$ is defined as follows: let $P^{\mathrm{R}}_v$ be the orthogonal projection onto the eigenspace of all eigenvalues of $U_v$ other than $\pm 1$. Let $(U_v\pm \id)_{\mathrm{R}}$ be the restriction of $(U_v\pm \id)$ to $P^{\mathrm{R}}_v\C^{d(v)}$. Then $\Lambda_v : P^{\mathrm{R}}_v\C^{d(v)} \to P^{\mathrm{R}}_v\C^{d(v)}$ is the Cayley transform $\Lambda_v = -\ii (U_v+\id)_{\mathrm{R}}^{-1}(U_v-\id)_{\mathrm{R}}$.
\end{hyp}

The last condition may seem unclear, but it is actually quite natural, as we'll see in Examples~\ref{subsec:exam}.

Hypothesis 1 is automatically satisfied for finite quantum graphs:

\begin{definition}
A quantum graph $\cQ=(V,E,L,W,\beta,U)$ will be called \emph{finite} if $|V|, |E|<\infty$. Its \emph{total length} is then defined as 
\[
\mathcal{L}(\cQ)= \sum_{e\in E} L(e)\,.
\]
\end{definition}

\subsection{The Schr\"odinger operator on a quantum graph}
Let $\cQ=(V,E,L,W,\beta,U)$ be a quantum graph as in Definition \ref{def:QG}. Consider the Hilbert space\footnote{In the sequel, all the scalar products in Hilbert spaces will be linear in the right variable, and anti-linear in the left one: $\langle z u , z' v \rangle = \overline{z} z' \langle u, v \rangle$.}
\begin{equation*}
\mathscr{H}:=\Big{\{}f= (f_b)_{b\in \mathcal{B}} \in \mathop\bigoplus_{b\in \mathcal{B}} L^2(0,L_b) \,\Big{|}\, f_{\widehat{b}}(L_b-\cdot) = f_{b} (\cdot) \text{ and } \sum_{b\in \mathcal{B}} \|f_b\|^2_{L^2(0,L_b)}<\infty \Big{\}}
\end{equation*}
 and its subset 
\[
\mathscr{H}^0:=\Big{\{}f= (f_b)_{b\in \mathcal{B}} \in \mathop\bigoplus_{b\in \mathcal{B}} H^2(0,L_b) \,\Big{|}\, f_{\widehat{b}}(L_b-\cdot) = f_{b} (\cdot) \text{ and } \sum_{b\in \mathcal{B}} \|f_b\|^2_{H^2(0,L_b)}<\infty \Big{\}}\subset \mathscr{H}\,.
\]
Note that $\mathscr{H}$ can be identified with the space $L^2(\cG)$ of functions $f : \cG \To \C$ which are square-integrable for the measure $\mathrm{d}\mathbf{x}$.

We define an operator $H_\cQ$ acting on $\psi= (\psi_b)_{b\in \mathcal{B}}\in  \mathscr{H}^0$ by 
\begin{equation}\label{e:H}
(H_{\cQ}\psi_b)(x) = -\psi''_b(x) + W_b(x)\psi_b(x).
\end{equation}
Note that \eqref{eq:ReversePot} implies that $H_\cQ : \mathscr{H}^0\To \mathscr{H}$.

Thus defined, $H_{\cQ}$ is not essentially self-adjoint: we need to impose suitable boundary conditions at each vertex. Given $v\in V$, $f\in \mathscr{H}^0$ and $U_v$, we define 
\begin{equation}\label{eq:defF}
F(v)\in \C^{d(v)}:=\big{(} f(o_{\beta^v(j)})\big{)}_{j=1}^{d(v)}\,,\qquad F'(v)\in \C^{d(v)}:=\big{(} f'(o_{\beta^v(j)})\big{)}_{j=1}^{d(v)}\,,
\end{equation}

\begin{equation}\label{e:a1a2}
A_1(v) = \ii(U_v-\id) \,, \qquad A_2(v) = U_v+\id \,.
\end{equation}

We then define the space 
\begin{equation}\label{e:HQ}
\mathscr{H}^{\cQ}:=\Big{\{}f\in \mathscr{H}^0 \,\big{|}\, \forall v\in V:  A_1(v) F(v) + A_2(v) F'(v)=0\Big{\}}\subset \mathscr{H}.
\end{equation}
It follow from \cite[Theorem 1.4.4]{BeKu13} that if $\cQ$ is finite, then $H_\cQ : \mathscr{H}^{\cQ} \To \mathscr{H}$ is self-adjoint. Moreover, any local self-adjoint boundary condition is of this form for some set of unitary matrices $(U_v)_{v\in V}$. The same holds for infinite graphs satisfying Hypothesis 1, see \cite[Section 1.4.5]{BeKu13}.

\subsection{Examples}\label{subsec:exam}
\begin{enumerate}[\rm (a)]
\item The most common choice of boundary conditions are the  \emph{Kirchhoff boundary conditions}, which correspond to $U_v = \frac{2}{d(v)} \mathbbm{1} - \id$, where $\mathbbm{1}$ is the matrix with all entries equal to $1$. In this case,
\begin{equation}\label{e:a1a2kir}
A_1(v)F(v)+A_2(v)F'(v) = \begin{pmatrix} 2\ii\big(\frac{1}{d}\sum_{j=1}^df(o_{\beta(j)})-f(o_{\beta(1)})\big) + \frac{2}{d}\sum_{j=1}^df'(o_{\beta(j)})\\ \vdots\\2\ii\big(\frac{1}{d}\sum_{j=1}^df(o_{\beta(j)})-f(o_{\beta(d)})\big) + \frac{2}{d}\sum_{j=1}^df'(o_{\beta(j)})\end{pmatrix} \,.
\end{equation}
Hence, \eqref{e:HQ} implies $\frac{2}{d}\sum_{j=1}^df'(o_{\beta(j)})=2\ii(f(o_{\beta(k)})-\frac{1}{d}\sum_{j=1}^df(o_{\beta(j)}))$ for each $k$, which implies all $f(o_{\beta(k)})$ must be equal and thus $\sum_{j=1}^df'(o_{\beta(j)})=0$. In other words, we have:

\smallskip

\begin{itemize}
\item \textbf{Continuity}: For all $b,b'\in \mathcal{B}$, we have $f_b(0) = f_{b'}(0) =:f(v)$ if $o(b)=o(b')=v$.
\item \textbf{Current conservation}: For all $v\in V$, 
\begin{equation}\label{e:current}
\sum_{b:o(b)=v} f_b'(0)= 0\,.
\end{equation}
\end{itemize}
\item One may take $U_v = -\id$ and $U_v = \id$ to obtain Dirichlet and Neumann conditions, respectively, at each vertex $v$. These conditions are rarely interesting as they decouple the graph into a direct sum of intervals.
\item In general, let $P_v^{\mathrm{D}}$ and $P_v^{\mathrm{N}}$ be the orthogonal projections onto the eigenspaces of $U_v$ for the eigenvalues $-1$ and $+1$, respectively, and $P_v^{\mathrm{R}} = \id-P_v^{\mathrm{D}}-P_v^{\mathrm{N}}$. Then \cite{Kuch} the quadratic form of $H_{\cQ}$ is given by
\[
h_{\cQ}[f,f] = \frac{1}{2}\sum_{b\in \mathcal{B}} \|f_b'\|^2 - \sum_{v\in V} \langle \Lambda_v F, F\rangle\,,
\]
where $\Lambda_v:P_v^{\mathrm{R}}\C^{d(v)}\To P_v^{\mathrm{R}}\C^{d(v)}$ is the invertible self-adjoint operator defined in \eqref{e:robinpart}. The domain $D(h_{\cQ})$ is the set of $(f_b)\in\bigoplus H^1(0,L_b)$ such that $P_v^{\mathrm{D}} F = 0$.

For example, if we replace \eqref{e:current} by $\sum_{b:o(b)=v} f_b'(0)=\alpha_v f(v)$, (often called the \emph{$\delta$-condition}) then $\langle \Lambda_v F, F\rangle = -\alpha_v |f(v)|^2$. In general, \eqref{e:robinpart} asks that the \emph{Robin part} $\sup_v \| \Lambda_v\|  $ is bounded, in other words, $\sup_v |\alpha_v|<\infty$ for $\delta$-conditions.
\end{enumerate}

Note that Kirchhoff (and more generally $\delta$-) boundary conditions are homogeneous around each vertex, so the data $\beta$ is not needed to describe the Schr\"odinger operator in this case.

\section{Benjamini-Schramm convergence for quantum graphs}\label{sec:BS}
\subsection{Convergence of graphs}
If $G=(V,E)$ is a connected combinatorial graph, we define for every $v\in V$ and $r\in \N$ the ball $B_G(v,r):= \{w\in V: d_G(v,w)\leq r\}$. 
Here, $d_G$ is the distance on the discrete graph $G$.
We write $E_G(v,r)$ and $\cB_G(v,r)$ for the sets of edges and oriented edges, respectively, connecting two vertices in $B_G(v,r)$.

\begin{definition}
A \emph{rooted quantum graph} $\mathrm{Q}= (V,E,L,W,\beta,U,\mathbf{x_0})$ is a quantum graph $\cQ= (V,E,L,W,\beta,U)$ together with a marked point $\mathbf{x_0}\in \cG$, called the root. We often denote $\mathrm{Q}=(\cQ,\mathbf{x_0})$.
\end{definition}

Given a rooted quantum graph $\mathrm{Q}=(\cQ,\mathbf{x_0})$, we may build from it a new quantum graph, by adding  at $\mathbf{x_0}$
 a new vertex $v_{\mathbf{x_0}}$ with Kirchhoff's boundary conditions. More precisely, we introduce the following definition.

\begin{definition}\label{def:adding}
\emph{(Adding a vertex.)}
Let $\mathrm{Q}=(\cQ,\mathbf{x_0})$ be a rooted quantum graph. We denote by $\cQ^{\mathbf{x_0}}$ the quantum graph such that
\[
\cQ^{\mathbf{x_0}}:=(V^{\mathbf{x_0}},E^{\mathbf{x_0}},L^{\mathbf{x_0}},W^{\mathbf{x_0}},\beta^{\mathbf{x_0}},U^{\mathbf{x_0}}) \,,
\]
where
\begin{itemize}
\item $V^{\mathbf{x_0}}:= V \sqcup \{v_{\mathbf{x_0}}\}$.
\item $E^{\mathbf{x_0}}:= E\setminus \{e(b_0)\} \cup \{\{o(b_0), v_{\mathbf{x_0}}\}, \{t(b_0), v_{\mathbf{x_0}}\}\}$.
\item $L^{\mathbf{x_0}}(e)= L(e)$ for $e\in  E\setminus \{e(b_0)\}$, $L^{\mathbf{x_0}}(\{o(b_0), v_{\mathbf{x_0}}\}) = x_0$ and $L^{\mathbf{x_0}}(\{t(b_0), v_{\mathbf{x_0}}\})= L_{b_0} - x_0$.
\item $W^{\mathbf{x_0}}_b = W_b$ $\ \forall b\in \mathcal{B}\setminus \{b_0, \widehat{b_0}\}$, 
while $W^{\mathbf{x_0}}_{(o(b_0), v_{\mathbf{x_0}})}= (W_{b_0})|_{[0,x_0]} $, 
$W^{\mathbf{x_0}}_{(v_{\mathbf{x_0}},o(b_0))}=(W_{\widehat{b_0}})|_{[x_0, L_{b_0}]}$,
 $W^{\mathbf{x_0}}_{(v_{\mathbf{x_0}}, t(b_0))}= (W_{b_0})|_{[x_0,L_{b_0}]} $,
 and $W^{\mathbf{x_0}}_{(t(b_0),v_{\mathbf{x_0}})}=(W_{\widehat{b_0}})|_{[0,x_0]}$.
\item $(\beta^{\mathbf{x_0}})^v= \beta^v$ for all $v\in V $, $(\beta^{\mathbf{x_0}})^{v_{\mathbf{x_0}}}(1)=(v_{\mathbf{x_0}},o(b_0))$, $(\beta^{\mathbf{x_0}})^{v_{\mathbf{x_0}}}(2)=(v_{\mathbf{x_0}},t(b_0))$.
\item $U^{\mathbf{x_0}}_v= U_v$ for all $v\in V $ and $U^\mathbf{x_0}_{v_{\mathbf{x_0}}}= \begin{pmatrix}
0 & 1 \\
1 & 0
\end{pmatrix}$. 
\end{itemize}
We denote by $G^{\mathbf{x_0}}=(V^{\mathbf{x_0}},E^{\mathbf{x_0}} ) $ the new combinatorial graph and $ \cB^{\mathbf{x_0}}$ the corresponding set of bonds.
\end{definition}

If $(G,v)$ and $(G',v')$ are two rooted discrete graphs, we denote
\[
\phi : (G,v) \xrightarrow{\sim} (G',v')
\]
if $\phi$ is a graph isomorphism satisfying $\phi(v)=v'$. 

Let $\mathrm{Q}=(V,E,L,W, \beta, U,\mathbf{x_0})$ and $\mathrm{Q}'=(V',E',L',W', \beta', U',\mathbf{x_0'})$ be two rooted quantum graphs. We define a pseudometric\footnote{Recall that a pseudometric satisfies all the properties of a distance, except that we don't have $d(\mathrm{Q},\mathrm{Q}')=0 \Longrightarrow \mathrm{Q}=\mathrm{Q}'$.} between them as follows:
\begin{equation}\label{eq:defdistance}
d\left( (\cQ,\mathbf{x_0}),(\cQ',\mathbf{x_0'})\right) = \frac{1}{1+\alpha},
\end{equation}
where
\[
\begin{aligned}
\alpha:= \sup\Big\{ r>0\ \big| \ &\exists \phi : B_{G^{\mathbf{x_0}}}(v_{\mathbf{x_0}}, \lfloor r\rfloor) \xrightarrow{\sim} B_{G'^{\mathbf{x'_0}}}(v_{\mathbf{x'_0}}, \lfloor r\rfloor) \text{ with } \phi\circ \beta^{\mathbf{x_0}} = \beta'^{\mathbf{x'_0}}  \\
& \text{and } \delta_{\lfloor r\rfloor ,\phi}((L^{\mathbf{x_0}},W^{\mathbf{x_0}},U^{\mathbf{x_0}}),(L'^{\mathbf{x'_0}},W'^{\mathbf{x'_0}},U'^{\mathbf{x'_0}}))<1/r\Big\} 
\end{aligned}
\]
and $\delta_{k,\phi}(\cdot,\cdot)$ is the distance between the data in a $k$-ball around the root:
\begin{multline*}
\delta_{k,\phi}((L^{\mathbf{x_0}},W^{\mathbf{x_0}},U^{\mathbf{x_0}}),(L'^{\mathbf{x'_0}},W'^{\mathbf{x'_0}},U'^{\mathbf{x'_0}})) = \max\bigg\{\max_{e\in E_{G^{\mathbf{x_0}}}(v_{\mathbf{x_0}}, k)}\big|L^{\mathbf{x_0}}_e - L'^{\mathbf{x_0'}}_{\phi(e)}\big|,\\
\max_{b \in \cB_{G^{\mathbf{x_0}}}(v_{\mathbf{x_0}}, k)}\sup_{t\in [0,1]}\left|W^{\mathbf{x_0}}_b \left(tL^{\mathbf{x_0}}_b\right)-W'^{\mathbf{x'_0}}_{\phi(b)}\left( tL'^{\mathbf{x'_0}}_{\phi(b)}\right)\right|,\max_{v\in B_{G^{\mathbf{x_0}}}(v_{\mathbf{x_0}},k)}\big\| U^{\mathbf{x_0}}_v-U'^{\mathbf{x'_0}}_{\phi(v)}\big\|\bigg\}.
\end{multline*}

\begin{definition}\label{def:equiv}
Let $\mathrm{Q}=(V,E,L,W, \beta, U,\mathbf{x_0})$ and $\mathrm{Q}'=(V',E',L',W', \beta', U',\mathbf{x_0'})$ be rooted quantum graphs. We say that $\mathrm{Q}$ and $\mathrm{Q}'$ are \emph{equivalent} if $\exists \phi: (G^{\mathbf{x_0}},v_{\mathbf{x_0}}) \xrightarrow{\sim} (G'^{\mathbf{x_0}}, v_{\mathbf{x'_0}})$ satisfying $L^{\mathbf{x_0}}= L'^{\mathbf{x_0}} \circ \phi$, $W^{\mathbf{x_0}} = W'^{\mathbf{x_0'}}\circ \phi$, $U^{\mathbf{x_0}}=U'^{\mathbf{x_0'}}\circ \phi$ and $\phi \circ \beta^{\mathbf{x_0}} = \beta'^{\mathbf{x_0'}}$.

We denote the equivalence class of $\mathrm{Q}$ by $[\mathrm{Q}]$.
\end{definition}

Let $\mathbf{Q}_*$ be the set of equivalence classes of rooted quantum graphs. 
\begin{lemme}
$(\mathbf{Q}_*,d)$ is a Polish space, i.e., a separable complete metric space.
\end{lemme}
\begin{proof}
Clearly if $[\mathrm{Q}_1]=[\mathrm{Q}'_1]$ and $[\mathrm{Q}_2]=[\mathrm{Q}'_2]$, then $d(\mathrm{Q}_1, \mathrm{Q}_2)= d(\mathrm{Q}'_1, \mathrm{Q}'_2)$, so $d$ is well-defined on $\mathbf{Q}_*$.

If $\mathrm{Q}$, $\mathrm{Q}'$ are rooted quantum graphs with $d(\mathrm{Q}, \mathrm{Q}')=0$, then for any $r\in \N$, $\exists \phi_r: B_{G^{\mathbf{x_0}}}(v_{\mathbf{x_0}}, r) \xrightarrow{\sim} B_{G'^{\mathbf{x_0'}}}(v_{\mathbf{x_0'}}, r)$ such that $\phi_r\circ \beta^{\mathbf{x_0}}= \beta'^{\mathbf{x'_0}}$ and $\delta_{\lfloor r\rfloor ,\phi}((L^{\mathbf{x_0}},W^{\mathbf{x_0}},U^{\mathbf{x_0}}),(L'^{\mathbf{x_0'}},W'^{\mathbf{x_0'}},U'^{\mathbf{x_0'}}))<1/r$.

Let $\phi_r^{(n)} = \phi_r|_{B_{G^{\mathbf{x_0}}}(v_{\mathbf{x_0}},n)}$ for $r\ge n$. One checks that $\phi_r^{(n)}:B_{G^{\mathbf{x_0}}}(v_{\mathbf{x_0}},n) \xrightarrow{\sim} B_{G'^{\mathbf{x_0'}}}(v_{\mathbf{x'_0}},n)$ for all $r \ge n$. Since $G$ and $G'$ are locally finite, $\phi_r^{(n)}$ has a convergent (in fact stationary) subsequence $\phi_{r_j}^{(n)}$. Denote its limit by $\phi^{(n)}$. Then $\phi^{(n)} :B_{G^{\mathbf{x_0}}}(v_{\mathbf{x_0}},n) \xrightarrow{\sim}B_{G'^{\mathbf{x_0}}}(v_{\mathbf{x'_0}},n)$.

The advantage of $\phi^{(n)}$ over $\phi_n$ is that $\phi^{(n+1)}|_{B_{G^{\mathbf{x_0}}}(v_{\mathbf{x_0}},n)} = \lim \phi_{r_j}^{(n+1)}|_{B_{G^{\mathbf{x_0}}}(v_{\mathbf{x_0}},n)} = \lim \phi_{r_j}^{(n)} = \phi^{(n)}$. So $\phi^{(m)}|_{B_{G^{\mathbf{x_0}}}(v_{\mathbf{x_0}},n)} = \phi^{(n)}$ for all $m \ge n$. Hence, if for $v\in G^{\mathbf{x_0}}$, say $v\in B_{G^{\mathbf{x_0}}}(v_{\mathbf{x_0}},n)$ for some $n$, we put $\phi(v) := \phi^{(n)}(v)$, then $\phi$ is a well-defined graph isomorphism $\phi:(G^{\mathbf{x_0}},v_{\mathbf{x_0}})\xrightarrow{\sim}(G'^{\mathbf{x'_0}},v_{\mathbf{x'_0}})$. Moreover, for any $n$ and $r \ge n$, we have $\|L'^{\mathbf{x'_0}} \circ \phi_r^{(n)}-L^{\mathbf{x_0}}\|_{E_{G{\mathbf{x_0}}}(v_{\mathbf{x_0}},n)}<1/r$, so $\|L'^{\mathbf{x_0}}\circ \phi^{(n)}-L^{\mathbf{x_0}}\|_{E_{G^{\mathbf{x_0}}}(v_{\mathbf{x_0}},n)} = 0$. This shows that $L'^{\mathbf{x_0'}}\circ \phi = L^{\mathbf{x_0}}$. Arguing similarly for $W^{\mathbf{x_0}},U^{\mathbf{x_0}}$, we get $[Q]=[Q']$ as required.

We showed that $d$ is a distance on $\mathbf{Q}_*$. Separability follows by considering the countable set of finite rooted graphs and using that $\R$, $C^0([0,L])$ and $U_d(\C)$, the set of unitary matrices, are separable for every $L$ and $d$. To prove completeness, note that if $[\cQ_n,\mathbf{x}_n]$ is Cauchy, then the (equivalence classes of) balls $[B_{G_n^{\mathbf{x}_n}}(v_{\mathbf{x}_n},\lfloor r\rfloor)]$ are stationary in $n$ for any $r$. Say it stabilizes for $n\ge n_r$. We may then define a limit discrete rooted graph $(G_{\star}, v_\star)$ iteratively such that for any $r$, $[B_{G_\star}(v_\star,\lfloor r\rfloor)] = [B_{G^{\mathbf{x}_{n_r}}_{n_r}}(v_{\mathbf{x}_{n_r}},\lfloor r\rfloor)]$. Using the completeness of $\R$, $C^0([0,L])$ and $U_d(\C)$, we may find a limiting data $(L,W,U)$ on $G_{\star}$, completing the proof.
\end{proof}

Let $\cP(\mathbf{Q}_*)$ be the set of probability measures on $\mathbf{Q}_*$.

\begin{definition}
Any finite quantum graph $\cQ=(V,E,L,W, \beta, U)$ defines a probability measure $\nu_{\cQ}$ obtained by choosing a root uniformly at random: 
\[
\nu_{\cQ}:= \frac{1}{ 2 \mathcal{L}(\cQ)} \sum_{b_0\in \mathcal{B}} \int_0^{L_{b_0}} \delta_{[\cQ,(b_0,x_0)]}\mathrm{d}x_0 = \frac{1}{ \mathcal{L}(\cQ)}  \int_{\cG} \delta_{[\cQ,\mathbf{x_0}]}\mathrm{d}\mathbf{x_0}.
\]

If $\cQ_N$ is a sequence of quantum graphs, we say that $\mathbb{P}\in \mathcal{P}(\mathbf{Q}_*)$ is the \emph{local weak limit} of $\cQ_N$, or that $\cQ_N$ \emph{converges in the sense of Benjamini-Schramm to $\mathbb{P}$}, if $\nu_{\cQ_N}$ converges weakly-* to $\mathbb{P}$.
\end{definition}

Let $D\in \N$, $M>m>0$. We define $\mathbf{Q}_*^{D,m,M}\subset \mathbf{Q}_*$ as the subset of equivalence classes $[\cQ,\mathbf{x_0}]=[V,E,L,W, \beta, U,\mathbf{x_0}]$ which satisfy
\begin{equation}\label{eq:ConditionCompQG}
\begin{cases}
\qquad\qquad\qquad\qquad\quad d(G)\leq D\\
\qquad\qquad\qquad\quad m\le \underline{L}(\cQ)\leq \overline{L}(\cQ) \le M\\
\qquad\qquad\qquad\quad \sup_{v\in V} \|\Lambda_v\|  \le M \\
 W_b\in \Lip([0,L_b]) \text{ and }  \max\left(\| W_b\|_{\infty},\Lip(W_b)\right) \le M \quad \forall b\in \mathcal{B}\,.
\end{cases}
\end{equation}
Here $\Lip(I)$ is the set of Lipschitz-continuous functions on $I$ and $\Lip(f)$ is the Lipschitz constant of $f$. If $\delta>0$, we also define 
\[
\mathbf{Q}_*^{D,m,M,\delta}:= \left\{ [\cQ, (b_0, x_0)]\in \mathbf{Q}_*^{D,m,M} \ | \  x_0\in [\delta, L_{b_0}-\delta]\right\}.
\]

\begin{lemme}\label{lem:QCompact}
The subset $\mathbf{Q}_*^{D,m,M,\delta}$ is compact.
\end{lemme}
\begin{proof}
Let $[V_N,E_N,L_N,W_N, \beta_N, U_N,\mathbf{x}_N]$ be a sequence in $\mathbf{Q}_*^{D,m,M,\delta}$.  Let $r\in \N$. Up to extracting a subsequence, we may suppose that the $B_{G_N^{\mathbf{x}_N}}(v_{\mathbf{x}_N},r)$ are all isomorphic to each other, since there are only finitely many isomorphy classes of balls of radius $r$ with degree $\leq D$.
We may extract a subsequence such that for every $v\in B_{G_N^{\mathbf{x}_N}}(v_{\mathbf{x}_N},r)$, for every $e\in E^{\mathbf{x}_N}_{G_N^{\mathbf{x}_N}}(v_{\mathbf{x}_N},r)$ and for every $b\in \cB_{G_N^{\mathbf{x}_N}}(v_{\mathbf{x}_N},r)$, the sequences $(L^{\mathbf{x}_N}_N(e))$, $((W^{\mathbf{x}_N}_N)_b)$ and $((U^{\mathbf{x}_N}_N)_v)$ are convergent. Here, we used Arzela-Ascoli's theorem for the $((W^{\mathbf{x}_N}_N)_b)$ along with the compactness of unitary matrices. For large $N$, we get $x_N \in [\delta, L_b]$, which has a subsequence converging to $x_{\star}$. Then $x_N-L_{b_N} \to x_{\star} - L_b$. But $x_N-L_{b_N}\le -\delta$ for all $N$, so $x_{\star} \in [\delta,L_b-\delta]$ as required.

We may then conclude by a diagonal extraction argument for the different $r\in \N$. In fact, the diagonal sequence is clearly Cauchy and thus converges. The conditions \eqref{eq:ConditionCompQG} pass to the limit, so the limit quantum graph lies in $\mathbf{Q}_*^{D,m,M,\delta}$. 
\end{proof}

Note that $\nu_{\cQ_N}(\mathbf{Q}_{\ast}^{D,m,M}\setminus \mathbf{Q}_{\ast}^{D,m,M,\delta}) \le \frac{2\delta}{m}$, so Lemma~\ref{lem:QCompact} implies that the family $(\nu_{\cQ_N})$ on $\mathbf{Q}_{\ast}^{D,m,M}$ is \emph{tight}. Using Prokhorov's theorem, we obtain:

\begin{corollaire}\label{cor:BSComp}
Let $D\in\N$, $M>m>0$, and let $\cQ_N=(V_N,E_N,L_N,W_N, \beta_N, U_N)$ be a sequence of quantum graphs satisfying \eqref{eq:ConditionCompQG} for all $N\in \N$.
Then there is a subsequence $\cQ_{N_k}$ which converges in the sense of Benjamini-Schramm (i.e. there exists $\mathbb{P}\in \mathcal{P}(\mathbf{Q}_*^{D,m,M})$ such that $\nu_{{Q}_{N_k}}\xrightarrow{w^*} \mathbb{P}$).
\end{corollaire}

In other words, for any bounded continuous $F :  \mathbf{Q}_*^{D,m,M}\To \C$, we have
\[
 \frac{1}{ \mathcal{L}(\cQ_{N_k})}  \int_{\cG_{N_k}}F\big{(}[\cQ_{N_k},\mathbf{x_0}]\big{)}\mathrm{d}\mathbf{x_0} \To \int_{\mathbf{Q}_*^{D,m,M}} F([\cQ,\mathbf{x_0}])\,\dd\mathbb{P}([\cQ,\mathbf{x_0}]) =: \mathbb{E}_\mathbb{P}\big{[} F\big{]}\,.
\]

\smallskip

In the classical theory of local-weak convergence, a probability measure $\rho$ on the set of discrete rooted graphs is said to be \emph{sofic} if it is a Benjamini-Schramm limit of finite graphs. It is known that any sofic measure is \emph{unimodular}. This means that it has a weak form of homogeneity known as the ``Mass Transport Principle'', which is important for applications \cite{AL,BS}. Knowing if conversely any unimodular measure is sofic, is an open problem, see \cite{Elek,BLS,AHNR} for recent progress.

In our context of quantum graphs, ``sofic'' measures also have a nice property:

\begin{lemme}\label{lem:nalini}
Suppose $\cQ_N$ is a sequence of finite quantum graphs which converges to $\mathbb{P}$ in the sense of Benjamini-Schramm. Then $\mathbb{P}$ satisfies, for any bounded continuous $F:\mathbf{Q}_{\ast}\To\C$:
\[
\mathbb{E}_{\mathbb{P}}\left(F[\cQ,\mathbf{x_0}]\right) = \mathbb{E}_{\mathbb{P}}\bigg(\frac{1}{|\mathbf{B}_{\mathbf{o}}^1|} \int_{\mathbf{B}_{\mathbf{o}}^1} F[\cQ,\mathbf{x}]\,\dd \mathbf{x}\bigg) ,
\]
where $\mathbf{B}_{\mathbf{o}}^1$ is the metric graph induced from the combinatorial unit ball around a random root $o(b_0)$, i.e. $\mathbf{B}_{\mathbf{o}}^1 = \{(b,y)\in \cQ : o(b)=o(b_0)\}$.
\end{lemme}

More explicitly, the lemma says that
\[
\int_{\mathbf{Q}_{\ast}} F[\cQ,\mathbf{x_0}]\,\dd\mathbb{P}[\cQ,\mathbf{x_0}] = \int_{\mathbf{Q}_{\ast}} \bigg(\frac{1}{|\mathbf{B}_{\mathbf{o}}^1|} \sum_{b:o(b)=o(b_0)}\int_0^{L_b} F[\cQ,(b,x)]\,\dd x\bigg) \dd \mathbb{P}[\cQ,(b_0,x_0)] \,.
\]
Note that the integrand on the RHS only depends on $o(b_0)$: for fixed $\cQ$ and $b_0\in \cQ$, the integrand is fixed as $x_0$ varies over $(0,L_{b_0})$. We give an application of this result in Lemma~\ref{lem:spectralmeasures}.

\begin{proof}[Proof of Lemma~\ref{lem:nalini}]
Denoting $\mathbf{B}_v^1$ the metric graph induced from a combinatorial unit ball around $v\in V_N$, we have
\begin{align*}
\mathbb{E}_{\nu_{\cQ_N}}(F[\cQ,\mathbf{x}_0]) &= \frac{1}{2\cL(\cQ_N)} \sum_{v\in V_N} \sum_{b,o(b)=v} \int_0^{L_b} F[\cQ_N,(b,x_0)]\,\dd x_0 \\
&= \frac{1}{2\cL(\cQ_N)} \sum_{v\in V_N} \int_{\mathbf{B}_v^1} F[\cQ_N,\mathbf{x}]\,\dd\mathbf{x} \\
&= \frac{1}{2\cL(\cQ_N)} \sum_{v\in V_N} \sum_{b_0,o(b_0)=v} \int_0^{L_{b_0}} \frac{1}{|\mathbf{B}_v^1|}\left( \int_{\mathbf{B}_v^1} F[\cQ_N,\mathbf{x}]\,\dd\mathbf{x}\right)\dd x_0\\
&= \mathbb{E}_{\nu_{\cQ_N}}\left(\frac{1}{|\mathbf{B}_{\mathbf{o}}^1|}\int_{\mathbf{B}_{\mathbf{o}}^1} F[\cQ,\mathbf{x}]\,\dd\mathbf{x}\right).
\end{align*}
The claim follows by taking $N\To\infty$.
\end{proof}

\subsection{Spectral theory}

Before we address the convergence of ESM discussed in the introduction, let us present some results of independent interest.

Let $\cQ=(V,E,L,W, \beta, U)$ be a quantum graph and $\mathbf{x_0},\mathbf{y_0}\in \cG$. For $z\in \C\setminus \R$, we denote by $G_z(\mathbf{x_0},\mathbf{y_0})$ the Schwartz kernel
of $(H_\cQ-z)^{-1}$ (called the Green's function). Its basic properties are given in Appendix \ref{sec:app1}.

\begin{theoreme}\label{th:GreenContinuous2}
Consider the map $\mathbf{G}_z:\mathbf{Q}_* \ni [(V,E,L,W, \beta, U,\mathbf{x_0}) ] \mapsto G_z(\mathbf{x_0},\mathbf{x_0})\in \C$. Let $M>m>0$, $D\in \N$. For all $z\in \C\setminus \R$, $\mathbf{G}_z$ is continuous on $\mathbf{Q}_*^{D,m,M}$.
\end{theoreme}

In other words, if $[\cQ_N,\mathbf{x}_N]\To [\cQ,\mathbf{x}]$ in the metric \eqref{eq:defdistance}, then $G_z^{\cQ_N}(\mathbf{x}_N,\mathbf{x}_N) \To G_z^{\cQ}(\mathbf{x},\mathbf{x})$.

\begin{corollaire}\label{cor:bsconv}
Let $\cQ_N$ be a sequence of quantum graphs satisfying \eqref{eq:ConditionCompQG} and converging in the sense of Benjamini-Schramm to $\mathbb{P}$. Then for all $z\in \C\setminus \R$, we have 
\[
\lim\limits_{N\To \infty} \frac{1}{\cL(\cQ_N)} \int_{\cG_N}G_z(\mathbf{x_0},\mathbf{x_0})\,\dd\mathbf{x_0} =  \mathbb{E}_\mathbb{P}\big{[}\mathbf{G}_z \big{]}\,.
\]
\end{corollaire}
\begin{proof}
The map $\mathbf{G}_z$ is continuous on $\mathbf{Q}_{\ast}^{D,m,M}$ by Theorem~\ref{th:GreenContinuous2}. It is also bounded by Lemma~\ref{lem:BoundGreenImagin}. By hypothesis, $\nu_{\cQ_N}$ converges weakly to $\mathbb{P}$. Hence, $\int \mathbf{G}_z\,\dd \nu_{\cQ_N}\To \int \mathbf{G}_z\,\dd \mathbb{P}$.
\end{proof}

Other interesting continuous functionals are given by \emph{integral kernels of functions of the Laplacian}, as explained in the following theorem, which we prove in Section \ref{subsec:HelfSj}.

\begin{theoreme}\label{th:IntegralKernelsAreCool}
Let $\chi\in C_c(\R)$. Then $\chi(H_\cQ)$ has an integral kernel $\chi(H_\cQ)(\mathbf{x_0}, \mathbf{x_1})$ which is continuous: for any $b_0, b_1\in \mathcal{B}$, the map $(0, L_{b_0}) \times (0, L_{b_1})\ni (x_0,x_1)\mapsto \chi\big{(}H_\cQ\big{)}\big{(}(b_0,x_0),(b_1,x_1)\big{)}$ is continuous.

Furthermore, the map $\mathbf{Q}_*^{D,m,M}\ni [\cQ, \mathbf{x_0}] \mapsto \chi(H_\cQ) (\mathbf{x_0}, \mathbf{x_0})$ is continuous.
\end{theoreme}

Let $\cQ$ be a finite quantum graph. Then the operator $H_\cQ$ has a compact resolvent \cite[Theorem 3.1.1]{BeKu13}, so its spectrum $\Sigma(\cQ)$ is a discrete set of eigenvalues.

We now define the empirical spectral measure of $\cQ$ by
\begin{equation}\label{eq: defEmpirical}
\mu_\cQ := \frac{1}{\mathcal{L}(\cQ)} \sum_{\lambda_k\in \Sigma(\cQ)} \delta_{\lambda_k},
\end{equation}
where each $\lambda_k$ in the sum is counted with its multiplicity as an eigenvalue of $H_\cQ$. This defines a locally finite measure on $\R$.

The following theorem is our main result:
\begin{theoreme}\label{th:Empirical}
Let $\cQ_N$ be a sequence of quantum graphs obeying \eqref{eq:ConditionCompQG}, converging in the sense of Benjamini-Schramm to a measure $\mathbb{P}\in \mathcal{P}(\mathbf{Q}_*)$. Then for any $\chi \in C_c(\R)$, we have
\begin{equation}\label{eq:ConvSpecMeas}
\begin{aligned}
\lim\limits_{N\rightarrow \infty} \int_\R \chi(\lambda)\, \mathrm{d}\mu_{\cQ_N}(\lambda) &= \lim\limits_{\varepsilon \downarrow 0} \frac{1}{\pi}\int_\R \chi(\lambda) \mathbb{E}_{\mathbb{P}} \left[ \Im \mathbf{G}_{\lambda+\ii\varepsilon}  \right] \mathrm{d}\lambda\\
&=\int_{\mathbf{Q}_*} \chi\big{(}H_\cQ\big{)}(\mathbf{x_0},\mathbf{x_0}) \dd \mathbb{P}[\cQ,\mathbf{x_0}]\,.
\end{aligned}
\end{equation}
\end{theoreme}

We prove this theorem in Section~\ref{sec:ConvEmp}. The first formula tells us in principle how the spectrum looks like for large $N$. For instance, when the \emph{mean} Green's function is regular (which is often a reasonable assumption), the graph of the density $\lambda\mapsto \mathbb{E}_{\mathbb{P}} [ \Im \mathbf{G}_{\lambda+\ii 0}]$ describes the distribution of the spectrum, with peaks indicating zones of high concentration of eigenvalues. The second formula is more geometric. Together with \eqref{eq:WhatIsMu} below, it says that the empirical measures $\mu_{\cQ_N}$ also approach a measure which is obtained by taking the spatial average of all spectral measures of the limiting operator(s) at random points $\mathbf{x_0}$. Let us explain this further.

Let $[\cQ,\mathbf{x_0}]\in \mathbf{Q}_*^{D,m,M}$. In Appendix~\ref{sec:app1}, we show that $z\mapsto G_z(\mathbf{x_0},\mathbf{x_0})$ is a Herglotz function, i.e., a holomorphic function in $\C\setminus \R$ such that $\Im G_z(\mathbf{x_0},\mathbf{x_0})>0$ if $\Im z>0$. 
By \cite[Theorem 5.9.1]{Simon15}, there exist $c_{\mathbf{x_0}}\geq 0$, $d_{\mathbf{x_0}}\in \R$, and a positive measure $\mu_{\mathbf{x_0}}$ on $\R$ satisfying $\int_\R \frac{\dd \mu_{\mathbf{x_0}}(t)}{1+t^2} <\infty$ such that
\[
G_z(\mathbf{x_0},\mathbf{x_0}) = c_{\mathbf{x_0}}z+ d_{\mathbf{x_0}}+ \int_\R \Big{(} \frac{1}{t-z}- \frac{t}{1+t^2}\Big{)}  \mathrm{d}\mu_{\mathbf{x_0}}(t)\,.
\]
In case of discrete graphs, Green's function takes the form $G_z(v,v) = \int \frac{1}{t-z}\,\dd\mu_v(t)$, where $\mu_v$ is the spectral measure at $v$: $\mu_v(I) = \langle \delta_v,\mathbf{1}_I(H) \delta_v\rangle$. For $\mu_{\mathbf{x_0}}$, we have the following.

\begin{lemme}\label{lem:spectralmeasures}
The following properties hold:
\begin{enumerate}[\rm (i)]
\item For any $[\cQ,\mathbf{x_0}]\in\mathbf{Q}_{\ast}^{D,m,M}$ and $\chi\in C_c(\R)$, we have
\begin{equation}\label{eq:WhatIsMu}
\chi(H_{\cQ})(\mathbf{x_0},\mathbf{x_0}) = \int_{\R} \chi(\lambda)\,\dd\mu_{\mathbf{x_0}}(\lambda)\,.
\end{equation}
\item Let $I$ be a bounded open interval. Then $H_{\cQ}$ has spectrum in $I$ iff there exists $e\in E(\cQ)$ such that $\int_0^{L_e} \mu_{\mathbf{x}}(I)\,\dd x>0$.
\item Let $\cQ_N$ is a sequence of graphs obeying \eqref{eq:ConditionCompQG}, converging to $\mathbb{P}$ in the sense of Benjamini-Schramm. Suppose $\mathbb{P}(H_{\cQ} \text{ has spectrum in } I)>0$. Let $\mathfrak{N}_N(I) = \#\{\lambda_j(H_{\cQ_N})\in I\}$. Then
\begin{equation}\label{e:limvals}
\liminf_{N\to \infty} \frac{\mathfrak{N}_N(I)}{|V_N|} = C_I>0 \,.
\end{equation}
\end{enumerate}
\end{lemme}
Points (i) and (ii) show that $\mu_{\mathbf{x_0}}$ can be interpreted as a spectral measure at $\mathbf{x_0}$: functions in $C_c(\R)$ obey the expected ``functional calculus'', and the measures are strongly related to the spectrum of the operator. Point (iii) is a very basic illustration of how we can use information at the limit to learn something about a sequence of converging quantum graphs.
\begin{proof}
(i) By  \cite[Theorem 5.9.1 (g)]{Simon15}, $\int_{\R} \chi(\lambda)\,\dd\mu_{\mathbf{x_0}}(\lambda) = \lim\limits_{\varepsilon \downarrow 0} \frac{1}{\pi}\int_\R \chi(\lambda) \Im G_{\lambda+\ii\varepsilon}(\mathbf{x_0}, \mathbf{x_0}) \mathrm{d}\lambda$. So \eqref{eq:WhatIsMu} follows from \eqref{eq:ExprIntKern} for $\chi \in C_c^{\infty}(\R)$, and by approximation for $\chi\in C_c(\R)$ (see \S\ref{sec:end}).

(ii) Suppose $H_{\cQ}$ has some spectrum in an open interval $I$. Then we may choose $0\neq \chi\in C_c^{\infty}(\R)$, $0\le\chi \le \mathbf{1}_I$, and $f\in L^2(\cQ)$ such that $\langle f,\chi(H_{\cQ})f\rangle >0$. We may assume $f$ is supported in an edge $b_0\in \cQ$.\footnote{For $f=(f_e)$, we have $\langle f, \chi(H)f\rangle =\sum_{e,e'\in E}\langle f_e,\chi(H)f_{e'}\rangle$ and $|\langle f_e,\chi(H) f_{e'}\rangle|^2 = |\langle \chi(H)^{1/2}f_e,\chi(H)^{1/2}f_{e'}\rangle|^2\le \langle f_e,\chi(H)f_e\rangle\langle f_{e'},\chi(H) f_{e'}\rangle$, so $\langle f_e,\chi(H) f_e\rangle=0$ for all $e$ would imply $\langle f,\chi(H) f\rangle = 0$.} We may also assume that $\|f\|=1$. Now
\begin{align}
\langle f,\chi(H_{\cQ})f\rangle &= \int_0^{L_{b_0}}\int_0^{L_{b_0}} \overline{f(x_0)}\chi(H_{\cQ})(x_0,x_1)f(x_1)\,\dd x_0\dd x_1 \nonumber \\
&\le \Big(\int_0^{L_{b_0}}\int_0^{L_{b_0}} |f(x_0)f(x_1)|^2\,\dd x_0 \dd x_1\Big)^{1/2} \Big(\int_0^{L_{b_0}}\int_0^{L_{b_0}} \chi(H_{\cQ})(x_0,x_1)^2\,\dd x_0\dd x_1 \Big)^{1/2} \nonumber \\
\label{e:newcool}
&\le \int_0^{L_{b_0}}\chi(H_{\cQ})(x_0,x_0)\,\dd x_0 \,,
\end{align}
where we used that $\|f\|=1$ along with \eqref{eq:Cauchy-Schwarz}. 

Recalling \eqref{eq:WhatIsMu}, we get $0<\int_0^{L_{b_0}}\int_{\R}\chi(\lambda)\,\dd\mu_{\mathbf{x_0}}(\lambda)\dd x_0 \le \int_0^{L_{b_0}}\mu_{\mathbf{x_0}}(I)\,\dd x_0$.

Conversely, suppose $H_{\cQ}$ has no spectrum in $I$. Then $\mathbf{1}_I(H_{\cQ})=0$, so for any $0\le \chi\le \mathbf{1}_I$ in $C_c(\R)$, the operator $\chi(H_{\cQ})$ is trivially trace-class. Using Lemma~\ref{lem:Simon} and arguing as in Lemma~\ref{lemTrace}, we see that $0 = \Tr \chi(H_{\cQ}) = \int_{\cG} \chi(H)(\mathbf{x_0},\mathbf{x_0})\,\dd \mathbf{x_0}$, so by (i), $\int_0^{L_e}\int_I \chi(\lambda)\,\dd\mu_{\mathbf{x_0}}(\lambda)\,\dd \mathbf{x_0}=0$ for all $e$ and the claim follows.

(iii) For $m$ as in \eqref{eq:ConditionCompQG} and $0\neq \chi\le \mathbf{1}_I$ as before, we have
\[
\frac{\mathfrak{N}_N(I)}{N} \ge m \int_{\R} \chi(\lambda)\,\dd\mu_{\cQ_N}(\lambda) \to m\mathbb{E}_{\mathbb{P}}[\chi(H_{\cQ})(\mathbf{x_0},\mathbf{x_0})] = m\mathbb{E}_{\mathbb{P}}\Big(\frac{1}{|\mathbf{B}_{\mathbf{o}}^1|}\int_{\mathbf{B}_{\mathbf{o}}^1} \chi(H_{\cQ})(\mathbf{x},\mathbf{x})\,\dd \mathbf{x}\Big)
\]
 by Lemma~\ref{lem:nalini}. Now suppose $\mathbb{E}_{\mathbb{P}}(\frac{1}{|\mathbf{B}_{\mathbf{o}}^1|}\int_{\mathbf{B}_{\mathbf{o}}^1} \chi(H_{\cQ})(\mathbf{x},\mathbf{x})\,\dd \mathbf{x})=0$. Then $\int_{\mathbf{B}_{\mathbf{o}}^1} \chi(H_{\cQ})(\mathbf{x},\mathbf{x})\,\dd \mathbf{x} =0$ for $\mathbb{P}$-a.e. $[\cQ,\mathbf{x_0}]$. But we showed in \eqref{e:newcool} that if $H_{\cQ}$ has some spectrum in $I$, then there is $b_0\in \cQ$ such that $\int_0^{L_{b_0}}\chi(H_{\cQ})(x_0,x_0)\,\dd x_0>0$. So for $\mathbb{P}$-a.e. $[\cQ,\mathbf{x_0}]$, the operator $H_{\cQ}$ has no spectrum in $I$. This contradicts the hypothesis, completing the proof.
\end{proof}

\section{The case of Kirchhoff conditions without potentials}\label{sec:Kirch}
In this section, we outline the proof of Theorem~\ref{th:GreenContinuous2} in the case where $W_b=0$ for every $b\in \mathcal{B}$, and where we have Kirchhoff boundary conditions at every vertex. This case is of physical relevance, has been much more studied mathematically (see for instance \cite{BeKu13})  and the theory is simpler in this case than the general constructions of the next section.

In this case, a quantum graph will simply be denoted by $\cQ=(V,E,L)$, and a rooted quantum graph will be given by $\mathrm{Q}=(V,E,L,b_0,x_0)$. We denote by $\mathbf{Q}_{\ast}^{D,m,M,\mathrm{K}}\subset \mathbf{Q}_{\ast}^{D,m,M}$ the set of such classes of Kirchhoff rooted quantum graphs.

Let $z\in \C$. For any $f\in \mathscr{H}$ solution of
$- f'' =z f$, we can find coefficients $(a(b))_{b\in \mathcal{B}}$ such that, for any $b\in \mathcal{B}$ and $x\in [0,L_b]$, we have $f_b(x)=a(b) \ee^{\ii\sqrt{z} x} + a(\hat{b}) \ee^{\ii \sqrt{z} (L_b-x)}$, where $\sqrt{z}$ is such that $\Im \sqrt{z}> 0$ whenever $\Im z>0$.

It is well-known (see e.g. \cite[\S~2.1.2]{BeKu13}) that the condition $A_1(v)F(v)+ A_2(v) F'(v)=0$ in \eqref{e:HQ} is equivalent to asking that the coefficients $a(b)$ to satisfy
\begin{equation}\label{e:upre}
a(b) = \sum_{b';\, t(b')= o(b)} \sigma_{b,\widehat{b'}}^{(o(b))}\ee^{\ii\sqrt{z} L_{b'}}a(b') \,,
\end{equation}
where $\sigma^{(v)}$ is the $d(v)\times d(v)$ \emph{scattering matrix}
\[
\sigma_{b,b'}^{(v)}= \frac{2}{d(v)} \text{ if } b\neq b', \qquad  \sigma_{b,b}^{(v)}= -1 + \frac{2}{d(v)}
\]
for $o(b)=o(b')=v$. An important feature of Kirchhoff boundary conditions is that the $\sigma$ matrices do not depend on $z$. This ceases to be true for conditions with a Robin part.

Consider the \emph{evolution operator} $\cU(z)$, which is a matrix of size $|\mathcal{B}|\times |\mathcal{B}|$ defined by
\[
\cU(z)=  SD(z)\,,
\]
where $D(z)$ is the diagonal matrix
\[
D(z)_{b,b}= \ee^{\ii\sqrt{z} L_b}\,,
\]
and $S$ is the unitary matrix\footnote{Given the action of $S$, it may be more accurate to say that $\cU(z)^{\ast}$ is the evolution operator. In this case $S_{b,b'}^{\ast}=\delta_{t(b)=o(b')} \sigma_{b',\widehat{b}}^{o(b')}$ means that we pass from $b$ to some $b'$ which is either outgoing from $b$, or equal to $\widehat{b}$, i.e. reflected. Indeed, \cite{KoSm99} essentially consider $\cU(\lambda^2)^{\ast}$ for real $\lambda$.} 
\[
S_{b,b'}= \delta_{t(b')=o(b)} \sigma^{(o(b))}_{b,\widehat{b'}}\,.
\]

If $\vec{a}=(a(b))_{b\in \mathcal{B}}$, then \eqref{e:upre} reads $\vec{a}=\cU(z)\vec{a}$. Hence, $z$ is an eigenvalue of $H_{\cQ}$ iff
\[
\det\left( \cU(z)-\mathrm{Id}\right) = 0\,.
\]

Let $z\in \C\setminus \R$ and $[\cQ,(b_0,x_0)]\in \mathbf{Q}_{\ast}^{D,m,M,\mathrm{K}}$. Denote $v=o(b_0)$. The function $g^{z,v}:=G_z(v,\cdot)$ is well-defined and satisfies $- f'' = z f$ on every edge  (see Lemma~\ref{lem:greenskernelcts}), so we may find coefficients $a(b)$ such that for any $b\in \mathcal{B}$, we have $g^{z,v}_b(x)=a(b) \ee^{\ii\sqrt{z} x} + a(\hat{b}) \ee^{\ii \sqrt{z} (L_b-x)}$. However, $g^{z,v}\notin\mathscr{H}^\cQ$ as it satisfies special conditions \eqref{eq:LimitCondGreen} at $v$.
 A simple computation shows that the equation $\vec{a}=\cU(z)\vec{a}$ should 
be replaced by
\begin{equation}\label{eq:EasyKirchhoff}
\left(\mathrm{Id}- \cU(z) \right)\vec{a} = \xi_v\,,
\end{equation}
where 
\[
\xi_v(b):= \frac{1}{d(v) \ii\sqrt{z}} \delta_{o(b)=v}\,.
\]
We do not give details of the calculation leading to \eqref{eq:EasyKirchhoff}
because it is a special case of a more general calculation done later
(leading to equation \eqref{eq:CoordGreen}).

When $\Im z>0$, we have $\|D(z)\|<\rme^{-m\Im\sqrt{z}}$ for every $b$, 
so that $\|\cU(z)\|\leq\alpha<1$ for some $\alpha>0$. We thus have
\begin{equation}\label{e:akircase}
\vec{a}= \sum_{k=0}^\infty \cU(z)^k \xi_v \,.
\end{equation}

To simplify the exposition, let us study the continuity of 
\begin{equation}\label{e:contieq}
[V,E,L,b_0,x_0]\mapsto G_z (o(b_0), o(b_0)), \qquad [V,E,L,b_0,x_0]\mapsto G_z (o(b_0), t(b_0))\,,
\end{equation}
instead of the continuity of $[V,E,L,\mathbf{x_0}]\mapsto G_z (\mathbf{x_0}, \mathbf{x_0})$. As $g^{z,v}_b(x)=a(b) \ee^{\ii\sqrt{z} x} + a(\hat{b}) \ee^{\ii \sqrt{z} (L_b-x)}$, 
it thus suffices to establish the continuity of $\mathbf{Q}_{\ast}^{D,m,M,\mathrm{K}}\ni [\cQ,b_0,x_0] \mapsto (a(b_0),a(\widehat{b}_0))\in \C^2$.

For each $z\in \C^+$, the $k$-th term in \eqref{e:akircase} depends only on the quantum  graph in a ball of size $k$ around $v$. If $[\cQ_n,b_n,x_n]\To [\cQ,b_0,x_0]$ with respect to the distance \eqref{eq:defdistance}, then for any $r\in \N^{\ast}$ $\exists\,n_r$ such that for $n\ge n_r$, $B_{G_n^{\mathbf{x}_n}}(v_{\mathbf{x}_n},r) \cong B_{G^{\mathbf{x_0}}}(v_{\mathbf{x_0}},r)$.
This means that if $\xi$ is supported in the ball 
$B_{G^{\mathbf{x_0}}}(v_{\mathbf{x_0}},r)$ then we have
\begin{equation}
  \label{eq:S_equiv}
  S_n\xi = S\xi.
\end{equation}
Furthermore our notion of convergence implies that within the same ball
the edge length data are at most $1/r$ apart. 

We have, for $k\leq r$,
\begin{align*}
  \cU_n(z)^k \xi_v(b_0) &= \cU_n(z)^{k-1} S_n D_n(z)\xi_v(b_0) \\
  &= \cU_n(z)^{k-1} S_n D(z)\xi_v(b_0) + O\Big( c_z \rme^{-(k-1)m \Im\sqrt{z}} 
    \max_{b'\in E_{G^{\mathbf{x_0}}}(v_{\mathbf{x}_0},r)} 
        |\rme^{\rmi\sqrt{z}L_{b'_n}} - \rme^{\rmi\sqrt{z}L_{b'}}| \Big) \\
  &= \cU_n(z)^{k-1} S D(z)\xi_v(b_0) + O\!\left( \frac{c_z}r
 \rme^{-km \Im\sqrt{z}} \right), \qquad\text{using \eqref{eq:S_equiv},} \\
   &= \cU_n(z)^{k-1} \cU(z) \xi_v(b_0) + O\!\left( \frac{c_z}r
 \rme^{-km \Im\sqrt{z}} \right) 
\end{align*}
where $c_z=\|\xi_v\|$ and we used that $S$ is unitary, $\Im z>0$ and 
$|L_{b'_n}-L_{b'}|<1/r$. 
Iterating this we get
\begin{equation}
  \label{eq:14}
  \cU_n(z)^k \xi_v(b_0) = \cU(z)^k\xi_v(b_0) + O\!\left( \frac{c_z}r
 k \rme^{-km \Im\sqrt{z}} \right) 
\end{equation}

Cutting the sum \eqref{e:akircase} at $r$ terms, we get
\begin{align*}
|a_n(b_n)-a(b_0)| &\le \sum_{k=1}^r \left|\cU_n(z)^k\xi_v(b_0) -
\cU(z)^k \xi_v(b_0) \right| + c_z\sum_{k>r} (\|\cU_n(z)\|^k+\|\cU(z)\|^k)\\
& \le O\!\left( \frac1r \right) 
 + O(\alpha^r)=O\Big(\frac{1}{r}\Big), \qquad\text{recalling that
 $0<\alpha<1$.}
\end{align*}
Hence, $a_n(b_n)\To a(b_0)$. 
The same calculation shows that $a_n(\widehat{b}_n)\To a(\widehat{b}_0)$. 

This proves the continuity of the maps \eqref{e:contieq} for $\Im z>0$, and also for $\Im z<0$ by \eqref{eq:GreenSym}.

For general quantum graphs, we will still be able to describe Green's functions by coefficients satisfying \eqref{eq:EasyKirchhoff}. However, $S$ will depend on $z$, and may grow with $\Im z$, and $D(z)$ has a less explicit behaviour. We will therefore have to work a bit to show that $\left(\mathrm{Id}- S D(z) \right)$ can be inverted by a Neumann series; we will only prove this when $\Im z$ is large enough.

\section{Construction of an evolution operator on a general quantum graph}\label{sec:Evolution}

\subsection{Scattering with potentials}
Fix a quantum graph $\cQ=(V,E,L,W,\beta,U)$.

For any $b\in \mathcal{B}$ and $z \in \C$, let $E^z_b:[0,L_b]\To \C$ be the solution of
\begin{equation}\label{eq:Eigenedge}
- f''(x) + W_b(x) f(x) = z f(x) \qquad \text{for } x\in [0,L_b]
\end{equation}
satisfying $E_b^z(0)=1$, $(E_b^z)'(0)= -\ii\sqrt{z}$, where as always, the square-root is chosen such that $\Im \sqrt{z} > 0$ whenever $\Im z>0$. Similarly $E_{\widehat{b}}$ solves the equation with $W_b$ replaced by $W_{\widehat{b}}$. In particular, when $(W_b)\equiv 0$, we have $E_b^z(x)=E_{\widehat{b}}^z(x)=\ee^{-\ii\sqrt{z} x}$ for $x\in[0,L_b]$. Thus, the family of solutions $(E_b^z)_{b\in \mathcal{B}}$ should not be regarded as living on $\cG$, in fact $E^z_{\widehat{b}}(L_b-x)\neq E^z_b(x)$. 

Thanks to \eqref{eq:ReversePot}, we know that $f=E^z_{\widehat{b}}(L_b - \cdot)$ is also a solution of \eqref{eq:Eigenedge}. In fact, $f''(x) = [W_{\widehat{b}}(L_b-x)-z]E^z_{\widehat{b}}(L_b-x) = [W_b(x)-z]f(x)$.
It follows from Lemma~\ref{lem:LargeIm} below that, when $\Im z$ is large enough, the functions $E_b^z (\cdot)$ and $E^z_{\widehat{b}}(L_b - \cdot)$ are two linearly independent solutions of \eqref{eq:Eigenedge}. Indeed, if we apply 
Lemma~\ref{lem:LargeIm} to $\widehat{b}$, then for $f(x)=E^z_{\widehat{b}}(L_b - x)$, we get $\frac{f'(x)}{f(x)}|_{x=0} = -\frac{(E^z_{\widehat{b}})'(L_b)}{E^z_{\widehat{b}}(L_b)} \approx \ii\sqrt{z}$. If we had $f=\alpha E_b^z$ for some $\alpha\in \C$, we would have $\frac{f'(x)}{f(x)}|_{x=0} = \frac{(E_b^z)'(0)}{E_b^z(0)} = -\ii \sqrt{z}$.

So assuming Lemma~\ref{lem:LargeIm} holds, if $(H_\cQ-z)f=0$, then $f_b(x) = a(b)E_b^z(x) + c(b) E_{\widehat{b}}^z(L_b-x)$ and $f_{\widehat{b}}(x) = a(\hat{b})E_{\widehat{b}}^z(x) + c(\hat{b})E_b^z(L_b-x)$. Since $f_{\widehat{b}}(L_b-x) = f_b(x)$, we obtain $c(b)=a(\hat{b})$.

Hence, for any $f=(f_b) \in \mathscr{H}^0$ satisfying $(H_\cQ-z)f=0$, we may find $a(b), a(\hat{b})\in \C$ such that
\[
f_b(x) = a(b) E_b^z(x) + a(\hat{b}) E_{\widehat{b}}^z(L_b-x) \qquad \text{for } x\in [0,L_b]\,.
\]

Given $v\in V$, let $\alpha_v:= (a(\beta^v(j)))_{j=1}^{d(v)}\in \C^{d(v)}$ and $\gamma_v:= \big(a(\widehat{\beta^v(j)}) E_{\widehat{\beta^v(j)}}^z(L_{\beta^v(j)})\big)_{j=1}^{d(v)}$. If we define $F, F'$ as in \eqref{eq:defF}, we get
\begin{equation}\label{e:FF'}
F(v)=\alpha_v + \gamma_v\,,\qquad F'(v)= -\ii\sqrt{z} \alpha_v + \Delta^v(z) \gamma_v,
\end{equation}
where $\Delta^v(z)$ is the diagonal matrix with entries 
\[
\Delta^v_{j,j}(z) = - \frac{\big{(}E^z_{\widehat{\beta^v(j)}}\big{)}'(L_{\beta^v(j)})}{E^z_{\widehat{\beta^v(j)}}(L_{\beta^v(j)})}.
\]

Therefore, the boundary condition $A_1(v)F(v)+A_2(v)F'(v)=0$ in \eqref{e:HQ} becomes
\begin{equation}\label{eq:CondVertex}
A_1(v) \left(\alpha_v+ \gamma_v\right) + A_2(v) \left(  -\ii\sqrt{z} \alpha_v + \Delta^v(z) \gamma_v\right)=0\,.
\end{equation}

When the matrix $\left( A_1(v)- \ii\sqrt{z} A_2(v)\right)$ is invertible, \eqref{eq:CondVertex} is equivalent to
\begin{equation}\label{eq:CondVertex2}
\alpha_v = -\left( A_1(v)- \ii\sqrt{z} A_2(v)\right)^{-1} \left( A_1(v)+ A_2(v)\Delta^v (z)\right) \gamma_v\,.
\end{equation}

Let us write, when it exists\footnote{This $\sigma^v$ is actually the inverse of the matrix $\sigma^{(v)}$ of Section~\ref{sec:Kirch} in the special case of no potential. This is because we take $E_b(x)= \ee^{-\ii\sqrt{z}x}$ as reference in this case while \cite{BeKu13} use $\ee^{\ii\sqrt{z}x}$ as reference.},
\[
\sigma^v(z)=  -\left( A_1(v)- \ii\sqrt{z} A_2(v)\right)^{-1} \left( A_1(v)+ A_2(v)\Delta^v (z) \right) \,.
\]
Equation \eqref{eq:CondVertex2} may then be rewritten as
\[
a(\beta^v(j)) = \sum_{k=1}^{d(v)} \sigma^v(z)_{j,k} a(\widehat{\beta^v(k)}) E_{\widehat{\beta^v(k)}}(L_{\beta^v(k)})\,,\quad \forall v\in V, j=1,\ldots,d(v)
\]
i.e. $a(v,w) = \sum_{u\sim v} \sigma^v(z)_{(\beta^v)^{-1}(v,w),(\beta^v)^{-1}(v,u)} a(u,v) E_{(u,v)}(L_{(v,u)})$. Hence, defining
\[
\begin{aligned}
\mathbf{S}_{b,b'}(z)&:= \delta_{t(b')= o(b)=v} \sigma^v_{(\beta^v)^{-1}(b), (\beta^v)^{-1}(\widehat{b'})}(z)\\
\mathbf{D}_{b,b'}(z)&:= \delta_{b,b'} E_{b}(L_b) \qquad\qquad\qquad\qquad\qquad\qquad \forall b,b'\in \mathcal{B},
\end{aligned}
\]
the relation $A_1(v)F(v)+A_2(v)F'(v)=0$ holds if and only if the vector $\vec{a}=(a(b))_{b\in \mathcal{B}}$ satisfies
\[
\big(\mathbf{S}(z) \mathbf{D}(z) - \mathrm{Id}\big) \vec{a} = 0\,.
\]

\subsection{Invertibility of vertex matrices}
\begin{lemme}\label{lem:InvertMatrix}
The following properties hold true.
\begin{enumerate}[\rm 1)]
\item For all $k\in\R\setminus\{0\}$, the matrix  $\left( A_1(v)- \ii k A_2(v)\right)$ is invertible, and  the matrix
\[
\Theta(k):= \left( A_1(v)- \ii k  A_2(v)\right)^{-1} \left( A_1(v)+ \ii k A_2(v)\right)
\]
is unitary.
\item For any $z\in \C^+$, the matrix $\left( A_1(v)- \ii\sqrt{z} A_2(v)\right)$ is invertible. Furthermore, for any $\Lambda>0$, there exists $C_\Lambda>0$ such that for any $z\in \C^+$ such that $\left| \frac{\Im \sqrt{z}}{\Re \sqrt{z}}\right| \leq \Lambda$,
 if we write $\Theta(\sqrt{z}):= \left( A_1(v)- \ii \sqrt{z}  A_2(v)\right)^{-1} \left( A_1(v)+ \ii \sqrt{z} A_2(v)\right)$, we have
\[
\left\| \Theta(\sqrt z)  \right\|^2\le C_\Lambda := \sqrt{1+\Lambda^2}+\Lambda.
\]
\end{enumerate}
\end{lemme}

\begin{remarque}
This lemma allows us to bound $\Theta(\sqrt{z})$ when $z$ has an argument contained in a compact subset of $(0,\pi)$. In the sequel, we will use it for $\Re z\in [-K,K]$ and $\Im z$ large enough.
\end{remarque}

\begin{proof}
A proof of the first point can be found in \cite[Lemma 1.4.7]{BeKu13}. 

For the second point, let $\sqrt{z} = k+\ii\eta$, $k,\eta\in\R$, and $A_i:=A_i(v)$. Then
\begin{multline*}
\Theta(\sqrt{z})=\left(A_1- \ii k A_2 +\eta A_2\right)^{-1} \left( A_1+ \ii k A_2 - \eta A_2\right) \\
=\left( \mathrm{Id}  +\eta \left( A_1- \ii k A_2\right)^{-1}A_2 \right)^{-1} \left( A_1- \ii k A_2\right)^{-1}  \left( A_1+ \ii k A_2\right) \left( \mathrm{Id} - \eta \left( A_1+ \ii k A_2\right)^{-1}A_2 \right) .
\end{multline*}
Note that $\mathrm{Id} - \Theta(k) = (A_1-\ii k A_2)^{-1}[(A_1-\ii kA_2)-(A_1+\ii kA_2)] = -2\ii k (A_1-\ii k A_2)^{-1}A_2$. Similarly, $\mathrm{Id} - \Theta^{-1}(k) = 2\ii k (A_1+\ii k A_2)^{-1}A_2$. Hence,

\begin{align*}
\Theta(\sqrt{z}) &= \left( \mathrm{Id} + \frac{\ii \eta}{2 k} \left( \mathrm{Id} - \Theta(k)\right) \right)^{-1} \Theta(k) \left( \mathrm{Id} + \frac{\ii \eta}{2 k} \left( \mathrm{Id} - \Theta^{-1}(k)\right) \right).
\end{align*}
Let us denote by $\mu=\eta/k$ and recall we assume $|\mu|\leq\Lambda$.
We have that $\Theta(\sqrt{z})$ has the same eigenvectors as $\Theta(k)$, and it is hence diagonalizable in an orthonormal basis. We may thus bound the norm of $\Theta(\sqrt{z})$ by controlling its spectral radius. If $\omega$ is an eigenvalue of $\Theta(k)$, the associated eigenvalue of $\Theta(\sqrt{z})$ is given
by the mapping
\begin{equation}\label{eq:NewEigenvalue}
\omega \mapsto
\omega \frac{1 + \frac{\ii \mu}{2} \left(1-\omega^{-1}\right)} {1 + \frac{\ii \mu}{2} \left( 1- \omega\right)}.
\end{equation}
The mapping \eqref{eq:NewEigenvalue} is a M\"obius transformation with
fixed points $\pm1$.  Therefore the image of the unit circle is another
circle passing through the points $+1$ and $-1$, i.e., symmetric about the
imaginary axis.  The maximum (and minimum) modulus occur where the new circle
cuts the imaginary axis.  It is a simple calculation to check that
the points $\omega_\pm = \frac{\mu^2 \pm 2\rmi\sqrt{1+\mu^2}}{2+\mu^2}$
get mapped to
\begin{displaymath}
  \pm \rmi \sqrt{1+\mu^2} - \rmi \mu,
\end{displaymath}
from which the bound $C_\Lambda$ follows\footnote{This said, the value of $C_\Lambda$ is actually not important for us, a rougher bound would suffice.}. 
\end{proof}

\subsection{The behaviour of solutions in the upper half-plane}
Let $\cQ$ be a quantum graph, $b\in \mathcal{B}$, and $z\in \C$. We define $C_b^z$ and $S_b^z$ to be the basis of solutions of \eqref{eq:Eigenedge} on $[0,L_{b}]$ satisfying
\begin{equation}\label{e:cs}
\begin{pmatrix} C_b^z(0) & S_b^z(0) \\ (C_b^z)'(0) & (S_b^z)'(0) \end{pmatrix} = \begin{pmatrix} 1 & 0 \\ 0 & 1 \end{pmatrix} .
\end{equation}
When $W_b\equiv 0$, these are cosines and sines. Note that
\begin{equation}\label{e:eincs}
E_b^z = C_b^z - \ii\sqrt{z}S_b^z \,,
\end{equation}
as both sides solve \eqref{eq:Eigenedge} with the same boundary conditions. Our aim now is to prove the following.

\begin{lemme}\label{lem:LargeIm}
Let $K>0$. There exists $C_0,C_1,C_2=C_i(\|W\|_{\infty}, K, \underline{L}(\cQ))>0$ such that for all $b\in \mathcal{B}$ and all $z\in \C$ satisfying $\Re z\in [-K,K]$ and $\Im z>C_0$, we have
\[
\left|E_b^z(L_b)\right|\geq C_1 \Im z \qquad \text{and}\qquad \left|\frac{\big{(} E_b^z \big{)}'(L_b)}{E_b^z(L_b)} + \ii \sqrt{z} \right| \leq C_2 \frac{|\sqrt{z}|}{\Im z} \,.
\]
\end{lemme}
Before proving the lemma, let us recall some facts. As shown in \cite[p.7]{PT87},
\begin{equation}
  \label{eq:integral_eqs}
  \begin{aligned}
  C_b^z(x) &= \cos\sqrt{z}x + \int_0^x \frac{\sin\sqrt{z}(x-t)}{\sqrt{z}} W_b(t)
           C_b^z(t)\,\dd t \\
  S_b^z(x) &= \frac{\sin\sqrt{z}x}{\sqrt{z}} + \int_0^x \frac{\sin\sqrt{z}(x-t)}%
           {\sqrt{z}} W_b(t) S_b^z(t)\,\dd t\,.
 \end{aligned}
\end{equation}

Using \eqref{e:eincs} and the bounds in \cite[p.13]{PT87}, we also have
\begin{align}\label{eq:4}
|E_b^z(x)| & \le |C_b^z(x) - \cos\sqrt{z} x + \ii \sin\sqrt{z} x - \ii \sqrt{z}S_b^z(x)| + |\cos\sqrt{z}x-\ii\sin\sqrt{z}x| \nonumber \\
& \le \exp(x\Im \sqrt{z}) + \frac{2}{|\sqrt{z}|}  \exp(x\Im\sqrt{z} + \|W_b\|_\infty
  \sqrt{xL_b})
\end{align}

\begin{proof}[Proof of Lemma~\ref{lem:LargeIm}]
  Using \eqref{e:eincs} and \eqref{eq:integral_eqs} we can write
  \begin{equation}\label{eq:7}
    E^z_b(x) = \ee^{-\ii\sqrt{z}x} +\int_0^{x}
       \frac{\sin\sqrt{z}(x-t)}{\sqrt{z}} W_b(t) E_b^z(t) \,\dd t.
    \end{equation}
    Trivially we can bound $|\sin\sqrt{z}(L_b-t)| \leq
    \ee^{(L_b-t)\Im\sqrt{z}}$ for $0\leq t \leq L_b$, so by \eqref{eq:4},
    \begin{align*}
      &\Bigg| \int_0^{L_b} \frac{\sin\sqrt{z}(L_b-t)}{\sqrt{z}} W_b(t) E_b^z(t) \dd t \Bigg| \\
      &\qquad \leq \frac{\|W_b\|_\infty}{|\sqrt{z}|} \ee^{L_b\Im\sqrt{z}}
        \int_0^{L_b} \ee^{-t \Im\sqrt{z}} \left(\ee^{t\Im \sqrt{z}} + \frac{2}{|\sqrt{z}|}  \ee^{t\Im\sqrt{z} + \|W_b\|_\infty\sqrt{tL_b}}\right)\dd t \\
      &\qquad \leq \frac{\|W_b\|_\infty}{|\sqrt{z}|} \ee^{L_b\Im\sqrt{z}}L_b\left(1+\frac{2\ee^{\|W_b\|_\infty L_b}}{|\sqrt{z}|}\right).
    \end{align*}
    Since $|\ee^{-\ii\sqrt{z}L_b}| = \ee^{L_b\Im\sqrt{z}}$, this proves
    that
    \begin{equation}
      \label{eq:8}
      |E_b^z(L_b)| = \ee^{L_b\Im \sqrt{z}} \left( 1 + \mathrm{O}\left(
        \frac1{|\sqrt{z}|} \right) \right),
  \end{equation}
  where the implied constant depends on $\|W_b\|_\infty$ and $L_b$. Since $\Im z = 2\Im\sqrt{z}\Re\sqrt{z}$, the first claim follows (using the lower bound on $L_b$ and the bound on $\Re z$).

  Differentiating \eqref{eq:7} and adding, we have
  \begin{equation}
    \label{eq:9}
    (E_b^z)'(L_b) + \ii\sqrt{z} E_b^z(L_b)
    = \int_0^{L_b} \ee^{\ii\sqrt{z}(L_b-t)} W_b(t) E_b^z(t)\dd t\,,
  \end{equation}
  so, assuming $|\sqrt{z}|\ge 1$, \eqref{eq:4} yields
  \begin{align}
    \left|(E_b^z)'(L_b) + \ii\sqrt{z} E_b^z(L_b)\right|
    &\leq 3\| W_b \|_\infty \int_0^{L_b} \ee^{-(L_b-t)\Im\sqrt{z}}
      \exp( t\Im\sqrt{z} + \|W_b\|_\infty\sqrt{tL_b})\,\dd t \nonumber \\
    &\leq 3 \|W_b\|_\infty \ee^{-L_b\Im\sqrt{z} + \|W_b\|_\infty L_b}
      \int_0^{L_b} \ee^{2t\Im\sqrt{z}}\,\dd t \nonumber \\
    &=3 \|W_b\|_\infty \ee^{\| W_b \|_\infty L_b} \frac{\sinh(L_b\Im\sqrt{z})}%
      {\Im\sqrt{z}}.
    \label{eq:10}
  \end{align}
  Combining \eqref{eq:8} and \eqref{eq:10}, we deduce that
  \begin{align}
    \left|\frac{\big{(} E_b^z \big{)}'(L_b)}{E_b^z(L_b)} +
    \ii \sqrt{z} \right| & \leq 3 \| W_b \|_\infty \ee^{\|W_b\|_\infty L_b}
                           \frac{\sinh(L_b\Im\sqrt{z})}{\Im\sqrt{z}}
                           \ee^{-L_b\Im\sqrt{z}} \left( 1 + \mathrm{O}\left(
                           \frac1{|\sqrt{z}|}\right)\right) \nonumber \\
   &\le \frac{3\|W_b\|_\infty\ee^{\|W_b\|_\infty L_b}}{\Im\sqrt{z}}
      \left( 1 + \mathrm{O}\left(\frac1{|\sqrt{z}|}\right)\right).
    \label{eq:11}
  \end{align}
	Since $\Im z = 2\Im\sqrt{z}\Re\sqrt{z}$, the proof is complete.
\end{proof}

\subsection{Back to the evolution operator}
\begin{lemme}\label{lem:InvertScatt}
Fix $0<K_1<K_2$ and $0<m<M$. There exist $C,C'$ depending on $K_1,K_2, m, M$ such that, for any quantum graph $\cQ$ satisfying \eqref{eq:ConditionCompQG} and any $z\in \C$ with $\Re z\in [K_1,K_2]$ and $\Im z>C'$, the matrix $\mathbf{S}(z)$ is invertible, and  
\[
\left\|\mathbf{S}(z)^{-1}\right\|_{\ell^2(\mathcal{B})\To \ell^2(\mathcal{B})}\le C\,.
\]
\end{lemme}
\begin{proof}
Recall that $[\mathbf{S}(z)a](v,w)=\sum_{u\sim v} \sigma^v(z)_{(\beta^v)^{-1}(v,w),(\beta^v)^{-1}(v,u)} a(u,v)$. 
If $J: \C^\mathcal{B}\rightarrow \C^\mathcal{B}$ is the orientation-reversing operator, then $\mathbf{S}(z)J$ is thus block-diagonal, with blocks $\sigma^v(z)$. As $J$ is an isometry, $\mathbf{S}(z)$ is hence invertible iff each $\sigma^v(z)$ is invertible, and we have $\big{\|}\mathbf{S}(z)^{-1} \big{\|} = \sup \limits_{v\in V} \big{\|}\big{(}\sigma^v(z)\big{)}^{-1}\big{\|}$.

Now,
\begin{align*}
\Theta(\sqrt{z}) & = \left( A_1(v)- \ii \sqrt{z} A_2(v)\right)^{-1}\left( A_1(v)+ A_2(v)\Delta^v(z) + \ii\sqrt{z} A_2(v) - A_2(v)\Delta^v(z)\right)\\
& = - \sigma^v(z)+ \left( A_1(v)- \ii \sqrt{z} A_2(v)\right)^{-1}  A_2(v) \left( \ii\sqrt{z}\,\mathrm{Id}- \Delta^v(z)\right)\\
&= - \sigma^v(z) + \frac{1}{2\ii \sqrt{z}} \big{(} \Theta(\sqrt{z})- \mathrm{Id}\big{)}\big{(} \ii\sqrt{z}\,\mathrm{Id}- \Delta^v(z)\big{)}\,,
\end{align*}
so
\[
\sigma^v(z) = \Theta(\sqrt{z})\left[-\mathrm{Id} + \frac{1}{2\ii\sqrt{z}}\left(\mathrm{Id}-\Theta^{-1}(\sqrt{z})\right)\left(\ii\sqrt{z}-\Delta^v(z)\right)\right] .
\]
Note that $\Theta^{-1}(\sqrt{z}) = \Theta(-\sqrt{z})$. By Lemma \ref{lem:InvertMatrix}, $\big{\|} \mathrm{Id}- \Theta^{-1}(\sqrt{z})\big{\|}$ is bounded, and by Lemma \ref{lem:LargeIm}, we have $\big{\|} \ii\sqrt{z}\,\mathrm{Id}- \Delta^v(z)\big{\|} \le C \frac{|\sqrt{z}|}{\Im z}$ when $\Im z$ is large enough. Hence, for $\Im z$ large enough,
\[
\left\| \frac{1}{2\ii\sqrt{z}}\left(\mathrm{Id}-\Theta^{-1}(\sqrt{z})\right)\left(\ii\sqrt{z}-\Delta^v(z)\right) \right\| \leq \frac{C}{\Im z}.
\]
The result follows.
\end{proof}

\begin{corollaire}\label{cor:InvertLargeIm}
Let $0<K_1<K_2$ and $0<m<M$. There exists $C>0$ depending on $K_1,K_2, m, M$ such that, for any quantum graph $\cQ$ satisfying \eqref{eq:ConditionCompQG} and any $z\in \C$ with $\Re z\in [K_1,K_2]$ and $\Im z>C$, we have 
\[
\left\|\big{(} \mathbf{S}(z) \mathbf{D}(z)\big{)}^{-1}\right\|_{\ell^2(\mathcal{B})\rightarrow \ell^2(\mathcal{B})} <\frac{1}{2} \,.
\]
\end{corollaire}
\begin{proof}
By Lemma \ref{lem:LargeIm}, we have $\|\big(\mathbf{D}(z)\big)^{-1}\big\|_{\ell^2\to \ell^2}\le \frac{C}{\Im z}$. 
The result follows from Lemma \ref{lem:InvertScatt}.
\end{proof}

\section{Continuity of Green's kernel}\label{sec:proof}

The aim of this section is to prove Theorem \ref{th:GreenContinuous2}. This follows from the following proposition:
\begin{proposition}\label{prop:GreenContinuous}
Let $M>m>0$, $D\in \N$, $0<K_1<K_2$. There exists $C=C(D,M,m,K_1,K_2)>0$ such that for all $z\in \C$ with $\Im z> C$, $\Re z\in [K_1,K_2]$,  $\mathbf{G}_z$ is continuous on $\mathbf{Q}_*^{D,m,M}$.
\end{proposition}
Recall that $\mathbf{G}_z$ is the notation for the map $\mathbf{G}_z:\mathbf{Q}_* \mapsto G_z(\mathbf{x_0},\mathbf{x_0})\in \C$.
\begin{proof}[Proof of Theorem \ref{th:GreenContinuous2} from Proposition \ref{prop:GreenContinuous}]
Let $\{\mathrm{Q}_n\}\subset \mathbf{Q}_*^{D,m,M}$ be a sequence of (equivalence classes of) rooted quantum graphs converging to $\mathrm{Q}\in  \mathbf{Q}_*^{D,m,M}$. By Proposition \ref{prop:GreenContinuous}, we have
\begin{equation}\label{eq:continuityG}
\lim\limits_{n\To\infty} \mathbf{G}_z(\mathrm{Q}_n)= \mathbf{G}_z(\mathrm{Q})
\end{equation}
for all $z\in \C$ such that $\Re z \in [K_1,K_2]$ and $\Im z>C$. Now $\mathbf{G}_z(\mathrm{Q}_n)$ is a sequence of holomorphic functions on $\C^+:=\{z\in \C, \Im z>0\}$, as shown in Lemma \ref{lem:Herglotz}, which is locally bounded by Lemma~\ref{lem:BoundGreenImagin}. We may thus apply the Vitali-Porter theorem to conclude that \eqref{eq:continuityG} holds on $\C^+$, and that the convergence is uniform on all compact sets of $\C^+$. From  \eqref{eq:GreenSym}, we see that \eqref{eq:continuityG} actually holds on all of $\C\setminus \R$.
\end{proof}

The rest of the section is devoted to the proof of Proposition~\ref{prop:GreenContinuous}. The idea is to use the fact (proved in Lemma~\ref{lem:greenskernelcts}) that $f(\mathbf{y_0})= G_z(\mathbf{x_0},\mathbf{y_0})$ solves \eqref{eq:Eigenedge} for $\mathbf{y_0}\neq \mathbf{x_0}$, so the formalism of Section~\ref{sec:Evolution} can be applied to $f$ for $\mathbf{y_0}$ in the two ``sub-edges'' of $b_0$ of lengths $x_0$ and $L_{b_0}-x_0$. An issue arises however: the new edge lengths $x_0$ or $L_{b_0}-x_0$ can be very small.  This can risk invalidating
the results of Section \ref{sec:Evolution}. (Recall we assumed a positive
lower bound on all edge lengths.)  To resolve this, given $\mathbf{x_0}$, we first study $G_z(\mathbf{x}_{b_0}^{(\kappa)},\mathbf{y_0})$ with $\mathbf{x}_{b_0}^{(\kappa)}$ near the midpoint of $b_0$. At the end, we relate this to $G_z(\mathbf{x_0},\mathbf{y_0})$ to conclude the proof.

\begin{proof}[Proof of Proposition~\ref{prop:GreenContinuous}]

Given $[\cQ,(b_0,x_0)]$, we consider the point 
\[
\mathbf{x_1}=\mathbf{x}_{b_0}^{(\kappa)}:=(b_0, \kappa L_{b_0})\in \mathcal{G}
\]
with $\kappa\in [\frac{1}{2},\frac{2}{3}]$. We define a quantum graph $\cQ^{\mathbf{x_1}}$ by adding a vertex at $\mathbf{x_1}$ with Kirchhoff boundary conditions, following the procedure described in Definition \ref{def:adding}.

The function $f=G_z(\mathbf{x_1}, \cdot)$ solves $Hf=zf$, so we may find coefficients $\vec{a}=\big(a(b)\big)_{b\in \cB^{\mathbf{x_1}}}$ such that inside every edge of $G^{\mathbf{x_1}}$, we have $f_b(x) = a(b) E_b^z(x) + a(\widehat{b}) E_{\widehat{b}}^z(L_b-x)$. Let us write $b_{\mathbf{x_0}}$ for the oriented edge of $\cQ^{\mathbf{x_1}}$ containing $\mathbf{x_0}$ with origin $v_1$ (this is either $(v_1,o(b_0))$ or $(v_1,t(b_0))$). Denoting by $\mathsf{x_0}$ the position of $\mathbf{x_0}$ on the oriented edge $b_{\mathbf{x_0}}$, we have
\begin{equation}\label{e:expamid}
G_z(\mathbf{x_1},\mathbf{x_0})= a(b_{\mathbf{x_0}}) E^z_{b_{\mathbf{x_0}}}(\mathsf{x_0}) + a(\widehat{b}_{\mathbf{x_0}}) E^z_{\widehat{b}_{\mathbf{x_0}}}(L_{b_{\mathbf{x_0}}}- \mathsf{x_0}) \,.
\end{equation}

The map $\mathbf{Q}_*^{D,m,M} \ni  [\cQ,\mathbf{x_0}]\mapsto (E^z_{b_{\mathbf{x_0}}}(\mathsf{x_0}), E^z_{\widehat{b}_{\mathbf{x_0}}}(L_{b_{\mathbf{x_0}}}-\mathsf{x}_0))\in \C^2$ is continuous by \cite[p. 10]{PT87}.\footnote{In fact, $E_b^z(x)=y_1(x,z,W_b)-\ii\sqrt{z}y_2(x,z,W_b)$ and $E^z_{\widehat{b}}(L_b-x)=y_1(L_b-x,z,W_{\widehat{b}})-\ii\sqrt{z}y_2(L_b-x,z,W_{\widehat{b}})$ with $y_j(x)$ from \cite{PT87} independent of $L_b$ (as they are defined using initial conditions at $0$). It follows from \cite[p. 10]{PT87} that the maps are continuous in $W$, $H^2$ in $x$, and the map $L_b\mapsto E^z_{\widehat{b}}(L_b-x)$ is $H^2$.} To show that $\mathbf{Q}_*^{D,m,M} \ni  [\cQ,\mathbf{x_0}]\mapsto \big{(}a(b_{\mathbf{x_0}}), a(\widehat{b}_{\mathbf{x_0}})\big{)}\in \C^2$ is continuous, we use the results of Section~\ref{sec:Evolution}, as we now demonstrate.

Recall that $f=G_z(\mathbf{x_1}, \cdot)$ satisfies \eqref{eq:LimitCondGreen} at $v_1$. Defining $F$ as in \eqref{eq:defF} and recalling \eqref{e:a1a2kir}, this corresponds to having $A_1(v_1)F(v_1)+A_2(v_1) F'(v_1) = \frac{-2}{d(v_1)}\begin{pmatrix}1\\1\end{pmatrix} = -\begin{pmatrix} 1\\ 1 \end{pmatrix}$. Using \eqref{e:FF'}, this yields 
\[
\alpha_{v_1}=\sigma^{v_1} \gamma_{v_1} - \big( A_1(v_1) - \ii \sqrt{z} A_2(v_1)\big)^{-1} \begin{pmatrix}
1\\1 \end{pmatrix}\,.
\]
In general $U_v = \frac{2}{d(v)} \mathbbm{1}-\id$ implies that $A_1(v) - \ii \sqrt{z} A_2(v) = \frac{2\ii}{d(v)}(1-\sqrt{z})\mathbbm{1}-2\ii\, \id$ and thus $(A_1(v) - \ii \sqrt{z} A_2(v))^{-1} = \frac{1}{2\ii}(\frac{\sqrt{z}-1}{d(v)\sqrt{z}}\mathbbm{1}-\id)$. Here $d(v_1)=2$ and we get
\[
\alpha_{v_1}=\sigma^{v_1} \gamma_{v_1} + \frac{1}{2\ii \sqrt{z}} \begin{pmatrix} 1\\1 \end{pmatrix}.
\]
By Lemma~\ref{lem:GreenBCs} (ii), at all remaining vertices $f$ satisfies the
standard boundary conditions, so we have that
\begin{equation}\label{eq:CoordGreen}
\left(\mathrm{Id}-\mathbf{S}(z) \mathbf{D}(z) \right)\vec{a} = \frac{1}{2\ii\sqrt{z}} \delta_{o_b=v_1}.
\end{equation}

Let us write $\xi_{v_1}(z)$ for the vector of $\C^{\cB^{\mathbf{x_1}}}$ with 2 non-zero components $\frac{1}{2\ii\sqrt{z}} \delta_{o_b=v_1}$. By Corollary~\ref{cor:InvertLargeIm} and \eqref{eq:CoordGreen}, we obtain that for any $0<K_1<K_2$, there exists $C=C(D,m,M,K_1,K_2)>0$ such that for all $z\in \C$ with $\Re z \in [K_1,K_2], \Im z \ge C$, we can expand
\begin{equation}\label{eq:NeumanSeries}
\vec{a} = -\sum_{k=1}^\infty \left(\mathbf{S}(z) \mathbf{D}(z)\right)^{-k} \xi_{v_1}(z) \,.
\end{equation}

Here we used the fact that if $\|A^{-1}\|<1$, then $(I-A)^{-1} =-A^{-1}(I-A^{-1})^{-1}=-A^{-1}\sum_{k=0}^{\infty}A^{-k}$.

Now suppose $[\cQ_n,\mathbf{x}_n]\To [\cQ,\mathbf{x_0}]$. By \eqref{eq:NeumanSeries}, we have
\[
a_n(b_{\mathbf{x}_n}) = - \sum_{k=1}^\infty \Big{(}\left(\mathbf{S}_n(z) \mathbf{D}_n(z)\right)^{-k} \xi_{v_1^n}(z)\Big{)}(b_{\mathbf{x}_n}) \,.
\]
As noted in Lemma~\ref{lem:InvertScatt}, if $J:\C^\mathcal{B}\To\C^\mathcal{B}$ is the orientation-reversing operator, then $\mathbf{S}(z)\mathbf{D}(z)J$ is block-diagonal, so $\big(\mathbf{S}(z)\mathbf{D}(z)J\big)^{-1}$ is also block-diagonal. Hence, $\big(\mathbf{S}(z)\mathbf{D}(z)\big)^{-1} = J\big(\mathbf{S}(z)\mathbf{D}(z)J\big)^{-1}$ is local, in the sense that the $k$-th term in the sum above only depends on the quantum graph in a ball of size $k$ around $v_1^n$. Consequently, for any $k\in \N$, we have
\[
\left(\left(\mathbf{S}_n(z) \mathbf{D}_n(z)\right)^{-k} \xi_{v_1^n}(z)\right)(b_{\mathbf{x}_n}) \To  \left(\left(\mathbf{S}(z) \mathbf{D}(z)\right)^{-k} \xi_{v_1}(z)\right)(b_{\mathbf{x_0}})\,.
\]
On the other hand, for all $n$, we have $\big|\big(\big(\mathbf{S}_n(z) \mathbf{D}_n(z)\big)^{-k} \xi_{v_1^n}(z)\big)(b_{\mathbf{x}_n})\big|\leq \frac{1}{\sqrt{z}} 2^{-k}$, which is summable. We may apply the dominated convergence theorem to conclude that $a_n(b_{\mathbf{x}_n}) \To a(b_{\mathbf{x_0}})$. The same argument works for $a(\widehat{b_{\mathbf{x_0}}})$. To summarize, we have shown that $\mathbf{Q}_*^{D,m,M} \ni  [\cQ,\mathbf{x_0}]\mapsto G_z(\mathbf{x}_{b_0}^{(\kappa)},\mathbf{x_0})$ is continuous for all $\kappa\in [\frac{1}{2},\frac{2}{3}]$ and all $z$ with $\Re z \in [K_1,K_2]$, $\Im z \ge C(D,m,M,K_1,K_2)$.

Let us fix $\kappa=\frac12$ and denote $\mathbf{x_1}=\mathbf{x}_{b_0}^{(1/2)}$.
Applying Lemma~\ref{lem:maxime} on the new edge $b_1=(v_1,t(b_0))$ with 
$\zeta = \frac16 L_{b_0} (=\frac23L_{b_0}-\frac12L_{b_0})$ we get that
\begin{displaymath}
   \frac{\dd}{\dd x} \Big|_{x =  L_{b_0}/2}  G_z((b_0, x),\mathbf{x_0}) 
= \langle Z_z^{b_1,L_{b_0}/6}, G_z(\cdot, \mathbf{x_0}) \rangle,
\end{displaymath}
where we have proved that $G_z(\cdot, \mathbf{x_0})$ is continuous throughout
the integration range.  As in the proof of Lemma~\ref{lem:BoundGreenImagin}
we may find an upper bound for $Z_z^{b_1,L_{b_0}/6}$, and 
Dominated Convergence permits us to conclude that
 $\mathbf{Q}_*^{D,m,M} \ni  [\cQ,\mathbf{x_0}]\mapsto \frac{\mathrm{d}}{\mathrm{d} x} \Big{|}_{x =  L_{b_0}/2}  G_z((b_0, x),\mathbf{x_0})$ is also continuous.

To conclude, keeping $\mathbf{x_1}=\mathbf{x}_{b_0}^{(\frac{1}{2})}$, let $C_{b_{\mathbf{x_0}}}^z$, $S_{b_{\mathbf{x_0}}}^z$ be the functions on $[0,\frac{L_{b_0}}{2}]$ defined in \eqref{e:cs}. Then, with $\mathsf{x_0}$ the position
of $\mathbf{x_0}$ on $b_{\mathbf{x_0}}$,
\begin{equation}\label{e:firstexpan}
G_z(\mathbf{x_0},\mathbf{y_0}) = C_{b_{\mathbf{x_0}}}^z( \mathsf{x_0}) G_z(\mathbf{x_1},\mathbf{y_0}) + S_{b_{\mathbf{x_0}}}^z( \mathsf{x_0}) \frac{\dd}{\dd x }\Big{|}_{x =\frac{L_{b_0}}{2}}  G_z((b_0,x),\mathbf{y_0}).
\end{equation}

By \cite[p. 10]{PT87}, the map $\mathbf{Q}_*^{D,m,M} \ni  [\cQ,\mathbf{x_0}]\mapsto (C_{b_{\mathbf{x_0}}}^z(\mathsf{x_0}), S_{b_{\mathbf{x_0}}}^z(\mathsf{x_0}))\in \C^2$ is continuous. Thus, $\mathbf{Q}^{D,m,M}_* \ni  [\cQ,\mathbf{x_0}]\mapsto G_z(\mathbf{x_0},\mathbf{x_0})$ is continuous for all $z$ with $\Re z \in [K_1,K_2]$, $\Im z \geq C$.
\end{proof}

\section{Convergence of empirical measures}\label{sec:ConvEmp}
\subsection{The limit of empirical measures in terms of the resolvent}\label{subsec:ConvV1}
We begin with the following lemma, which gives the first part of Theorem \ref{th:Empirical} when the function $\chi$ is smooth.

\begin{lemme}\label{lem:FirstPart}
Let $\cQ_N$ be a sequence of quantum graphs satisfying \eqref{eq:ConditionCompQG} and converging in the sense of Benjamini-Schramm to $\mathbb{P}$.

Let $\chi \in C_c^\infty(\R)$. We have
\begin{align*}
\lim\limits_{N\rightarrow \infty} \int_{\R} \chi(\lambda) \mathrm{d}\mu_{\cQ_N}(\lambda) = \lim\limits_{\varepsilon \downarrow 0} \frac{1}{\pi}\int_{\R} \chi(\lambda) \mathbb{E}_{\mathbb{P}} \big{[} \Im \mathbf{G}_{\lambda+\ii\varepsilon}  \big{]} \mathrm{d}\lambda \,.
\end{align*}
\end{lemme}

\begin{proof}
If $\chi \in C_c^\infty(\R)$, we will denote by $\tilde{\chi}$  an \emph{almost analytic} extension of $\chi$, i.e., a smooth function $\tilde{\chi} : \C \mapsto \C$ such that $\tilde{\chi}(z) = \chi(z)$ for $z\in \R$, 
\begin{equation}\label{e:dbar}
\frac{\partial \tilde{\chi}}{\partial \overline{z}}(z) = O ((\Im z)^2),
\end{equation}
 and 
\begin{equation}\label{eq:support extension}
\mathrm{supp }~ \tilde{\chi}\subset \{z; \Re z \in \mathrm{supp}~\chi\}.
\end{equation}
Here $\frac{\partial}{\partial\overline{z}}=\frac{1}{2}(\frac{\partial}{\partial x}+\ii\frac{\partial}{\partial y})$.
For instance, one can take $\tilde{\chi}(x+\ii y) = \chi(x) + \ii y \chi'(x) - \frac{y^2}{2} \chi''(x)$.
We refer the reader to \cite[\S 2.2]{Davies} for more details about almost analytic extensions.

Assume $\supp \chi\subset (-a,a)$, $a>0$. Consider the rectangle $\Omega_\varepsilon := \left[-a-\varepsilon; a+\varepsilon\right] \times \ii \left[-\varepsilon; \varepsilon\right]$ and the curve $\Gamma_\varepsilon := \partial \Omega_\varepsilon$. Cauchy's integral formula $f(\lambda) = \frac{1}{2\pi \ii} \int_{\Gamma} \frac{f(z)}{z-\lambda}\,\dd z$ for analytic $f$ generalizes to
\[
\tilde{\chi}(\lambda) = \frac{1}{2\pi \ii} \left\{\int_{\Gamma_{\varepsilon}} \frac{\tilde{\chi}(z)}{z-\lambda}\,\dd z + \int_{\Omega_{\varepsilon}} \frac{\partial \tilde{\chi}/ \partial \overline{z}}{z-\lambda}\,\dd z\wedge \dd \overline{z}\right\}
\]
for $\lambda \in \Omega_{\varepsilon}$, where $\dd z\wedge \dd \overline{z}=-2\ii\,\dd x\dd y$, see e.g. \cite[Theorem 1.2.1]{Hor}. Hence,
\begin{align}\label{e:uneq}
\cL(\cQ_N)\int_{\R} \chi(\lambda)\,\dd \mu_{\cQ_N}(\lambda) &= \sum_{\lambda_i^N\in (-a,a)} \chi(\lambda_i^N)\\
& = \frac{-1}{2\pi \ii}\left\{\int_{\Gamma_{\varepsilon}} \tilde{\chi}(z)f_N(z)\,\dd z + \int_{\Omega_{\varepsilon}}f_N(z) \frac{\partial \tilde{\chi}}{\partial \overline{z}}  \,\dd z\wedge \dd \overline{z}\right\} ,\nonumber
\end{align}
where $f_N(z):= \mathrm{Tr} \big{[} (H_{\cQ_N}-z)^{-1} \big{]} = \sum_{i \in \N} \frac{1}{\lambda_i^N-z}$. Thanks to Lemma \ref{lemTrace},
\begin{equation}\label{e:encoreuneq}
f_N(z)=   \int_{\cG_N} G_z(\mathbf{x_0},\mathbf{x_0})\mathrm{d}\mathbf{x_0}\,.
\end{equation}

From Lemma \ref{lem:BoundGreenImagin}, we know that there exists $C(D,m,M,a)$ such that for all $z\in \Omega_{\varepsilon} \setminus \R$, we have
\begin{equation}\label{e:EstimTrace}
\int_{\cG} \left|G_z(\mathbf{x_0},\mathbf{x_0})\right|\mathrm{d}\mathbf{x_0} \leq \frac{C(D,m,M,a)}{|\Im z|} \mathcal{L}(\cQ) \,.
\end{equation}

Recalling that $\frac{\partial \tilde{\chi}}{\partial \overline{z}}(z) = O ((\Im z)^2)$, \eqref{e:encoreuneq} and \eqref{e:EstimTrace} imply that 
\[
\left|\frac{1}{\mathcal{L}(\cQ_N)}\int_{\Omega_{\varepsilon}} f_N(z) \frac{\partial \tilde{\chi}}{\partial \overline{z}} \mathrm{d} z \wedge \mathrm{d} \overline{z}\right| \le C \varepsilon
\]
for some $C>0$ independent of $N$ and $\varepsilon$. On the other hand,
\[
 \frac{1}{\mathcal{L}(\cQ_N)} \int_{\Gamma_\varepsilon} f_N(z) \tilde{\chi}(z)\mathrm{d}z =  \int_{\Gamma_\varepsilon}\frac{1}{\mathcal{L}(\cQ_N)}\int_{\cG_N} G_z(\mathbf{x_0},\mathbf{x_0})\mathrm{d}\mathbf{x_0} \,\tilde{\chi}(z)\mathrm{d}z \underset{N\To+\infty}{\To} \int_{\Gamma_\varepsilon}\mathbb{E}_\mathbb{P} \left[ \mathbf{G}_z \right] \tilde{\chi}(z) \mathrm{d}z
\]
by Corollary~\ref{cor:bsconv}. Therefore, from \eqref{e:uneq} we get
\begin{equation}\label{e:afirstlimit}
\int_\R \chi(\lambda) \mathrm{d}\mu_{\cQ_N}(\lambda) \underset{N\To+\infty}{\To}   \frac{-1}{2\pi\ii} \int_{\Gamma_\varepsilon}\tilde{\chi}(z) \mathbb{E}_\mathbb{P} \left[ \mathbf{G}_z \right]\mathrm{d}z + O(\varepsilon).
\end{equation}

Lemma \ref{lem:Herglotz}, along with \eqref{eq:GreenSym} tells us that $z\mapsto G_z(\mathbf{x_0},\mathbf{x_0})$ satisfies the assumptions of the following lemma.

\begin{lemme}\label{lem:HerglotzRocks}
Let $F$ be a holomorphic function on $\C\setminus \R$, such that $\Im F(z)>0$ if $\Im z>0$, and
 $F(\overline{z})= \overline{F}(z)$. Then there exists a constant $C= C(\chi, a, F(\ii))$ such that
\begin{align*}
\Big{|}\frac{-1}{2\pi\ii}\int_{\Gamma_\varepsilon}  \tilde{\chi}(z) F(z) \mathrm{d}z - \frac{1}{\pi}  \int_\R \chi(\lambda) \Im F(\lambda +\ii\varepsilon) \mathrm{d}\lambda \Big{|} \leq C\sqrt{\varepsilon}\,.
\end{align*}
\end{lemme}

Recalling \eqref{e:afirstlimit}, we may hence conclude that 
\[
\lim\limits_{N\rightarrow \infty} \int_\R \chi(\lambda) \mathrm{d}\mu_{\cQ_N}(\lambda) = \lim\limits_{\varepsilon \downarrow 0} \frac{1}{\pi}\int_{\R} \chi(\lambda) \mathbb{E}_\mathbb{P}\big{[} \Im \mathbf{G}_{\lambda +\ii\varepsilon} \big{]} \mathrm{d}\lambda \,,
\]
which concludes the proof of Lemma \ref{lem:FirstPart}.
\end{proof}

\begin{proof}[Proof of Lemma \ref{lem:HerglotzRocks}]

  By \cite[Theorem 5.9.1]{Simon15}, for a function $F$ satisfying the
  assumptions, there exist $c\geq 0$, $d\in \R$, and a positive finite
  measure $\mu$ on $\R$ such that
\begin{equation}\label{eq:HerglotzRep}
F(z) = cz+ d+ \int_\R \frac{1+xz}{x-z}\, \mathrm{d}\mu(x)\,.
\end{equation}
One immediately checks that $F(\ii)= d + \ii c + \ii \mu(\R)$, so that $c\leq \Im F(\ii)$,  $\mu(\R)\leq \Im F(\ii)$, $d=\Re F(\ii)$.

Using  \eqref{eq:support extension}, \eqref{eq:HerglotzRep} and the hypothesis $F(\overline{z})=\overline{F}(z)$, we see that
\begin{align*}
&\frac{-1}{2\pi\ii}\int_{\Gamma_\varepsilon}  \tilde{\chi}(z) F(z) \mathrm{d}z = \frac{1}{2\pi\ii}\int_{-a}^{a} \Big{(}\tilde{\chi}(\lambda+\ii\varepsilon) F(\lambda+\ii\varepsilon) - \tilde{\chi}(\lambda-\ii\varepsilon) F(\lambda-\ii\varepsilon) \Big{)}   \mathrm{d} \lambda\\
&\qquad= \frac{1}{2\pi\ii}\int_{-a}^{a} \Big{(}\tilde{\chi}(\lambda+\ii\varepsilon)  (c\lambda + \ii \varepsilon c + d)- \tilde{\chi}(\lambda-\ii\varepsilon) (c\lambda - \ii \varepsilon c + d)\Big{)}   \mathrm{d} \lambda\\
&\qquad \quad+ \frac{1}{2\pi\ii}\int_{-a}^{a} \Big{(}\tilde{\chi}(\lambda+\ii\varepsilon)  \int_\R \frac{1+x(\lambda+\ii\varepsilon)}{x-\lambda-\ii\varepsilon} \mathrm{d}\mu(x)- \tilde{\chi}(\lambda-\ii\varepsilon)\int_\R \frac{1+x(\lambda-\ii\varepsilon)}{x-\lambda+\ii\varepsilon} \mathrm{d}\mu(x)  \Big{)}\mathrm{d} \lambda\\
&\qquad= \frac{1}{2\pi\ii} \int_\R \Big{(}\int_{-a}^{a} \Big{(} \tilde{\chi}(\lambda+\ii\varepsilon) \frac{1+x(\lambda+\ii\varepsilon)}{x-\lambda-\ii\varepsilon}  -  \tilde{\chi}(\lambda- \ii\varepsilon) \frac{1+x(\lambda-\ii\varepsilon)}{x-\lambda+ \ii\varepsilon} \Big{)} \mathrm{d}\lambda \Big{)} \mathrm{d}\mu(x) + O(\varepsilon)\,,
\end{align*}
where we used Fubini's theorem in the last equality to exchange the integrals. Here, the $O(\varepsilon)$ depends only on $\chi$, $a$, and $F(\ii)$.

Consider the quantity
\[
f_\varepsilon(\lambda,x):= \left(1+x(\lambda+\ii\varepsilon)\right) \frac{ \tilde{\chi}(\lambda+\ii\varepsilon) - \chi(\lambda)}{x-\lambda -\ii\varepsilon}.
\]

When $|x|\ge 1+a(a+2)$, we have $\frac{|1+x\lambda+\ii x\varepsilon|}{|x-\lambda-\ii\varepsilon|} \le \frac{1+(a+1)|x|}{|x|-a} \le a+2$. Hence,
\[
|f_\varepsilon(\lambda,x)|\leq (a+2)  \varepsilon \sup\limits_{z\in [-a,a] + \ii [-1,1]} |\tilde{\chi}'(z)|\,.
\]

Next, suppose $|x|\leq 1+a(a+2)$. If $|x-\lambda|\geq \sqrt{\varepsilon}$, we have 
\begin{align*}
 |f_\varepsilon(\lambda,x)|&\le \left( 1+ [1+a(a+2)](a+1) \right) \frac{|\tilde{\chi}(\lambda+\ii\varepsilon) - \chi(\lambda)|}{\sqrt{\varepsilon}} \\
 &\le \sqrt{\varepsilon} \left( 1+ [1+a(a+2)](a+1) \right) \sup\limits_{z\in [-a,a] + \ii [-1,1]} |\tilde{\chi}'(z)|.
 \end{align*}

Finally, if $|x-\lambda|< \sqrt{\varepsilon}$, we have $f_\varepsilon(\lambda, x) \leq  \left( 1+ [1+a(a+2)](a+1) \right) \sup\limits_{z\in [-a,a] + \ii [-1,1]} |\tilde{\chi}'(z)|$.

All in all, we may find a constant $C=C(a,\tilde{\chi})$ such that for all $x\in \R$, we have
\[
\Big|\int_{-a}^a f_\varepsilon(\lambda,x) \dd\lambda \Big| \leq C \sqrt{\varepsilon}\,.
\]

Therefore, we obtain that 
\[
\Big|\int_\R \int_{-a}^a  \tilde{\chi}(\lambda+\ii\varepsilon) \frac{1+x(\lambda+\ii\varepsilon)}{x-\lambda-\ii\varepsilon} \dd\lambda \dd\mu(x) - \int_\R \int_{-a}^a  \chi(\lambda) \frac{1+x (\lambda+\ii\varepsilon)}{x-\lambda-\ii\varepsilon} \dd\lambda \dd\mu(x)\Big|\leq C\sqrt{\varepsilon} \mu(\R)\,.
\]
Similarly, we obtain
\[
\Big{|}\int_\R \int_{-a}^a  \tilde{\chi}(\lambda- \ii\varepsilon) \frac{1+x(\lambda- \ii\varepsilon)}{x-\lambda+\ii\varepsilon} \dd\lambda \dd\mu(x) -\int_\R \int_{-a}^a  \chi(\lambda) \frac{1+x (\lambda-\ii\varepsilon) }{x-\lambda + \ii\varepsilon} \dd\lambda \dd\mu(x)\Big{|}\leq C\sqrt{\varepsilon} \mu(\R)\,.
\]

Therefore, we have
\begin{align*}
\frac{-1}{2\pi\ii}\int_{\Gamma_\varepsilon}  \tilde{\chi}(z) F(z) \mathrm{d}z 
&= \frac{1}{2\pi\ii} \int_\R \int_{-a}^{a} \Big{(}\chi(\lambda) \frac{1+x (\lambda+\ii\varepsilon) }{x-\lambda-\ii\varepsilon}  -  \chi (\lambda) \frac{1+x(\lambda-\ii\varepsilon)}{x-\lambda+ \ii\varepsilon} \Big{)} \mathrm{d}\lambda \mathrm{d}\mu(x) + O(\sqrt{\varepsilon})\\
&= \frac{1}{2\pi\ii} \int_{-a}^a \chi(\lambda)\Big{(} 2\ii c\varepsilon + 2\ii \Im \Big[\int_\R \frac{1+x(\lambda+\ii\varepsilon)}{x-\lambda-\ii\varepsilon} \mathrm{d}\mu(x)\Big]\Big{)}\,\dd\lambda +O(\sqrt{\varepsilon})\\
&= \frac{1}{\pi}  \int_{-a}^a  \chi(\lambda) \Im F(\lambda +\ii\varepsilon) \mathrm{d}\lambda +O(\sqrt{\varepsilon})\,,
\end{align*}
with the constant in the $O$ depending only on $\chi$, $a$, and $F(\ii)$.
\end{proof}

\subsection{Integral kernels of functions of Schrödinger operators}\label{subsec:HelfSj}
The aim of this section is to prove Theorem \ref{th:IntegralKernelsAreCool}, which is the combination of Lemma \ref{lem:Simon} and of Lemma \ref{lem:FoncOpCont} below.

The first lemma is an adaptation of \cite[Lemma B.7.9]{Simon82}.

\begin{lemme}\label{lem:Simon}
Let $\cQ$ be a quantum graph satisfying \eqref{eq:ConditionCompQG}, and let $\chi\in L^\infty(\R)$ have compact support. The operator $\chi(H_\cQ)$, defined through the functional calculus, has a continuous kernel, in the sense that for any $b_0,b_1\in \mathcal{B}$, $ (0, L_{b_0}) \times (0, L_{b_1})\ni (x_0,y_0)\mapsto \chi\big{(}H_\cQ\big{)}\big{(}(b_0,x_0),(b_1,y_0)\big{)}$ is continuous.

Furthermore, there exists a constant $C$ depending on the support of $\chi$ and on $D,m,M$ such that for any $\mathbf{x_0}, \mathbf{y_0}\in \cG$, we have
\begin{equation}\label{eq:KernLinfBound}
 \big{|}\chi\big{(}H_\cQ\big{)}(\mathbf{x_0}, \mathbf{y_0})\big{|} \leq C \|\chi\|_{L^\infty}.
\end{equation}

Finally, if $\chi$ is nonnegative, then for any $\mathbf{x_0}, \mathbf{y_0}\in \cG$, we have
\begin{equation}\label{eq:Cauchy-Schwarz}
\big{|}\chi\big(H_\cQ\big)(\mathbf{x_0}, \mathbf{y_0})\big{|} \leq \sqrt{\chi\big(H_\cQ\big)(\mathbf{x_0}, \mathbf{x_0}) \chi\big{(}H_\cQ\big{)}(\mathbf{y_0}, \mathbf{y_0})}.
\end{equation}
\end{lemme}

\begin{proof}
Let us fix $z\in \C$. We set $\varphi(x):= (x-z)^2 \chi(x)$, so that we may find $C$ depending only on the support of $\chi$ such that $\|\varphi\|_{L^{\infty}}\leq C \|\chi\|_{L^{\infty}}$. By the usual functional calculus, we have $\chi(H_{\cQ}) = \big( H_\cQ-z\big{)}^{-1} \varphi\big( H_\cQ\big) \big( H_\cQ-z\big)^{-1}$.

Denote by $g^{z,\mathbf{x}}$ the function $g^{z,\mathbf{x}}(\mathbf{y}) =G_z(\mathbf{x}, \mathbf{y}) = ( H_\cQ-z)^{-1} (\mathbf{x},\mathbf{y})$.
If $f_1,f_2\in \mathscr{H}$, we have using \eqref{eq:GreenSym},
\[
\langle f_1, \chi(H_\cQ) f_2 \rangle_{L^2(\cG)} = \int_{\cG\times \cG} \dd \mathbf{x_0} \dd \mathbf{x_1} \overline{f_1}(\mathbf{x_0}) f_2(\mathbf{x_1})  \langle g^{\bar{z},{\mathbf{x_0}}}, \varphi (H_\cQ) g^{z,{\mathbf{x_1}}} \rangle_{L^2(\cG)}\,.
\]
Therefore, $\chi(H_\cQ)$ has integral kernel $\langle g^{\bar{z},{\mathbf{x_0}}}, \varphi (H_\cQ) g^{z,{\mathbf{x_1}}}\rangle_{L^2(\cG)}$. 

We know from Remark \ref{rem:L2Cont} that $\mathbf{x_0}\mapsto g^{z,\mathbf{x_0}}$ is $L^2$-continuous. On the other hand, from the spectral theorem, $\|\varphi (H_\cQ)\|\le \|\varphi\|_{\infty} \le C \|\chi\|_{\infty}$, as previously observed. We deduce that $(\mathbf{x_0},\mathbf{x_1})\mapsto \langle g^{\bar{z},{\mathbf{x_0}}}, \varphi (H_\cQ) g^{z,{\mathbf{x_1}}} \rangle_{L^2(\cG)}$ is jointly continuous. The bound \eqref{eq:KernLinfBound} also follows using Cauchy-Schwarz and Lemma  \ref{lem:greenskernelL2}.

Finally, let us prove \eqref{eq:Cauchy-Schwarz}. As in the proof of Lemma \ref{lem:Herglotz}, we may use the continuity of the integral kernel to find sequences $f^0_n$, $f^1_n$ of functions in $\mathscr{H}$ such that for $i,j=0,1$, we have
\[
\chi\big{(}H_\cQ\big{)}(\mathbf{x}_i, \mathbf{x}_j) =\lim\limits_{n\rightarrow \infty} \langle f^i_n, \chi(H_\cQ) f^j_n\rangle.
\]
Now, we have
\begin{align*}
 \big|\langle f_n^0, \chi(H_\cQ) f_n^1\rangle\big| &= \big|\langle \chi(H_\cQ)^{1/2} f_n^0, \chi(H_\cQ)^{1/2} f_n^1\rangle\big| \\
& \le \|\chi(H_{\cQ})^{1/2}f_n^0\|\cdot \|\chi(H_{\cQ})^{1/2}f_n^1\|= \sqrt{\langle f_n^0, \chi(H_\cQ) f_n^0\rangle\langle f_n^1, \chi(H_\cQ) f_n^1\rangle},
 \end{align*}
 which converges to  $\sqrt{\chi\big{(}H_\cQ\big{)}(\mathbf{x_0}, \mathbf{x_0}) \chi\big{(}H_\cQ\big{)}(\mathbf{x_1}, \mathbf{x_1})}$. The result follows.
\end{proof}

When $\chi$ is smooth, we can obtain an expression for $\chi(H_\cQ)(\mathbf{x_0}, \mathbf{x_1})$ using the Helffer-Sjöstrand formula. In the following, $\tilde{\chi}$ is an almost-analytic extension of $\chi$, as described in \S \ref{subsec:ConvV1}.

\begin{lemme}\label{lem:IntegralKernel}
Let $\cQ$ be a quantum graph satisfying \eqref{eq:ConditionCompQG} and let $\chi \in C_c^\infty(\R)$. Then we have 
\begin{equation}\label{eq:intKern}
\chi\big(H_\cQ\big)(\mathbf{x_0}, \mathbf{y_0}) = \frac{-1}{2\pi\ii} \int_\C \frac{\partial \tilde{\chi}(z)}{\partial \overline{z}} \big( H_\cQ - z\big)^{-1}(\mathbf{x_0}, \mathbf{y_0}) \,\mathrm{d}z  \wedge \dd \overline{z}\,.
\end{equation}
\end{lemme}

\begin{proof}
Thanks to the Helffer-Sj\"ostrand formula (\cite[equation 2.2.3]{Davies}), we have
\[
\chi(H_\cQ) = \frac{-1}{2\pi\ii} \int_\C \frac{\partial \tilde{\chi}(z)}{\partial \overline{z}} \big(H_\cQ-z\big)^{-1} \dd z  \wedge \dd \overline{z}\,.
\]

Let $f_1,f_2\in \mathscr{H}$ be continuous and compactly supported. Since all the integrands below are continuous and compactly supported, we may exchange the integrals by Fubini's theorem to obtain
\begin{align*}
\int_\cG \dd & \mathbf{x_0}\int_\cG \dd \mathbf{y_0}\Big[ \frac{-1}{2\pi\ii} \int_\C \frac{\partial \tilde{\chi}(z)}{\partial \overline{z}} \big(H_\cQ-z\big)^{-1}(\mathbf{x_0}, \mathbf{y_0})\,\dd z  \wedge \dd \overline{z}\Big] \overline{f_1(\mathbf{x_0})} f_2(\mathbf{y_0})\\
 &= \frac{-1}{2\pi\ii}  \int_\C \frac{\partial \tilde{\chi}(z)}{\partial \overline{z}} \int_\cG \mathrm{d}\mathbf{x_0}\int_\cG \dd \mathbf{y_0} \big(H_\cQ-z\big{)}^{-1}(\mathbf{x_0}, \mathbf{y_0}) \overline{f_1(\mathbf{x_0})} f_2(\mathbf{y_0}) \,\dd z  \wedge \dd \overline{z}\\
 &= \frac{-1}{2\pi\ii}  \int_\C \frac{\partial \tilde{\chi}(z)}{\partial \overline{z}} \langle f_1, \big{(} H_\cQ-z\big{)}^{-1}f_2 \rangle \,\dd z  \wedge \dd \overline{z}\\
 &= \Big{\langle} f_1, \Big[\frac{-1}{2\pi\ii} \int_\C \frac{\partial \tilde{\chi}(z)}{\partial \overline{z}} \big(H_\cQ-z\big)^{-1} \,\dd z  \wedge \dd\overline{z} \Big] f_2 \Big{\rangle} = \langle f_1, \chi(H_\cQ) f_2\rangle.
\end{align*}
This proves \eqref{eq:intKern}.
\end{proof}

\begin{lemme}\label{lem:FoncOpCont}
Let $\chi\in C_c(\R)$. The map $\mathbf{Q}_*^{D,m,M}\ni [\cQ, \mathbf{x_0}] \mapsto \chi(H_\cQ)(\mathbf{x_0}, \mathbf{x_0})$ is continuous.
\end{lemme}
\begin{proof}
Let $[\cQ_N, \mathbf{x}_N]$ be a sequence in $\mathbf{Q}_*^{D,m,M}$ converging to $[\cQ_0, \mathbf{x_0}]$.

We first suppose that $\chi\in C_c^\infty$, with $\supp \chi\subset (-a,a)$, $a>0$, and let $\varepsilon>0$. The kernel  $\chi(H_\cQ)(\mathbf{x_0}, \mathbf{x_0})$ then takes the form \eqref{eq:intKern}.
Consider the rectangle $\Omega_\eta := \left[-a-\eta; a+\eta\right] \times \ii \left[-\eta; \eta\right]$. Since $\tilde{\chi}$ satisfies \eqref{e:dbar}, we may use Lemma \ref{lem:BoundGreenImagin} to find $\eta>0$ such that for all $N\in \N\cup \{0\}$,
\[
\Big|\int_{\Omega_\eta} \frac{\partial \tilde{\chi}(z)}{\partial \overline{z}} \big( H_{\cQ_N} - z\big)^{-1}(\mathbf{x}_N,\mathbf{x}_N) \,\dd z \wedge \dd \overline{z} \Big{|} \leq \frac{\varepsilon}{3}.
\]
On the other hand, if $\Gamma_\eta := \partial \Omega_\eta $, then by Stokes' theorem and the fact that $z\mapsto (H-z)^{-1}(\mathbf{x},\mathbf{x})$ is analytic on $\supp\tilde{\chi}\setminus \Omega_{\eta}$, we have
\[
\int_{\C\setminus \Omega_\eta} \frac{\partial \tilde{\chi}(z)}{\partial \overline{z}} \big(H_{\cQ_N}-z\big)^{-1}(\mathbf{x}_N,\mathbf{x}_N)\, \dd z \wedge \dd \overline{z} = \int_{\Gamma_\eta } \tilde{\chi}(z) \big(H_{\cQ_N}-z\big)^{-1}(\mathbf{x}_N,\mathbf{x}_N) \,\dd z\,.
\]
But  $(H_{\cQ_N}-z)^{-1}(\mathbf{x}_N,\mathbf{x}_N)$ converges to  $(H_{\cQ_0}-z)^{-1}(\mathbf{x_0},\mathbf{x_0})$ uniformly on compact subsets of $\C\setminus \R$ by the Vitali-Porter theorem, so we may find $n_0\in \N$ such that for all $N\geq n_0$,
\[
\Big|\int_{\Gamma_\eta } \tilde{\chi}(z) \big(H_{\cQ_0}-z\big)^{-1}(\mathbf{x_0},\mathbf{x_0}) \dd z - \int_{\Gamma_\eta } \tilde{\chi}(z) \big(H_{\cQ_N}-z\big)^{-1}(\mathbf{x}_N,\mathbf{x}_N) \dd z \Big| \leq \frac{\varepsilon}{3}\,.
\]
We deduce from this that for all $N\geq n_0$, we have 
\[
\left|\chi(H_{\cQ_0}) (\mathbf{x_0}, \mathbf{x_0}) -  \chi(H_{\cQ_N}) (\mathbf{x}_N, \mathbf{x}_N)\right| \le \varepsilon\,,
\]
which concludes the proof when $\chi \in C_c^\infty(\R)$.

The general case follows from the density of $C_c^\infty(\R)$ in $C_c(\R)$ along with \eqref{eq:KernLinfBound}.
\end{proof}

\subsection{End of proof of the main result}\label{sec:end}
We may now proceed with the proof of the second part of Theorem \ref{th:Empirical}, which we recall.

\begin{theo}
Let $\cQ_N$ be a sequence of quantum graphs satisfying \eqref{eq:ConditionCompQG} and converging in the sense of Benjamini-Schramm to $\mathbb{P}$. Then for any $\chi \in C_c(\R)$, we have
\begin{align*}
\lim\limits_{N\rightarrow \infty} \int_{\R} \chi(\lambda) \mathrm{d}\mu_{\cQ_N}(\lambda) = \int_{\mathbf{Q}_*} \chi \big{(}H_\cQ\big{)}(\mathbf{x_0}, \mathbf{x_0}) \dd \mathbb{P}[\cQ,\mathbf{x_0}]\,.
\end{align*}
\end{theo}

This result implies that Lemma~\ref{lem:FirstPart} remains true for $\chi\in C_c(\R)$, concluding the proof of Theorem~\ref{th:Empirical}. The equality of the two limits on the RHS of \eqref{eq:ConvSpecMeas} for $\chi\in C_c(\R)$ follows by two approximations, the first using $|\lim_\varepsilon \int_{\R} [\chi(\lambda)-\varphi(\lambda)]\mathbb{E}[\Im \mathbf{G}_{\lambda+\ii\varepsilon}]\dd\lambda| \le \|\chi-\varphi\|_{\infty}\lim_{\varepsilon} \int_K \mathbb{E}[\Im \mathbf{G}_{\lambda+\ii\varepsilon}]\,\dd\lambda \le C \|\chi-\varphi\|_{\infty}$ for some compact $K$, where $C_c^\infty\ni \varphi\approx \chi$ and $C=\lim_{\varepsilon}\int_K\mathbb{E}[\Im G_{\lambda+\ii\varepsilon}]\,\dd\lambda\le \lim_{\varepsilon}\int h(\lambda)\mathbb{E}[\Im G_{\lambda+\ii\varepsilon}]\,\dd\lambda = \int_{\mathbf{Q}_{\ast}}h(H_{\cQ})(\mathbf{x_0},\mathbf{x_0})\,\dd\mathbb{P}[\cQ,\mathbf{x_0}]\le C'\|h\|_{\infty}$, with $\mathbf{1}_K\le h\in C_c^\infty(\R)$ and we used \eqref{eq:KernLinfBound}. The second approximation uses \eqref{eq:KernLinfBound}.

\begin{proof}
We first consider the case $\chi \in C_c^\infty(\R)$, with support in $(-a,a)$. We may then compute $\chi(H_\cQ) (\mathbf{x_0}, \mathbf{x_0})$ using \eqref{eq:intKern}. Arguing as in Lemma~\ref{lem:FoncOpCont}, the integral over $\Omega_\varepsilon$ is bounded by $C \varepsilon$ for some $C$ depending only on $\chi$, and on $D,m,M$. As for the integral on $\Gamma_\varepsilon$, we may express it using Lemma \ref{lem:HerglotzRocks} (cf. Appendix~\ref{sec:app1}). We obtain that there exists $C'>0$ such that
\begin{equation}\label{eq:ExprIntKern}
\Big{|}\chi\big{(}H_\cQ\big{)}(\mathbf{x_0}, \mathbf{x_0})-  \frac{1}{\pi}  \int_\R \chi(\lambda) \Im \Big{[} \big(H_\cQ - \lambda - \ii\varepsilon\big)^{-1}(\mathbf{x_0},\mathbf{x_0}) \Big{]} \mathrm{d}\lambda\Big{|} \leq C' \sqrt{\varepsilon}.
\end{equation}
The constant $C'$, which is given by Lemma \ref{lem:HerglotzRocks}, depends on $\chi$, $a$, and on $\big( H_\cQ-\ii\big{)}^{-1}(\mathbf{x_0},\mathbf{x_0})$. By Lemma \ref{lem:BoundGreenImagin}, the constant $C'$ can be taken to depend only on $\chi$, $a$, $D$, $m$ and on $M$.

We deduce that
\begin{align*}
\int_{\mathbf{Q}_*} \chi \big(H_\cQ\big)(\mathbf{x_0}, \mathbf{x_0}) \dd \mathbb{P}[\cQ,\mathbf{x_0}] 
&= \frac{1}{\pi} \int_{\mathbf{Q}_*} \dd \mathbb{P}[\cQ,\mathbf{x_0}]\int_\R \chi(\lambda) \Im \left[ G_{\lambda+\ii\varepsilon}(\mathbf{x_0},\mathbf{x_0}) \right] \dd \lambda +O(\sqrt{\varepsilon})\\
&=  \frac{1}{\pi} \int_\R \dd\lambda \chi(\lambda)\int_{\mathbf{Q}_*} \dd \mathbb{P}[\cQ,\mathbf{x_0}] \Im \left[ G_{\lambda+\ii\varepsilon}(\mathbf{x_0},\mathbf{x_0}) \right]  +O(\sqrt{\varepsilon}),
\end{align*}
where we exchanged the integrals using Fubini.
Taking the limit $\varepsilon\downarrow 0$, we obtain
\begin{align*}
\int_{\mathbf{Q}_*} \chi \big{(}H_\cQ\big{)}(\mathbf{x_0}, \mathbf{x_0}) \dd \mathbb{P}[\cQ,\mathbf{x_0}] 
&= \lim\limits_{\varepsilon\downarrow 0}  \frac{1}{\pi} \int_\R \mathrm{d}\lambda \chi(\lambda)\int_{\mathbf{Q}_*} \dd \mathbb{P}[\cQ,\mathbf{x_0}] \Im \left[ G_{\lambda+\ii\varepsilon}(\mathbf{x_0},\mathbf{x_0}) \right] \\
&=\lim\limits_{N\rightarrow \infty} \int_{\R} \chi(\lambda) \mathrm{d}\mu_{\cQ_N}(\lambda),
\end{align*}
thanks to Lemma \ref{lem:FirstPart}. This concludes the proof when $\chi \in C_c^\infty(\R)$.

Now, let $\chi\in C_c(\R)$ and choose an approximation $\varphi\in C_c^\infty(\R)$ supported in some compact $K\supseteq \supp \chi$. We have
\begin{multline*}
\Big{|} \int_\R \chi(\lambda)\, \dd\mu_{\cQ_N}(\lambda) - \int_{\mathbf{Q}_*} \chi \big{(}H_\cQ\big{)}(\mathbf{x_0}, \mathbf{x_0}) \dd \mathbb{P}[\cQ,\mathbf{x_0}]\Big{|} \\
\leq \Big{|} \int_\R \varphi(\lambda)\, \mathrm{d}\mu_{\cQ_N}(\lambda) -\int_{\mathbf{Q}_*} \varphi \big{(}H_\cQ\big{)}(\mathbf{x_0}, \mathbf{x_0}) \dd \mathbb{P}[\cQ,\mathbf{x_0}]\Big{|} +  \int_\R \big{|}\chi(\lambda)-\varphi(\lambda)\big{|}\, \mathrm{d}\mu_{\cQ_N}(\lambda) \\
+ \int_{\mathbf{Q}_*}  \big{|}\chi \big(H_\cQ\big)(\mathbf{x_0}, \mathbf{x_0})-\varphi \big{(}H_\cQ\big{)}(\mathbf{x_0}, \mathbf{x_0})\big|\,\dd \mathbb{P}[\cQ,\mathbf{x}_0]\,.
\end{multline*}
Let $\tilde{\chi}\in C_c^\infty(\R)$ be equal to 1 on $K$ and denote
\[
C:= \max \Big{(} 1, \sup_{N\in \N} \int_\R \tilde{\chi}(\lambda)\, \mathrm{d}\mu_{\cQ_N}(\lambda) \Big{)}\,,
\]
which is finite, since the sequence in $N$ converges.

Given $\varepsilon>0$, assume that $|\chi(t)-\varphi(t)|\leq \frac{\varepsilon}{2C} \tilde{\chi}(t)$. Then using \eqref{eq:KernLinfBound}, we deduce that
\[
\limsup\limits_{N\To \infty} \Big{|} \int_\R \chi(\lambda)\, \mathrm{d}\mu_{\cQ_N}(\lambda) -\int_{\mathbf{Q}_*} \chi \big{(}H_\cQ\big{)}(\mathbf{x_0}, \mathbf{x_0}) \dd \mathbb{P}[\cQ,\mathbf{x_0}]\Big{|}<\varepsilon\,.
\]
As this holds for any $\varepsilon>0$, the result follows.
\end{proof}

\section{Examples}\label{sec:exam}

\subsection{Rooted convergence}\label{subsec:rooted}

Before delving into Benjamini-Schramm convergence, let us consider examples where $[\cQ_N,\mathbf{x}_N]\To [\cQ,\mathbf{x}]$ and the consequence of Theorem~\ref{th:GreenContinuous2}.

\begin{enumerate}[(a)]
\item Fix an infinite quantum graph $\cQ$, fix some $\mathbf{x_0}\in \cG$ and consider a sequence of growing balls $\cQ_N\subset \cQ$ centered at $\mathbf{x_0}$, of combinatorial radius $N$. Then $d([\cQ_N,\mathbf{x_0}],[\cQ,\mathbf{x_0}]) = \frac{1}{1+N} \To 0$. Consequently, we have $G_z^{\cQ_N}(\mathbf{x_0},\mathbf{x_0}) \To G_z^{\cQ}(\mathbf{x_0},\mathbf{x_0})$.

This pointwise result also holds strongly: we know from \cite[Theorem 17]{Kuch} that functions of compact support on $\cQ$ form a core, so using \cite[Theorem VIII.25]{RS} we have $(H_{\cQ_N}-z)^{-1} \xrightarrow{s} (H_{\cQ}-z)^{-1}$.
\item Given a quantum graph $\cQ$, Theorem~\ref{th:GreenContinuous2} implies that $G_z^{\cQ}(\mathbf{x_0},\mathbf{x_0})$ is continuous under varying the data $(L_b,W_b,U_v)_{b\in \mathcal{B},v\in V}$.

In the special case where $\cQ$ is a tree with Kirchhoff conditions and zero potential, the continuity of $G_z^{\cQ}(\mathbf{x_0},\mathbf{x_0})$ in $(L_b)_{b\in \mathcal{B}}$ was proved in \cite{ASW06} using special arguments
(see also \cite{BeKu12}).
\end{enumerate}

\medskip

We now turn to Benjamini-Schramm (BS) convergence. It is useful to recall some examples in the discrete setting. Roughly speaking, a sequence $(G_N)$ converges to $G$ in the BS sense if for any $k$, a $k$-ball taken around a random point in $G_N$ looks like a $k$-ball taken at random in $G$. In view of this, one sees that in the BS sense,
\begin{itemize}
\item Cycle graphs $C_N$ converge to $\Z$,
\item lattice cubes $\Lambda_N=\{1,\dots,N\}^d$ converge to $\Z^d$,
\item $d$-regular graphs with a negligible number of cycles converge to the $d$-regular tree $\mathbb{T}$.
\item If we turn to random graphs, it is also true that Erd\"os-R\'enyi graphs with edge probability $p(N)= c/N$ converge a.s. to the Poisson-Galton-Watson tree with mean $c$. 
\item Sequences of balls $T_N=B_{\mathbb{T}_q}(o,N)$ in a $(q+1)$-regular tree $\mathbb{T}_q$ \emph{do not} converge to $\mathbb{T}_q$ in the BS sense. The problem is that $k$-balls near the boundary of $T_N$ look very different than $k$-balls in $\mathbb{T}_q$, and there is a huge number of such problematic points (as $\frac{|\partial T_N|}{|T_N|} \To \frac{q-1}{q}>0$), so this problem cannot be neglected. It turns out that the graphs $T_N$ do have a BS limit: it is a \emph{random rooted canopy tree}; see \cite{AWcan} and \cite[Chapter 14]{Gab}.
\end{itemize}

With these basic examples in the discrete setting, we now consider quantum graphs.

\subsection{Graphs with a single data}\label{subsec:singledat}
Suppose that $(G_N)$ are discrete graphs which converge in the sense of Benjamini-Schramm to some $\rho\in\cP( \mathscr{G}_*)$, where $\mathscr{G}_*$ is the set of equivalence classes of discrete rooted graphs. Endow all edges/vertices of $G_N$ with the same data $(L_\star,W_\star,U_\star)$, in particular $\cQ_N$ is equilateral (meaning that all edges have the same length). Then $\cQ_N = (G_N,L_\star,W_\star,U_\star)$ converges in the sense of Benjamini-Schramm to $\mathbb{P}\in \cP(\mathbf{Q}_*)$ defined by
\[
\mathbb{E}_{\mathbb{P}}(f) = \frac{1}{L_\star\mathbb{E}_{\rho}(\deg o)}\int_{\mathscr{G}_{\ast}}\sum_{o'\sim o} \int_0^L f[G,L_\star,W_\star,U_\star,(b_{o'},x_0)]\,\dd x_0\,\dd \rho([G,o]) \,,
\]
where $b_{o'}=(o,o')$. To see this, note that $\mathbb{E}_{\nu_{\cQ_N}}(f)= \frac{1}{L_\star\sum_{v\in V_N}\deg(v)}\sum_{v\in V_N} F[G_N,v]$, where $F[G,v] = \sum_{w\sim v}\int_0^{L_\star} f[G,L_\star,W_\star,U_\star,b_{(v,w)},x_0]\,\dd x_0$. If $f$ is continuous on $\mathbf{Q}_*$, then $F$ is continuous on $\mathscr{G}_*$ and the result follows.

\subsection{\texorpdfstring{\textit{N}}{N}-lifts}\label{subsec:nlifts}
Fix a finite quantum graph $\cQ_1$. Consider a (connected) $N$-lift $G_N$ of the underlying combinatorial graph $G_1=(V,E)$. This is just an $N$-cover over $G$. Lift the data naturally to $G_N$ by $L_{(v,w)} = L_{(\pi_N v,\pi_Nw)}$, $W_{(v,w)}=W_{(\pi_Nv,\pi_Nw)}$, $U_v = U_{\pi_N v}$, where $\pi_N:G_N\To G_1$ is the covering projection. Then we get a quantum graph $\cQ_N$ which we call the $N$-lift of $\cQ_1$.

Many studies have been recently devoted to $N$-lifts of combinatorial graphs, as they are a natural model of random irregular graphs. In particular, we know that they are typically connected and most of their points have a large injectivity radius---see \cite[\S~4.2]{Bornew}, \cite[Lemma 9]{BDGHT}. More precisely, the following condition \textbf{(BST)} holds generically:

\smallskip

\textbf{(BST)} For any $r$, $\frac{\#\{v\in V_N:\rho_{G_N}(v)<r\}}{N} \To 0$, where $\rho_{G_N}(v)$ is the injectivity radius at $v\in V_N$, i.e. the largest $\rho$ such that $B_{G_N}(x,\rho)$ is a tree.

\smallskip

It follows that the combinatorial $N$-lifts $(G_N)$ generically converge to the universal cover $\mathbb{T}$ of $G_1$; more precisely to $\rho = \frac{1}{|G_1|}\sum_{v\in G_1} \delta_{[\T,\widetilde{v}]}$. In fact, let us prove along the same lines that if $(G_N)$ is a sequence of $N$-lifts satisfying \textbf{(BST)}, then $\cQ_N$ converge to the universal cover $\cT$ of the quantum graph $\cQ_1$, more precisely, to the measure
\begin{equation}\label{e:pforlifts}
\mathbb{P} = \frac{1}{\sum_{b\in \mathcal{B}_1} L_b}\sum_{b\in \mathcal{B}_1} \int_0^{L_b} \delta_{[\mathbb{T},\tilde{L}^1,\tilde{W}^1,\tilde{U}^1,(\tilde{b},x_0)]}\,\dd x_0\,,
\end{equation}
where $(L^1,W^1,U^1)$ is the data on the base graph $G_1$.

Consider the set $\mathscr{A}$ of continuous functions on $\mathbf{Q}_{\ast}^{D,m,M}$, which ``depend only on a finite-size neighborhood of the root''. That is, let $\mathscr{A} = \bigcup_{r\in \N} \mathscr{A}_r$, where
\begin{multline}\label{e:ar}
\mathscr{A}_r = \big\{ f:\mathbf{Q}_{\ast}^{D,m,M}\to \R : f \text{ is bounded, continuous and } \\
\qquad\qquad\qquad\quad  f([G,L,W,U,\mathbf{x}]) = f([G',L',W',U',\mathbf{x'}]) \\
 \text{ if } [B_{G^{\mathbf{x}}}(v_{\mathbf{x}},  r),L_{|B_{G^{\mathbf{x}}}(v_{\mathbf{x}},  r)}^{\mathbf{x}},U_{|B_{G^{\mathbf{x}}}(v_{\mathbf{x}},  r)}^{\mathbf{x}},\mathbf{x}]=[B_{G^{\mathbf{x'}}}(v_{\mathbf{x'}},  r),L_{|B_{G^{\mathbf{x'}}}(v_{\mathbf{x'}},  r)}'^{\mathbf{x'}},U_{|B_{G^{\mathbf{x'}}}(v_{\mathbf{x'}},  r)}'^{\mathbf{x'}},\mathbf{x'}] \big\} \, .
\end{multline}
To prove the asserted convergence to $\mathbb{P}$ in \eqref{e:pforlifts}, it suffices to show that for any $f\in \mathscr{A}$,
\begin{equation}\label{e:bsco}
\frac{1}{2\cL(\cQ_N)}\sum_{b\in \mathcal{B}_N}\int_0^{L_b} f(\cQ_N,\mathbf{x_0})\,\dd \mathbf{x_0} \To \frac{1}{\sum_{b\in \mathcal{B}_1} L_b}\sum_{b\in \mathcal{B}_1} \int_0^{L_b} f[\mathbb{T},L^1,W^1,U^1,\mathbf{x_0}]\,\dd \mathbf{x_0}
\end{equation}
(instead of proving it for all $f\in C(\mathbf{Q}_{\ast}^{D,m,M})$). This follows from the Stone-Weierstrass and Prokhorov theorems. In detail:

\begin{enumerate}[(a)]
\item $\mathscr{A}\cap C_0(\mathbf{Q}_{\ast}^{D,m,M})$ is an algebra of continuous functions. It separates points. In fact, if $[\cQ_1,\mathbf{x_1}]\neq [\cQ_0,\mathbf{x_0}]$, we may find $r\in \N$ with $d([\cQ_1,\mathbf{x_1}],[\cQ_0,\mathbf{x_0}]) > \frac{1}{1+r}$. Let $\phi \in C_c^\infty([0, \infty))$ be such that $\phi=1$ on $[0,\frac{1}{r+2}]$ and $\phi(t)=0$ if $x\geq \frac{1}{r+1}$. We set $\chi([\cQ,\mathbf{x}] ):= \phi \left(d([\cQ,\mathbf{x}], [\cQ_0,\mathbf{x_0}])\right)$. 

Up to taking $r$ larger, we may suppose that the ball centred at $[\cQ_0,\mathbf{x_0}]$ and of radius $\frac{1}{r+1}$ is included in $\mathbf{Q}_{\ast}^{D,m,M,\delta}$ for some $\delta>0$, so that $\chi \in C_0(\mathbf{Q}_{\ast}^{D,m,M})$, and $\chi ([\cQ_1,\mathbf{x_1}] ) \neq \chi ([\cQ_0,\mathbf{x}_0] )$. Furthermore, by definition of the distance $d$, we clearly have $\chi \in \mathscr{A}_{r+1}$. 
\item Now Lemma~\ref{lem:QCompact} implies that $\mathbf{Q}_\ast^{D,m,M}$ is locally compact, it also implies that $\nu_{\cQ_N}$ is tight. Using the previous point along with \cite[Corollary V.8.3]{Conway}, we see that $\overline{\mathscr{A}\cap C_0(\mathbf{Q}_{\ast}^{D,m,M})} = C_0(\mathbf{Q}_\ast^{D,m,M})$. Combining \cite[Theorems 13.6, 13.11, 13.34]{Klenke} and an $\varepsilon$-argument as in \cite[Corollary 15.3]{Klenke} thus shows it suffices to prove \eqref{e:bsco} for $f\in \mathscr{A}\cap C_0(\mathbf{Q}_{\ast}^{D,m,M})$.
\end{enumerate}

So let $f\in \mathscr{A}_r$ and let us show that \eqref{e:bsco} holds for $f$. We have
\[
\sum_{b\in \mathcal{B}_N}\int_0^{L_b} f[\cQ_N,\mathbf{x_0}] = \sum_{v\in V_N} \sum_{w\sim v} \int_0^{L_{b_{vw}}}f[\cQ_N,(b_{vw},x)] = \sum_{v\in V_1} \sum_{w\in \pi_N^{-1} v} \sum_{z\sim w} \int_0^{L_{b_{wz}}} f[\cQ_N,(b_{wz},x)]
\]
where $b_{vw}=(v,w)$. Suppose $\rho_{G_N}(w)\ge r$. Then $B_{G_N}(w,r)$ is a tree. If $w\in \pi_N^{-1}v$, then this tree is precisely $B_{\mathbb{T}}(v,r)$, by definition of the universal cover (and because $\mathbb{T}$ is also the universal cover of $G_N$). Since data is lifted naturally, it follows that $[B_{G_N}(w,r),L^N,W^N,U^N,(b_{wz},x)]=[B_{\mathbb{T}}(v,r),L^1,W^1,U^1,(b_{vu},x)]$, where $u=\pi_N z$. As $f\in \mathscr{A}_r$, we get $f[\cQ_N,(b_{wz},x)]=f[\cT,(b_{vu},x)]$. The $w$ with $\rho_{G_N}(w)<r$ are considered as errors. We thus get
\begin{multline*}
\left|\frac{1}{N}\sum_{b\in \mathcal{B}_N}\int_0^{L_b} f[\cQ_N,\mathbf{x_0}]\,\dd \mathbf{x_0} - \sum_{b\in \mathcal{B}_1} \int_0^{L_b} f[\mathbb{T},L^1,W^1,U^1,\mathbf{x_0}]\,\dd \mathbf{x_0}\right| \\
 \le 2D\overline{L}\|f\|_{\infty}\cdot \frac{\# \{w\in V_N:\rho_{G_N}(w)<r\}}{N}
\end{multline*}
which tends to zero by \textbf{(BST)}. Finally, $2\cL(\cQ_N) = \sum_{b\in \mathcal{B}_N} L_b^N = N \sum_{b\in \mathcal{B}_1} L_b^1$, since each edge has precisely $N$ lifts. We thus showed the limit is the measure $\mathbb{P}$ given in \eqref{e:pforlifts}.

The universal cover $\cT$ endowed with such lifted data is studied in \cite{AISW}, where it is shown that $\cT$ has bands of absolutely continuous spectra on which the Green's kernel exists and is continuous. It follows from Theorem~\ref{th:Empirical} that $\mu_{\cQ_N}(I) \To \frac{1}{\pi\sum_{b\in \mathcal{B}_1} L_b}\sum_{b\in \mathcal{B}_1} \int_I\int_0^{L_b} \Im G_{\lambda+\ii 0}(\mathbf{x_0},\mathbf{x_0})\,\dd\mathbf{x_0}\,\dd\lambda$ for $I$ in such band. In the very special case of a regular graph with a single $\delta$-data $(L,W,\alpha)$, the density $\Im G_{\lambda+\ii 0}(\mathbf{x_0},\mathbf{x_0})$ can be computed explicitly \cite{ISW} and is in fact independent of $\mathbf{x_0}$ when $W=\alpha=0$.
 
\subsection{Graphs with a random data}
Let $(G_N)$ be a sequence of finite $d$-regular graphs satisfying \textbf{(BST)}. Endow $G_N$ with i.i.d. edge lengths $(L_e^{\omega})$ and i.i.d. coupling constants $(\alpha_v^{\omega})$. More precisely, let $\Omega_N^{(1)}=\R^{E(G_N)}$, $\cP_N^{(1)} = \bigotimes_{e\in E(G_N)}\nu_1$ on $\Omega^{(1)}_N$, $\Omega_N^{(2)}=\R^{V(N)}$, $\cP_N^{(2)}=\bigotimes_{v\in V(N)} \nu_2$ on $\Omega_N^{(2)}$, $\Omega_N=\Omega_N^{(1)}\times\Omega_N^{(2)}$ with $\cP_N=\cP_N^{(1)}\otimes\cP_N^{(2)}$. Let $\widetilde{\Omega} = \prod_{N\in\N}\Omega_N$ and $\cP$ a probability on $\widetilde{\Omega}$ having marginal $\cP_N$ on $\Omega_N$ (e.g. $\cP=\bigotimes \cP_N$). Then we obtain quantum graphs $\cQ_{\omega_N}=(G_N,L^{\omega_N^{(1)}},U^{\omega_N^{(2)}})$ for each $\omega_N=(\omega_N^{(1)},\omega_N^{(2)})$. Here $U^{\omega_N^{(2)}}$ is the unitary matrix encoding the boundary conditions $f_b(0)=f_{b'}(0)=:f(v)$ if $o(b)=o(b')=v$ and $\sum_{b,o(b)=v} f_b'(0)= \alpha_v^{\omega_N^{(2)}}f(v)$.

Next, let $\T$ be the $d$-regular tree and consider the probability space $(\Omega,\mathbf{P})$, where $\Omega=\R^{E(\T)}\times\R^{V(\T)}$ and $\mathbf{P} = \bigotimes_{e\in E(\T)}\nu_1 \bigotimes_{v\in V(\T)}\nu_2$. We obtain an analogous quantum tree $(\mathbb{T},L^{\omega^{(1)}},U^{\omega^{(2)}})$.

We assume the measures $\nu_1,\nu_2$ have compact support:
\[
\supp \nu_1 \subseteq [L_0,L_1]\,, \qquad \supp \nu_2 \subseteq [-A,A] \,,
\]
for some $L_0,L_1,A>0$.

Then for $\cP$-a.e. $(\omega_N)_{N\in\N}\in \widetilde{\Omega}$, the sequence $\cQ_{\omega_N}$ has a local weak limit $\mathbb{P}$ which is concentrated on $\{[\mathbb{T},L^{\omega^{(1)}},U^{\omega^{(2)}},\mathbf{x}_0]:(\omega^{(1)},\omega^{(2)})\in \Omega\}$. The measure $\mathbb{P}$ acts as follows:
\[
\mathbb{E}_{\mathbb{P}}(f) = \frac{1}{\mathbf{E}(L_{b_o}^{\omega^{(1)}})}\mathbf{E}\bigg(\int_0^{L_{b_o}^{\omega^{(1)}}}f([\T,L^{\omega^{(1)}},U^{\omega^{(2)}},(b_o,x_0)])\,\dd x_0\bigg) ,
\]
where $b_o\in \mathcal{B}(\T)$ is fixed and arbitrary.

 If $G_1$ is a finite, possibly irregular graph, $G_N$ is an $N$-lift of $G_1$ satisfying \textbf{(BST)}, and we endow $G_N$ with i.i.d. lengths and couplings, then $\cQ_{\omega_N}$ converge to the probability measure $\mathbb{P}$ concentrated on the universal cover $\T=\widetilde{G}_1$ of $G_1$, defined by
\[
\mathbb{E}_{\mathbb{P}}(f) = \frac{1}{\mathbf{E}(L_e^{\omega^{(1)}})|\mathcal{B}_1|}\sum_{b\in \mathcal{B}_1}\mathbf{E}\bigg( \int_0^{L_b^{\omega^{(1)}}} f[\widetilde{G}_1,L^{\omega^{(1)}},U^{\omega^{(2)}},(\tilde{b},x_0)]\,\dd x_0\bigg)
\]

We only prove the first claim. To lighten the notation, we henceforth remove the superscripts in $\omega_N$ and $\omega$.

Consider the algebra $\mathscr{A}$  introduced in \eqref{e:ar}. Arguing as before,
it suffices to show that there exists $\Omega_0 \subseteq \widetilde{\Omega}$ with $\mathcal{P}(\Omega_0)=1$ such that for any $(\omega_N) \in \Omega_0$ and any $f\in \mathscr{A}\cap C_0(\mathbf{Q}_{\ast}^{D,m,M})$,
\begin{equation}\label{eq:lln}
\lim_{N\to\infty} \frac{1}{2\cL(\cQ_{\omega_N})}\sum_{b\in \mathcal{B}_N}\int_0^{L_b^{\omega_N}}f(\cQ_{\omega_N},\mathbf{x}_0)\,\dd\mathbf{x}_0 = \frac{\mathbf{E}(\int_0^{L_{b_o}^{\omega}}f[\mathbb{T},L^{\omega},U^{\omega},(b_o,x_0)]\dd x_0)}{\mathbf{E}(L_{b_o}^{\omega})}\,.
\end{equation}
We first note that $\frac{1}{|\mathcal{B}_N|}\sum_{b\in \mathcal{B}_N} L_b^{\omega_N} \To \mathbf{E}(L_e^{\omega})$ by the law of large numbers. Next, let $\mathcal{E}$ the expectation w.r.t. $\mathcal{P}$. Given $f\in \mathscr{A}_r$, let
\[
Y_b=Y_b^N = \int_0^{L_b^{\omega_N}}f([G_N,L^{\omega_{N}},U^{\omega_N},(b,x)])\dd x - \cE\bigg[\int_0^{L_b^{\omega_N}}f([G_N,L^{\omega_{N}},U^{\omega_N},(b,x)])\dd x\bigg]
\]
and $S_N = \frac{1}{|\mathcal{B}_N|}\sum_{b\in \mathcal{B}_N} Y_b$

Then $Y_b^N$ only depends on $(\omega_e,\omega_v)_{e\in E_{G_N}(o(b),r),v\in B_{G_N}(o(b),r+1)}$, since $f([G_N,L^{\omega_N},U^{\omega_N},(b,x)]) = f([G_N,\tilde{L}^{\omega_N},\tilde{U}^{\omega_N},(b,x)])$ if $L^{\omega_N} = \tilde{L}^{\omega_N}$ and $U^{\omega_N}=\tilde{U}^{\omega_N}$ on the graph $B_{G_N}(o(b),r+1)$. Hence, $Y_b^N$ and $Y_{b'}^N$ are independent if $d_{G_N}(o(b),o(b'))>2r+2$. Moreover, each $Y_b^N$ is bounded by $2L_1\|f\|_\infty$. Using such independence at a distance, one can adapt the proof of the strong law of large numbers using $4$-th moments to conclude that $S_N \To 0$ $\cP$-a.s.,
i.e. for $(\omega_N)\in \Omega_f$ with $\cP(\Omega_f)=1$; see \cite{AS3}. Since $\mathscr{A}\cap C_0(\mathbf{Q}_{\ast}^{D,m,M})$ has a countable dense subset $\{f_j\}$, we let $\Omega_0=\bigcap_j \Omega_{f_j}$ and get $\cP(\Omega_0)=1$.

Now if $f\in \mathscr{A}$, say $f\in \mathscr{A}_r$, we have
\begin{multline}    \label{eq:replace}
 \left| \frac{1}{|\mathcal{B}_N|}\sum_{b\in \mathcal{B}_N} \int_0^{L_b^{\omega_N}}f([G_N,L^{\omega_N},U^{\omega_N},(b,x)])\dd x - \mathbf{E}\bigg[\int_0^{L_{b_o}^{\omega}}f([\mathbb{T}_q,L^{\omega},U^{\omega},(b_o,x)])\dd x\bigg]\right| \\
 \qquad \le |S_N| + \bigg|\frac{1}{|\mathcal{B}_N|}\sum_{b\in \mathcal{B}_N} \mathcal{E}_N\bigg[\int_0^{L_b^{\omega_N}}f([G_N,L^{\omega_N},U^{\omega_N},(b,x)])\dd x\bigg] \\
- \mathbf{E}\bigg[\int_0^{L_{b_o}^{\omega}}f([\mathbb{T}_q,L^{\omega},U^{\omega},(b_o,x)])\dd x\bigg]\bigg| \, .
\end{multline}
Fix $o\in \mathbb{T}_q$ and let $b_o$ with $o(b_o)=o$. If $\rho_{G_N}(v)\ge r$ and $o(b_v)=v$, there is a graph isomorphism $\phi_{b_v}:B_{G_N}(v,r) \to B_{\mathbb{T}_q}(o,r)$ with $\phi_{b_v}(b_v)=b_o$. We now copy the data to $B_{\mathbb{T}_q}(o,r)$. More precisely, we define $L^{\omega}_{b_v} = L^{\omega_N}\circ\phi_{b_v}^{-1}$ and $U^{\omega}_{b_v}=U^{\omega_N}\circ \phi_{b_v}^{-1}$ to get $[B_{G_N}(v,r),L^{\omega_N},U^{\omega_N},(b_v,x)] = [B_{\mathbb{T}_q}(o,r),L_{b_v}^{\,\omega},U_{b_v}^{\omega},(b_o,x)]$. Thus, $f([G_N,L^{\omega_N},U^{\omega_N},(b_v,x)]) = f([\mathbb{T}_q,L_{b_v}^{\omega},U_{b_v}^{\omega},(b_o,x)])$, and by definition $L_{b_v}^{\omega}(b_o) = L_{b_v}^{\omega_N}$. As the transformation $T:\R^{E_{\mathbb{T}_q}(o,r)}\times\R^{B_{\mathbb{T}_q}(o,r)} \to \R^{E_{\mathbb{T}_q}(o,r)}\times\R^{B_{\mathbb{T}_q}(o,r)}$ mapping $(\omega_e,\omega_w)\mapsto(\omega_{\phi_{b_v}^{-1}(e)},\omega_{\phi_{b_v}^{-1}(w)})$ preserves the measure $\mathbf{P}$, it follows that
\[
\mathcal{E}_N\bigg[\int_0^{L_{b_v}^{\omega_N}}f([G_N,L^{\omega_N},U^{\omega_N},(b_v,x)])\dd x\bigg] = \mathbf{E}\bigg[\int_0^{L_{b_o}^{\omega}}f([\mathbb{T}_q,L^{\omega},U^{\omega},(b_o,x)])\dd x\bigg] .
\]
Hence, \eqref{eq:replace} may be bounded by $|S_N| + \frac{\# \{v:\rho_{G_N}(v)<r\}}{|\mathcal{B}_N|}(2L_1\|f\|_{\infty})$.
Taking $N\to \infty$, it follows by \textbf{(BST)} that if $(\omega_N)\in \Omega_0$, then \eqref{eq:lln} is true for any $f\in \{f_j\}$, the dense subset of $\mathscr{A}\cap C_0(\mathbf{Q}_{\ast}^{D,m,M})$. Arguing as in \cite[Corollary 15.3]{Klenke}, the proof is complete.

\appendix

\section{Properties of the Green's kernel}\label{sec:app1}
Let $\cQ$ be a quantum graph satisfying Hypothesis 1 and
$z\in \C\setminus \R$.  Recall that we defined
$G_z(\mathbf{x},\mathbf{y})$ to be the Schwartz kernel of the
resolvent operator,
\begin{displaymath}
  \int_{\mathcal{G}} G_z(\mathbf{x},\mathbf{y}) f(\mathbf{y})\,\dd\mathbf{y}
  = (H_{\mathcal{Q}} - z)^{-1} f(\mathbf{x}),\qquad f\in L^2(\mathcal{G}).
\end{displaymath}
The Schwartz kernel theorem only guarantees that $G_z(\cdot,\cdot)$
exists as a distribution over $\cG\times\cG$.  We aim to show that it can be represented
as a continuous function in both variables.  As a first step we have
the following.

\begin{lemme}\label{lem:greenskernelL2}
 For each $\mathbf{x_0}\in\mathcal{G}$ set $g^{z,\mathbf{x_0}}(\mathbf{y})
  =G_z(\mathbf{x_0},\mathbf{y})$.  Then $g^{z,\mathbf{x_0}}\in
L^2(\cG).$

Furthermore, if $D\in \N, 0<m<M$, there exists $C_1(D,m,M)$ such that for any
  $[\cQ,\mathbf{x_0}]\in \mathbf{Q}_*^{D,m,M}$ and any $z\in \C$ with
  $\Im z > C_1$, we have 
\begin{equation}\label{eq:13}
\left\|g^{z,\mathbf{x_0}} \right\|_{L^2(\mathcal{G})}\le C_2(z;D,m,M) \,.
\end{equation}
\end{lemme}
\begin{proof}
  Let $f\in L^2(\cG)$ and define $\phi=\phi_f = (H_{\cQ}-z)^{-1}f$, so that, for almost every $\mathbf{x_0}\in \cG$, we have
  \begin{equation}\label{eq:PointwiseResolv}
    \phi_f(\mathbf{x_0}) = \int_{\mathcal{G}} G_z(\mathbf{x_0},\mathbf{y})
    f(\mathbf{y})\, \dd\mathbf{y}.
  \end{equation}
 
Since $H_\mathcal{Q}$ is
  self-adjoint we have $\|\phi_f\|_{L^2(\cG)} \leq |\Im z|^{-1}
  \|f\|_{L^2(\cG)}$.
  Furthermore, $(H_{\cQ}-z)^{-1}:L^2(\cG)\To D(H_{\cQ}) = \mathscr{H}^\cQ$. In particular, $\phi_b\in H^2(0,L_b)$ for each $b$, so \eqref{eq:PointwiseResolv} actually holds for every $\mathbf{x_0}\in \mathcal{G}$. Moreover,
\[
    \phi_b'' = -(H_\cQ-z)\phi_b + (W_b-z)\phi_b = -f_b+ (W_b-z)(H_\cQ-z)^{-1}f_b,
\]
  giving
  \begin{equation}
    \label{eq:6}
    \|\phi_f''\|_{L^2(\cG)} \leq \|f\|_{L^2(\cG)} +
    \frac{\|W\|_\infty + |z|}{|\Im z|}\| f\|_{L^2(\cG)}\,.
\end{equation}

On the other hand, for $\psi\in H^2(0,L)$, using $\psi(x)=\psi(0)+\int_0^x\psi'(t)\,\dd t$, \cite[Lemma 8]{Kuch} and Cauchy-Schwarz, we have $|\psi(x)|\le \sqrt{\frac{2}{L}}\|\psi\|_2+2\sqrt{L}\|\psi'\|_2$. But by \cite[Lemma p.23]{Maz} and Cauchy-Schwarz,
\begin{equation}\label{e:Mazya}
|\psi'(t)|\le \sqrt{L}\,\|\psi''\|_2+\frac{4}{L^{3/2}}\|\psi\|_2\,.
\end{equation}

Taking $C=\max(\sqrt{\frac{2}{L_b}}+\frac{8}{\sqrt{L_b}},2L_b^{3/2})$, we thus get for $\phi\in\mathscr{H}^\cQ$,
  \begin{equation}
    \label{eq:5}
    |\phi(\mathbf{x_0})| \leq C \left( \|\phi \|_{L^2(\cG)} +
    \| \phi''\|_{L^2(\cG)}\right).
\end{equation}

Combining  \eqref{eq:5} and \eqref{eq:6}, we see that for each $\mathbf{x_0}\in \cG$, the mapping
$f\mapsto \phi_f(\mathbf{x_0})$ is a bounded linear functional.  By the
Riesz representation theorem, there is hence a unique $h^{z,\mathbf{x_0}}\in
L^2(\cG)$ such that
\begin{equation}\label{eq:12}
  \phi_f(\mathbf{x_0})= \langle h^{z,\mathbf{x_0}}, f \rangle_{L^2(\cG)}= \int_{\cG} \overline{h^{z,\mathbf{x_0}}(\mathbf{y})} f(\mathbf{y})\,\dd\mathbf{y}
\end{equation}
for all $f\in L^2(\cG)$. Recalling \eqref{eq:PointwiseResolv}, this implies $\overline{h^{z,\mathbf{x_0}}} = g^{z,\mathbf{x_0}}$ a.e. In particular, $g^{z,\mathbf{x_0}}\in L^2(\cG)$.

Finally, for any $f\in L^2(\cG)$, we have
\begin{align*}
|\langle f,  g^{z,\mathbf{x_0}}\rangle_{L^2(\cG)}| = |\phi_f(\mathbf{x_0})| \leq C \left( \|f\|_{L^2(\cG)} +
    \frac{|z|+ \|W\|_\infty}{|\Im z|}\| f\|_{L^2(\cG)} \right),
\end{align*}
so that $\| g^{z,\mathbf{x_0}}\|_{L^2(\cG)} \leq C \left( 1 +    \frac{|z|+ \|W\|_\infty}{|\Im z|} \right)$, which gives us the second point.
\end{proof}

\begin{remarque}\label{rem:L2Cont} 
Using \eqref{e:Mazya} and \eqref{eq:6}, we also get for $\mathbf{x}=(b_0,x)$ near $\mathbf{x_0}=(b_0,x_0)$,
\[
|\langle f,g^{z,\mathbf{x}}-g^{z,\mathbf{x_0}}\rangle| = |\phi_f(\mathbf{x})-\phi_f(\mathbf{x_0})|\le \tilde{C}\cdot |x-x_0|\Big(1+\frac{|z|+ \|W\|_\infty}{|\Im z|}\Big)\|f\|_2
\]
for any $f\in L^2(\cG)$.  In particular, taking $f=\frac{g^{z,\mathbf{x}}-g^{z,\mathbf{x_0}}}{\|g^{z,\mathbf{x}}-g^{z,\mathbf{x_0}}\|_{L^2}}$, we see that the map $\mathcal{G}\ni \mathbf{x_0}\mapsto g^{z,\mathbf{x_0}}\in L^2(\mathcal{G})$ is continuous.
\end{remarque}

Beware that, in the previous remark (as well as in the second statement of the following lemma), $\mathbf{x_0}$ (and $\mathbf{x_1}$) belong to $\mathcal{G}$ which does not contain the vertices. Hence, we do not claim that the Green's function is continuous at the vertices: its behaviour at the vertices is dictated by the boundary conditions, as we will see in Lemma \ref{lem:GreenBCs}.

Now we upgrade $L^2$-regularity of $g^{z,\mathbf{x_0}}$ to continuity
using the elliptic regularity theorem.
\begin{lemme} \label{lem:greenskernelcts}
  Let $\cQ$ be a quantum graph, and $z\in \C\setminus \R$.
\begin{enumerate}[\rm(i)]
\item For every $\mathbf{x_0}\in\cG$, $g^{z,\mathbf{x_0}}\in \mathscr{H}^0$, and  $[(H_{\cQ}-z)g^{z,\mathbf{x_0}}](\mathbf{y_0})=0$ for
  $\mathbf{y_0}\neq \mathbf{x_0}$.
\item The kernel $G_z(\mathbf{x_0},\mathbf{x_1})$ is continuous in each argument. Moreover, for any $\mathbf{x_0},\mathbf{x_1} \in\cG$, we have 
\begin{equation}\label{eq:GreenSym}
G_z(\mathbf{x_0}, \mathbf{x_1}) = \overline{G_{\overline{z}}(\mathbf{x_1}, \mathbf{x_0}) } = G_z(\mathbf{x_1}, \mathbf{x_0}).
\end{equation}
\end{enumerate}
\end{lemme}
\begin{proof}
  \begin{enumerate}[\rm(i)]
  \item  We have $\int_\mathcal{G}
  g^{z,\mathbf{x_0}}(\mathbf{y})\, (H_{\cQ}-z)f(\mathbf{y})\,\dd\mathbf{y} =
  \{(H_{\cQ}-z)^{-1}[(H_{\cQ}-z)f]\}(\mathbf{x_0})=f(\mathbf{x_0})$
  for any $f\in\mathscr{H}^{\cQ}$. This means that
  $(H_{\cQ} - z)g^{z,\mathbf{x_0}} =\delta_{\mathbf{x_0}}$ in the
  sense of distributions. Thus, $g^{z,\mathbf{x_0}}$ is an $L^2$
  function which solves the Schr\"odinger equation away from
  $\mathbf{x_0}$. By elliptic regularity (e.g. \cite[Theorem
  6.1]{Ag}), it follows that $g^{z,\mathbf{x_0}}\in \mathscr{H}^0$. In particular, $\forall \mathbf{x_0}\in \cG$,  the map
  \begin{equation}\label{eq:FirstContinuity}
  \cG \ni \mathbf{x_1} \mapsto  G_z(\mathbf{x_0}, \mathbf{x_1}) \text{ is continuous}.
  \end{equation}
\item For any $f,g\in L^2(\cG)$, we have $\langle f,(H_{\cQ}-z)^{-1}g\rangle = \langle (H_{\cQ}-\overline{z})^{-1}f,g\rangle$. By density of $L^2(\cG)\otimes L^2(\cG)$ in $L^2(\cG\times\cG)$, this implies that as distributions over $\cG\times \cG$, we have $G_z (\mathbf{x_0}, \mathbf{x_1})= \overline{G_{\overline{z}} (\mathbf{x_1}, \mathbf{x_0})}$.

Similarly, one may define a pairing $(f,g) = \int_{\cG} f(x)g(x)\,\dd x$ and define the \emph{transpose} $A^{\intercal}$ of an operator $A$ to be the unique operator satisfying $(A^{\intercal}f,g)=(f,A g)$ (i.e. $A^{\intercal}f = \overline{A^{\ast}(\overline{f})}$). Then $(AB)^{\intercal}=B^{\intercal}A^{\intercal}$, it follows that the inverse of a ``self-transpose'' operator (satisfying $A^{\intercal}=A$) is self-transpose. Since $(H_{\cQ}-z)^{\intercal}=H_{\cQ}-z$, we get $(f,(H_{\cQ}-z)^{-1}g)=( (H_{\cQ}-z)^{-1}f,g)$, from which it follows as before that $G_z (\mathbf{x_0}, \mathbf{x_1}) = G_z (\mathbf{x_1}, \mathbf{x_0})$ as distributions over $\cG\times\cG$.

But we know from \eqref{eq:FirstContinuity} that the distribution $G_z(\mathbf{x_0},\mathbf{x_1})$ is a continuous functions in $\mathbf{x_1}$. It follows that the distribution $G_z (\mathbf{x_0}, \mathbf{x_1}) = G_z (\mathbf{x_1}, \mathbf{x_0})$ is also a continuous function of $\mathbf{x_0}$, and that \eqref{eq:GreenSym} actually holds pointwise, for every $\mathbf{x_0},\mathbf{x_1}\in \cG$. \qedhere 
\end{enumerate}
\end{proof}

In the previous lemma, we have seen that $g^{z,\mathbf{x_0}}$ satisfies $(H_{\cQ}-z) g^{z,\mathbf{x_0}}=0$ away from $\mathbf{x_0}$. Let us now see which boundary condition it satisfies, at vertices of $V$ as well as at $\mathbf{x_0}$.

\begin{lemme} \label{lem:GreenBCs}
Let $[\cQ,\mathbf{x_0}]$ be a quantum graph, $v_0\in V^{\mathbf{x_0}}$ and $z\in \C\setminus \R$.
\begin{enumerate}[\rm(i)]
\item If the boundary conditions on $v_0$ are Kirchhoff, then $g^{z,v_0}=G^z(v_0,\cdot)$ is well-defined, meaning $g^{z,(b,0)}=g^{z,(b',0)}$ if $o(b)=o(b')=v_0$. Moreover,
\begin{equation}\label{eq:LimitCondGreen}
\begin{cases} g^{z,v_0}(o(b)) = g^{z,v_0}(o(b')) & \text{if } o(b)=o(b')=v_0\,,\\
\sum\limits_{b,o(b)=v_0} \frac{\dd}{\dd x_b}g^{z,v_0}(x_b)|_{x_b=0} = -1 \,.
\end{cases}
\end{equation}
\item At any vertex $v\neq v_0$, $g^{z,v_0}$ satisfies the boundary conditions \eqref{e:HQ}.
\end{enumerate}
\end{lemme}
\begin{proof}
(i) For any $\phi\in L^2$, $\langle \overline{g^{z,o(b)}},\phi\rangle_{L^2(\cG)} = \{(H_{\cQ}-z)^{-1}\phi\}(o(b)) = \{(H-z)^{-1}\phi\}(v_0)$ for all $o(b)=v_0$ since $(H_{\cQ}-z)^{-1}\phi\in D(H_{\cQ})$. Hence, $g^{z,v_0}$ is well-defined. Next, given a vertex $v$, let $f_v$ be a smooth function compactly supported on the edges adjacent to $v$. Integrating by parts, we get
\[
f_v(v_0) = \langle \overline{g^{z,v_0}}, (H_{\cQ}-z)f_v\rangle_{L^2(\cG)} = \sum_{b,o(b)=v} \left[g^{z,v_0}(x_b)f'_v(x_b)|_{x_b=0} - f_v(x_b) (g^{z,v_0})'(x_b)|_{x_b=0}\right] .
\]
For $v=v_0$, choose $f_{v_0}\equiv 1$ near $v_0$, then $\sum g'(o_b)=-1$. Choosing $f_{v_0}$ such that $f_{v_0}(v_0)=0$, $f_{v_0}=0$ except on two edges and $f'_{v_0}(x_b)|_{x_b=0}\in \{1,-1\}$ otherwise, shows that $g^{z,v_0}(o_b)$ is the same for all $b$.

(ii) For $v\neq v_0$, we get $\langle \overline{G^{z,v_0}},F'_v\rangle_{\C^{d(v)}}-\langle \overline{(G^{z,v_0})'},F_v\rangle_{\C^{d(v)}} = 0$, where $F_v,G^{z,v_0}$ are defined as in \eqref{eq:defF}. As in \cite[\S 1.4]{BeKu13}, choose $f$ such that $F_v = -A_2(v)^\ast E$ and $F'_v = A_1^\ast(v)E$, where $E$ is arbitrary in $\C^{d(v)}$, then $A_1(v)F_v+ A_2(v)F'_v=0$ (since $A_1A_2^*=A_2A_1^*$), and we get $\langle A_1\overline{G}+A_2\overline{G'},E\rangle=0$ for any $E$, so $\overline{g^{z,v_0}} = g^{\overline{z},v_0}$ satisfies \eqref{e:HQ} at $v$ (we used \eqref{eq:GreenSym}). Replacing $z$ by $\overline{z}$ proves (iii).
\end{proof}

The following lemma controls the Green function and its derivative when $z$ stays away from the real axis:

\begin{lemme}\label{lem:BoundGreenImagin}
Let $Z\subset \C\setminus \R$ be a bounded set, and let $M>m>0$, $D\in \N$. There exists $C=C(Z, D,M,m)>0$ such that for all $z\in Z$, for all $[\cQ, \mathbf{x_0}]\in \mathbf{Q}_*^{D,m,M}$, with $\mathbf{x_0}=(b_0,x_0)\in \cG$ and all $\mathbf{y_0}= (b_0',y_0)$, we have  
\begin{equation}\label{eq:BoundGreenImagin}
\begin{aligned}
|G_z (\mathbf{x_0}, \mathbf{y_0})|&\leq \frac{C}{|\Im z|}\\
\Big{|}\frac{\partial}{\partial x}\Big{|}_{x=x_0} G_z\big{(} (b_0,x),\mathbf{y_0}\big{)} \Big{|}&\leq \frac{C}{|\Im z|}\\
\Big{|}\frac{\partial}{\partial y} \Big{|}_{y=y_0} G_z \big{(}\mathbf{x_0}, (b'_0,y)\big{)}\Big{|}&\leq \frac{C}{|\Im z|}\, .
\end{aligned}
\end{equation}
\end{lemme}

Before proving the lemma, we introduce a few notations. Let $b\in \mathcal{B}$, and $z\in \C$. We define $C_b^z$ and $S_b^z$ on $[0,L_b]$ as in \eqref{e:cs}. 
Next, let, for $0<\zeta\leq L_b$,
\begin{equation*}
\begin{aligned}
\Sigma_1(z;b,\zeta)&:= \int_0^{\zeta} |S_b^z(x)|^2 \,\dd x\\
\Sigma_2(z;b,\zeta)&:= \int_0^{\zeta} |C_b^z(x)|^2 \,\dd x\\
\Sigma_3(z;b,\zeta)&:= \int_0^{\zeta} C_b^z(x) \overline{S_b^z(x)} \,\dd x.
\end{aligned}
\end{equation*}

For each $z\in \C$ and each $b\in \mathcal{B}$, we denote by 
$\cS^{b,\zeta}_{z}$ and $\cC^{b,\zeta}_{z}$ the functions 
taking value $S_b^z(x)$ and $C_b^z(x)$ on segment $(0,\zeta)\ni x$ of the 
edge $b$ and zero everywhere else.  We set the values on $\hat{b}$ 
in such a way to ensure that $\cS^{b,\zeta}_{z}, \cC^{b,\zeta}_{z}\in
\mathscr{H}$.

If $f\in \mathscr{H}^0$ satisfies \eqref{eq:Eigenedge}, then given $b$ and $x\in [0,L_{b}]$, we have
\[
f(b,x) =  C_b^z(x) f(b,0) + S_b^z(x) \frac{\dd}{\dd x}\Big|_{x=0} f(b,x),
\]
so that
\[
\begin{aligned}
\langle \cS_{z}^{b,\zeta}, f \rangle &=f(b,0)  \Sigma_3(z;b,\zeta) +  
\frac{\dd}{\dd x}\Big{|}_{x=0} f(b,x) \Sigma_1(z;b,\zeta)\\
\langle  \cC_z^{b,\zeta}, f \rangle &=f(b,0)  \Sigma_2(z;b,\zeta) + 
\frac{\dd}{\dd x}\Big{|}_{x=0} f(b,x) \overline{\Sigma_3(z;b,\zeta)}.
\end{aligned}
\]
Solving these linear equations leads to a proof of the following lemma:
\begin{lemme} \label{lem:maxime}
Let $\Sigma_1, \Sigma_2, \Sigma_3$ be defined as above and let
\begin{align*}
Y_z^{b,\zeta} &:=  \frac{1}{\Sigma_1(z;b,\zeta)\Sigma_2(z;b,\zeta)- 
|\Sigma^2_3(z;b,\zeta)|} \Big{(} \Sigma_1(z;b,\zeta) \cC_{z}^{b,\zeta}-
\Sigma_3(z;b,\zeta) \cS_{z}^{b,\zeta} \Big{)}\\
Z_z^{b,\zeta}&:= \frac{1}{\Sigma_1(z;b,\zeta)\Sigma_2(z;b,\zeta)- 
|\Sigma^2_3(z;b,\zeta)|} \Big{(} \Sigma_2(z;b,\zeta) \cS_{z}^{b,\zeta}
-\overline{\Sigma_3(z;b,\zeta)} \cC_{z}^{b,\zeta} \Big{)}.
\end{align*}
Then for any $f\in\mathscr{H}^0$ satisfying \eqref{eq:Eigenedge} we have
\begin{equation}\label{e:Maxime}
f(b,0) = \langle Y_z^{b,\zeta}, f \rangle \qquad \text{and} \qquad 
\frac{\dd}{\dd x}\Big{|}_{x=0} f(b,x) = \langle Z_z^{b,\zeta}, f \rangle \,.
\end{equation}
\end{lemme}
In situations where $\zeta=L_b$ we will drop $\zeta$ from notations to
economise.

We may now proceed with the proof of Lemma \ref{lem:BoundGreenImagin}.

\begin{proof}[Proof of Lemma \ref{lem:BoundGreenImagin}]
Let $\mathbf{x_0}= (b_0,x_0),\mathbf{y_0}=(b_0',y_0)\in \cG$. Let $\mathbf{x_1},\mathbf{y_1}$ be the midpoints of $b_0,b_0'$, respectively. Let $b_1 := (v_1, o(b_0))$, $b_2:= (v_1', t(b_0') )$, where $v_1,v_1'$ are Kirchhoff vertices added at $\mathbf{x_1},\mathbf{y_1}$, respectively (Definition~\ref{def:adding}).

The map $(0, L_{b_2})\ni y \mapsto G_z \big((b_1, x), (b_2,y)\big)$ satisfies \eqref{eq:Eigenedge}, so that $\big((H_\cQ-z )^{-1} \overline{Y_z^{b_2}} \big) (b_1,x) = G_z \big( (b_1,x), \mathbf{y_1} \big)$. The map $(0, L_{b_1})\ni x \mapsto G_z \big((b_1, x),  \mathbf{y_1}\big)$ also satisfies \eqref{eq:Eigenedge}, so that 
\[
 \Big{\langle} Y_z^{b_1},  \big{(}H_{\cQ}-z\big{)}^{-1} \overline{Y_z^{b_2}} \Big{\rangle} = G_z (\mathbf{x_1}, \mathbf{y_1})= G_z \big{(}(b_0, L_{b_0}/2), (b_0',L_{b_0'}/2) \big{)}.
\]

Similarly, we have 
 \begin{align*}
 \Big{\langle} Z_z^{b_1},  \big(H_{\cQ}-z\big)^{-1} \overline{Y_z^{b_2}} \Big{\rangle} &=   \frac{\partial}{\partial x}\Big{|}_{x= \frac{L_{b_0}}{2}}  G_z \big((b_0,x), \mathbf{y_1} \big{)}\\
 \Big{\langle} Y_z^{b_1},  \big(H_{\cQ}-z\big)^{-1} \overline{Z_z^{b_2}} \Big{\rangle} &=   \frac{\partial}{\partial y}\Big{|}_{y= \frac{L_{b'_0}}{2}}  
G_z \big( \mathbf{x_1}, (b'_0,y) \big{)}\\
  \Big{\langle} Z_z^{b_1},  \big(H_{\cQ}-z\big)^{-1} \overline{Z_z^{b_2}} \Big{\rangle} &=   \frac{\partial^2}{\partial x \partial y}\Big|_{x= \frac{L_{b_0}}{2},y=\frac{L_{b_0'}}{2}}    G_z \big((b_0,x), (b_0',y) \big)\,.
 \end{align*}

As in Section~\ref{sec:proof}, let $b_{\mathbf{x_0}}$, $b_{\mathbf{y_0}}$ be the edges containing $\mathbf{x_0}, \mathbf{y_0}$, respectively, and let
$\mathsf{x_0}, \mathsf{y_0}$ be the positions of $\mathbf{x_0}$,
$\mathbf{y_0}$ on edges $b_{\mathbf{x_0}}$, $b_{\mathbf{y_0}}$. Then since
$y \mapsto G_z(\mathbf{x_1},(b'_0, y)$ satisfies
\eqref{eq:Eigenedge} we can write
\begin{equation}\label{e:secondexpan}
G_z(\mathbf{x_1},\mathbf{y_0}) = C_{b_{\mathbf{y_0}}}^z(\mathsf{y_0}) G_z(\mathbf{x_1},\mathbf{y_1}) + S_{b_{\mathbf{y_0}}}^z(\mathsf{y_0}) \frac{\partial}{\partial y}\Big{|}_{y= \frac{L_{b_0'}}{2}}  G_z(\mathbf{x_1},(b_0',y)) \,.
\end{equation}
Equation \eqref{e:secondexpan} remains true if $\mathbf{x}_1= (b_0, x)$ with $x\neq L_{b_0}/2$. Differentiating in $x$, we obtain
\begin{multline}\label{e:secondexpan2}
\frac{\partial}{\partial x}\Big|_{x=\frac{L_{b_0}}{2}} G_z\left((b_0, x),\mathbf{y_0}\right) \\= \frac{\partial}{\partial x}\Big|_{x=\frac{L_{b_0}}{2}} \bigg[ C_{b_{\mathbf{y_0}}}^z(\mathsf{y_0}) G_z((b_0,x),\mathbf{y_1}) + S_{b_{\mathbf{y_0}}}^z(\mathsf{y_0})\frac{\partial}{\partial y}\Big{|}_{y= \frac{L_{b_0'}}{2}}  G_z((b_0,x),(b_0',y)) \bigg]\,.
\end{multline}

In a similar way, as $x \mapsto G_z((b_0,x),\mathbf{y_0})$ satisfies
\eqref{eq:Eigenedge}, we can write
\begin{equation}\label{e:firstexpan2}
G_z(\mathbf{x_0},\mathbf{y_0}) = C_{b_{\mathbf{x_0}}}^z( \mathsf{x_0}) G_z(\mathbf{x_1},\mathbf{y_0}) + S_{b_{\mathbf{x_0}}}^z(\mathsf{x_0}) \frac{\dd}{\dd x }\Big{|}_{x =\frac{L_{b_0}}{2}}  G_z((b_0,x),\mathbf{y_0}).
\end{equation}
We deduce that
\begin{multline*}
G_z(\mathbf{x_0},\mathbf{y_0}) = C_{b_{\mathbf{x_0}}}^z(\mathsf{x_0}) \Big( C_{b_{\mathbf{y_0}}}^z(\mathsf{y_0}) G_z(\mathbf{x_1},\mathbf{y_1}) + S_{b_{\mathbf{y_0}}}^z(\mathsf{y_0}) \frac{\partial}{\partial y}\Big|_{y= L_{b_0'}/2}  G_z(\mathbf{x_1},(b_0,y))\Big)  \\
 + S_{b_{\mathbf{x_0}}}^z(\mathsf{x_0})\Big(C_{b_{\mathbf{y_0}}}^z(\mathsf{y_0}) \frac{\partial}{\partial x}\Big|_{x= L_{b_0}/2} G_z\big((b_0,x) , \mathbf{y_1}\big{)} \\
 \qquad\qquad\qquad\qquad + S_{b_{\mathbf{y_0}}}^z(\mathsf{y_0}) \frac{\partial^2}{\partial x\partial y}\Big|_{x= \frac{L_{b_0}}{2},y=\frac{L_{b_0'}}{2}}  G_z\big((b_0,x), (b_0',y)\big)\Big) \\
=  C_{b_{\mathbf{x_0}}}^z(\mathsf{x_0})C_{b_{\mathbf{y_0}}}^z(\mathsf{y_0}) \Big{\langle} Y_z^{b_1},  \big{(}H_{\cQ}-z\big{)}^{-1} \overline{Y_z^{b_2}} \Big{\rangle} + C_{b_{\mathbf{x_0}}}^z(\mathsf{x_0})S_{b_{\mathbf{y_0}}}^z(\mathsf{y_0}) \Big{\langle} Y_z^{b_1},  \big(H_{\cQ}-z\big)^{-1} \overline{Z_z^{b_2}} \Big{\rangle} \\
 + C_{b_{\mathbf{y_0}}}^z(\mathsf{y_0})S_{b_{\mathbf{x_0}}}^z(\mathsf{x_0})\Big{\langle} Z_z^{b_1},  \big(H_{\cQ}-z\big)^{-1} \overline{Y_z^{b_2}} \Big{\rangle} +  S_{b_{\mathbf{x_0}}}^z(\mathsf{x_0})S_{b_{\mathbf{y_0}}}^z(\mathsf{y_0}) \Big{\langle} Z_z^{b_1},  \big(H_{\cQ}-z\big)^{-1} \overline{Z_z^{b_2}} \Big{\rangle}\,.
\end{multline*}

Using \cite[p. 7]{PT87}, we see that for all $z\in Z$ and data obeying \eqref{eq:ConditionCompQG}, $\max(|C_b^z(x)|,|S_b^z(x)|) \le C$.

By the Cauchy-Schwarz inequality, $\Sigma_1(z;b)\Sigma_2(z;b)- |\Sigma^2_3(z;b)|>0$. Clearly, the lower bound may only depend on $L_b$, $W_b$ and $z$.
We deduce that there exists $C>0$ such that, for all $z\in \overline{Z}$, all data satisfying \eqref{eq:ConditionCompQG} and all $b\in \mathcal{B}$, we have $\| Y_z^b\|_{L^2}, \| Z_z^b\|_{L^2}\leq C$.

Since for all $z\in \C\setminus \R$, we have $\big{\|}\big{(} H_{\cQ}-z \big{)}^{-1}\big{\|}_{L^2\rightarrow L^2} \leq \frac{1}{|\Im z|}$, we deduce that
\begin{equation}
|G_z(\mathbf{x_0}, \mathbf{y_0})| \leq \frac{C_0}{|\Im z|},
\end{equation}
for some constant $C_0=C_0(Z,D,M,m)$.

Differentiating \eqref{e:firstexpan2}, we have
\[
\frac{\partial}{\partial x}\Big|_{x=x_0} G_z\big( (b_0,x),\mathbf{y_0}) = \big(C_{b_{\mathbf{x_0}}}^z\big)'(\mathsf{x_0}) G_z(\mathbf{x_1},\mathbf{y_0}) + \big(S_{b_{\mathbf{x_0}}}^z\big)'(\mathsf{x_0}) \frac{\partial}{\partial x}\Big|_{x=\frac{L_{b_0}}{2}}  G_z((b_0, x),\mathbf{y_0}),
\]
so combining with \eqref{e:secondexpan}, we deduce as before that
\[
\Big{|}\frac{\partial}{\partial x}\Big{|}_{x=x_0} G_z\big{(} (b_0,x),\mathbf{y_0}\big{)} \Big{|}\leq \frac{C_1}{|\Im z|}\,. 
\]
The last claim in \eqref{eq:BoundGreenImagin} follows from symmetry.
\end{proof}

In the paper we also use the following results.

\begin{lemme}\label{lemTrace}
Let $\cQ$ be a finite quantum graph. Then, for any $z\in \C\setminus \R$, the operator $(H_{\cQ}-z)^{-1}$ is trace-class, and we have
\[
\mathrm{Tr} \big{[} (H_{\cQ}-z)^{-1}\big{]} = \int_{\cG} G_z(\mathbf{x},\mathbf{x})\, \mathrm{d}\mathbf{x} \,.
\]
\end{lemme}
The first statement is known (see \cite[Theorem 18]{Kuch} or \cite[Section 4]{BoEgRu}), but we recall the proof for clarity.
\begin{proof}
Since $G_z:=(H_{\cQ}-z)^{-1}$ is a bounded map from $L^2(\cG)\To D(H^{\cQ})$, and $(D(H^{\cQ}),\|\cdot\|_{H^2})$ is a closed subspace of $(\bigoplus_{e\in E} H^2(0,L_e),\|\cdot\|_{H^2})$ by usual trace estimates, and each embedding $H^2(0,L_e)\hookrightarrow L^2(0,L_e)$ is trace-class\footnote{In fact, for any bounded domain $\Omega\subset \R^n$ with the cone property, the embedding $H^k(\Omega)\hookrightarrow H^\ell(\Omega)$ is $p$-Schatten class if and only if $k-\ell>\frac{n}{p}$. This is a theorem of Gramsch \cite{gramsch}.}, then $G_z:L^2(\cG)\To L^2(\cG)$ is trace class.

For each $b\in \mathcal{B}$, let $\pi_b\in \mathscr{H}$ be the function with $\pi_b(b',x) =1$ if $b'=b$ or $\hat{b}$, and $0$ otherwise. The integral kernel of $\pi_bG_z\pi_b$ is $G_z((b,x),(b,y))$, which is jointly continuous (see e.g. \cite[Theorem 4.1]{BoEgRu}). It follows from \cite[Corollary III.10.2]{GoKr} or \cite[Chapter 30, Theorem 12]{Lax} that $\mathrm{Tr}(\pi_b G_z \pi_b) = \int_0^{L_b} G_z(\mathbf{x},\mathbf{x})\,\dd\mathbf{x}$.

If $\{(\phi_n^b)_{n\ge 0}\}_{b\in \mathcal{B}}$ is an orthonormal basis of $L^2(\cG)$, where $(\phi_n^b)_{n\ge 0}$ is an o.n.b. of $L^2(b)$ for each $b$, then $\mathrm{Tr}(G_z) = \frac{1}{2}\sum_{b\in \mathcal{B}} \sum_{n\ge 0} \langle \phi_n^b,G_z \phi_n^b \rangle = \frac{1}{2}\sum_{b\in \mathcal{B}} \mathrm{Tr}(\pi_b G_z \pi_b)$. The result follows.
\end{proof}

\begin{lemme}\label{lem:Herglotz}
Let $\cQ$ be a quantum graph, and let $\mathbf{x_0}, \mathbf{x_1}\in \cG$. Then, $G_z(\mathbf{x_0}, \mathbf{x_1})$ is analytic on $\C\setminus \R$, and for any $z\in \C^+=\{z:\Im z>0\}$, we have $\Im G_z(\mathbf{x_0}, \mathbf{x_0}) >0$.
\end{lemme}

\begin{proof}
Let $\varphi \in C_c^\infty( \R)$ be a function with $\varphi \ge 0$ and $\int_{\R} \varphi(y) \mathrm{d}y = 1$.

Fix $\mathbf{x}_j = (b_j,x_j)$,  with $0<x_j<L_{b_j}$, for $j=0,1$. Let $\psi^j_n(x) = n\varphi(n(x_j-x))$. If $\supp \varphi \subset [-K,K]$, then $\supp \psi^j_n \subset [\frac{-K}{n}+x_j,\frac{K}{n}+x_j]$, and taking $n=n(x_0, x_1)$, we may ensure that $\supp \psi^j_n \subset (0,L_{b_j})$ for $j=0,1$.

Define $f_n^j : \cG\To \R$ by $f_n^j(b,x)=0$ if $b\notin\{b_j,\hat{b}_j\}$, $f_n^j(b_j,x)=\psi^j_n(x)$, $f_n^j(\hat{b}_j,x) = f_n^j(b_j,L_{b_j}-x)$. 

Let us write $g_z(x,y)= G_z \big{(}(b_0,x), (b_1,y)\big{)}$. We extend $g_z$ to a function $\tilde{g}_z(x,y)$ on $\R^2$ taking value zero on $\R^2 \setminus  [0, L_{b_0}]\times[0,L_{b_1}]$. If $\rho_n(t,t')=n^2\varphi(nt)\varphi(nt')$, we have
\[
\langle f_n^0, G_z f_n^1 \rangle = (\rho_n\star g_z)(x_0,x_1) \,.
\]

Since the function $g_z$ is continuous at $(x_0,x_1)$, we deduce that 
\[
G_z(\mathbf{x_0},\mathbf{x_1}) = g_z(x_0,x_1) = \lim_{n\to\infty}
 (\rho_n\star g_z)(x_0,x_1) = \lim_{n\to\infty}
 \langle f^0_n, G_z f^1_n\rangle \,.
\]
More precisely, we have
\[
 G_z(\mathbf{x_0},\mathbf{x_1}) - \langle f^0_n, G_z f^1_n\rangle  = \int_{\R^2} \big{(} g_z(x_0,x_1) -g_z(x,y)\big{)}  \rho_n(x_0-x,x_1-y)\,\dd x\,\dd y.
\]

Now, if $Z$ is a compact subset of $\C\setminus \R$, equation \eqref{eq:BoundGreenImagin} along with \eqref{eq:GreenSym} tells us that there exists a constant $C(Z)$ such that $|g_z(x_0,x_1) -g_z(x,y)|\leq C(Z) \big{(} |x_0-x| + |x_1-y|\big{)}$. From the information we have on the support of $\rho_n$, we deduce that
\[
|G_z(\mathbf{x_0},\mathbf{x}_1) - \langle f^0_n, G_z f^1_n\rangle| \leq C(Z)\frac{ K}{n} \,.
\]
Since, by \eqref{eq:BoundGreenImagin}, $G_z(\mathbf{x_0}, \mathbf{x_1})$ is bounded on the compact subsets of $\C\setminus \R$, we deduce that $\langle f^0_n, G_z f^1_n\rangle $ is a family of functions which are holomorphic on $\C\setminus \R$, which is uniformly bounded on compact sets, and which converges on compact sets. Therefore, the limit $G_z(\mathbf{x_0},\mathbf{x}_1)$ is holomorphic.

Now, by the spectral theorem, we have $ \langle f_n^0, G_z f_n^0 \rangle= \int_\R \frac{1}{\lambda-z}\, \mathrm{d}\mu_{f_n^0}(\lambda)$, which has a positive imaginary part if $z\in\C^+$. Therefore, we must have $\Im G_z(\mathbf{x_0}, \mathbf{x_0}) \geq 0$.

But the map $z\mapsto \Im G_z(\mathbf{x_0}, \mathbf{x_0})$ is harmonic (since it is the imaginary part of a holomorphic function), so it cannot have a local minimum. Therefore, we must either have $\Im G_z(\mathbf{x_0}, \mathbf{x_0}) > 0$ for all $z\in \C^+$, or $\Im G_z(\mathbf{x_0}, \mathbf{x_0}) = 0$ for all $z\in \C^+$.

Let us show that this last possibility is absurd. If $\Im G_z (\mathbf{x_0}, \mathbf{x_0})= 0$, the argument leading to \eqref{eq:Cauchy-Schwarz} would imply that $\Im G_z (\mathbf{x_0}, \mathbf{y_0}) = 0$ for all $\mathbf{y_0}\in \cG$. Since $\mathbf{y_0}\mapsto G_z (\mathbf{x_0}, \mathbf{y_0})$ satisfies an ODE with a complex potential, this implies that  $G_z (\mathbf{x_0}, \mathbf{y_0}) = 0$ for all $\mathbf{y_0}\in \cG$.
Therefore, for all $f\in \mathscr{H}$, we would have $\big((H_{\cQ}-z)^{-1} f\big)(\mathbf{x_0}) =0$. So the image of $(H_{\cQ}-z)^{-1}$ would be included in the set of functions vanishing at $\mathbf{x_0}$. This is absurd, since $(H_{\cQ}-z)^{-1}$ is surjective onto $\mathscr{H}^{\cQ}$.
\end{proof}

\medskip

{\bf{Acknowledgements:}}
N.A. was supported by Institut Universitaire de France, by the ANR project GeRaSic ANR-13-BS01-0007 and by USIAS (University of Strasbourg Institute of Advanced Study).

M.I. was supported by the Labex IRMIA  during part of the redaction of this paper.

M.S. was supported by a public grant as part of the \textit{Investissement d'avenir} project,
reference ANR-11-LABX-0056-LMH, LabEx LMH. He thanks the Universit\'e Paris Saclay
for excellent working conditions, where part of this work was done.

\providecommand{\bysame}{\leavevmode\hbox to3em{\hrulefill}\thinspace}
\providecommand{\MR}{\relax\ifhmode\unskip\space\fi MR }
\providecommand{\MRhref}[2]{%
  \href{http://www.ams.org/mathscinet-getitem?mr=#1}{#2}
}
\providecommand{\href}[2]{#2}


\begin{thebibliography}{10} 
%

\bibitem{ATV}
M.~Ab\'ert, A.~Thom and B.~Vir\'ag, \emph{Benjamini-Schramm convergence and pointwise convergence of the spectral measure}, preprint,
\href{https://tu-dresden.de/mn/math/geometrie/thom/forschung/publikationen}%
{author homepage}.

\bibitem{Ag}
S.~Agmon, \emph{The $L_p$ approach to the Dirichlet problem. Part I : regularity theorems}, Ann. Scuola Norm. Sup. Pisa \textbf{13} (1959), 405--448. 

\bibitem{AWcan}
M.~Aizenman, S.~Warzel, \emph{The canopy graph and level statistics for random operators on trees}, Math. Phys. Anal. Geom. \textbf{9} (2006), 291--333.

\bibitem{ASW06}
M.~Aizenman, R.~Sims and S.~Warzel, \emph{Absolutely continuous spectra of quantum tree graphs with weak disorder}, Comm. Math. Phys. \textbf{264} (2006) 371--389.

\bibitem{AL}
D.~Aldous, R.~Lyons, \emph{Processes on unimodular random networks}, Electron. J. Probab. \textbf{12} (2007) 1454--1508.

\bibitem{AS2}
N.~Anantharaman, M.~Sabri, \emph{Quantum ergodicity on graphs : from spectral to spatial delocalization},  Ann. of Math. \textbf{189}, (2019), 753--835.

\bibitem{AS3}
N.~Anantharaman, M.~Sabri, \emph{Quantum ergodicity for the Anderson model on regular graphs}, J. Math. Phys. \textbf{58}, 091901 (2017).

\bibitem{AISW}
  N.~Anantharaman, M.~Ingremeau, M.~Sabri, B.~Winn, \emph{Absolutely
    continuous spectrum for quantum trees}, preprint. %
  \href{https://arxiv.org/abs/2003.12765}{\texttt{arXiv:2003.12765}}.

\bibitem{AHNR}
O.~Angel, T.~Hutchcroft, A.~Nachmias and G.~Ray, \emph{Hyperbolic and parabolic unimodular random maps}, Geom. Funct. Anal. Vol. \textbf{28} (2018) 879--942.

\bibitem{bai:cl}
  Z.~D.~Bai, \emph{Circular law}, Ann.\ Probab.\ \textbf{25} (1997) 494--529.

\bibitem{BS}
I.~Benjamini, O.~Schramm, \emph{Recurrence of distributional limits of finite planar graphs}, Electron. J. Probab. \textbf{6} (2001) 13 pp.

\bibitem{BLS}
I.~Benjamini, R.~Lyons and O.~Schramm, \emph{Unimodular random trees}, Ergod. Th. \& Dynam. Sys \textbf{35} (2015), 359--373.

\bibitem{BeKu13}
G.~Berkolaiko and P.~Kuchment, \emph{Introduction to quantum graphs}, SURV 186, AMS 2013.

\bibitem{BeKu12}
G.~Berkolaiko and P.~Kuchment, \emph{Dependence of the spectrum of a quantum graph on vertex conditions and edge lengths}, in ``Spectral Geometry'', 
A.H.~Barnett, C.S.~Gordon, P.A.~Perry and A.~Uribe Eds., Proc.\ Symp.\ Pure 
Mathematics \textbf{84} (2012).

\bibitem{BoEgRu}
J.~Bolte, S.~Egger, R.~Rueckriemen , \emph{Heat-kernel and Resolvent Asymptotics for Schr\"odinger Operators on Metric Graphs}, Applied Mathematics Research eXpress \textbf{1} (2015) 129-165.

\bibitem{Bornew}
C.~Bordenave, \emph{A new proof of Friedman's second eigenvalue Theorem and its extension to random lifts}, Annales scientifiques de l'\'Ecole normale 
sup\'erieure, to appear.
  \href{https://arxiv.org/abs/1502.04482}{\texttt{arXiv:1502.04482}}.

\bibitem{BDGHT}
G. Brito, I. Dumitriu, S. Ganguly, C. Hoffman, L. V. Tran, \emph{Recovery and 
Rigidity in a Regular Stochastic Block Model}, in ``Proceedings of the 
twenty-seventh annual ACM-SIAM symposium on Discrete algorithms'',
 pp. 1589--1601. SIAM (2016). Longer version on
  \href{https://arxiv.org/abs/1507.00930v3}{\texttt{arXiv:1507.00930v3}}.

\bibitem{Conway}
J.~B.~ Conway, \emph{A course in functional analysis}, Second edition, GTM 96, Springer 1990.

\bibitem{Davies}
E.~B.~Davies, \emph{Spectral Theory and Differential Operators}, CUP 1995.

\bibitem{Elek}
G.~Elek, \emph{On the limit of large girth graph sequences}, Combinatorica \textbf{30} (2010), 553--563.

\bibitem{eva:pde}
  L.~C.~Evans, \emph{Partial Differential Equations}, Graduate Studies
  in Mathematics \textbf{19}, American Mathematical Society 1998.
  
\bibitem{Gab}
P.~G\'abor, \emph{Probability and Geometry on Groups}, 
\href{https://math.bme.hu/~gabor/}{author homepage}.

\bibitem{gir:cl}
  V.~L.~Girko, \emph{Circular law}, Teor.\ Veroyatnosti i Primenen,
  \textbf{29} (1984) 669--679.

\bibitem{GoKr}
I. C.~Gohberg, M.G.~Krein, \emph{Introduction to the theory of linear nonselfadjoint operators in Hilbert spaces}, Translations of Mathematical Monographs, Vol 18, AMS.

\bibitem{gramsch}
B.~Gramsch, \emph{Zum Einbettungssatz von Rellich bei Sobolevr\"aumen},
Mathematische Zeitschrift, \textbf{106} (1968) 81--87.

\bibitem{Hor}
L.~H\"ormander, \emph{An Introduction to Complex Analysis in Several Variables}, Third edition, North-Holland 1990.

\bibitem{ISW}
  M.~Ingremeau, M.~Sabri and B.~Winn, \emph{Quantum ergodicity for large
    equilateral quantum graphs}, J.\ London.\ Math.\ Soc.\ \textbf{101},
    (2020) 82--109.
\bibitem{Ke59}
H.~Kesten, \emph{Symmetric random walks on groups}, Trans. Amer. Math. Soc. \textbf{92} (1959) 336--354.
%
\bibitem{Kirsch}
W.~Kirsch, \emph{An invitation to Random Schr\"odinger operators}, in Panor. Synth\`eses \textbf{25}, 1--119, Soc. Math. France 2008.

\bibitem{Klenke}
A.~Klenke, \emph{Probability Theory. A Comprehensive Course}, Second edition, Springer 2014.

\bibitem{KoSm99}
T.~Kottos and U.~Smilansky, \emph{Periodic orbit theory and spectral statistics for quantum graphs}, Ann. Physics \textbf{274} (1999) 76--124.

\bibitem{Kuch}
P.~Kuchment, \emph{Quantum graphs I. Some basic structures}, Waves in Random media. \textbf{14} (2004) 107--128.

\bibitem{Lax}
P. D.~Lax, \emph{Functional analysis}, Wiley 2002.

\bibitem{Maz}
V.~Maz'ya, \emph{Sobolev spaces. With applications to partial differential equations}, 2nd revised edition, Springer 2011.

\bibitem{Mc81}
B.~D.~McKay, \emph{The expected eigenvalue distribution of a large regular graph}, Linear Algebra Appl. \textbf{40} (1981) 203--216.

\bibitem{PT87}
J.~P\"oschel and E.~Trubowitz, \emph{Inverse Spectral Theory}, Academic Press 1987.

\bibitem{RS}
M.~Reed and B.~Simon, \emph{Methods of modern mathematical physics. I. Functional analysis}, Second Edition, Academic Press 1980.

\bibitem{RuSmi12}
R.~Rueckriemen, U.~Smilansky, \emph{Trace formulae for quantum graphs with edge potentials},  J. Phys. A \textbf{45} (2012), 475205, 14 pp. 

\bibitem{Simon15}
B.~Simon, \emph{A comprehensive course in analysis Part 3: Harmonic analysis.}, Amer.
Math. Soc. (2015).

\bibitem{Simon82}
B.~Simon, \emph{Schrödinger semigroups}, Bulletin of the American Mathematical Society, \textbf{7}(3), (1982).

\bibitem{tao:rmu}
  T.~Tao and V. Vu, \emph{Random matrices: universality of ESDs and the
    circular law}, Ann.\ Probab.\ \textbf{38} (2010) 2023--2065.

\bibitem{wig:cvo}
  E.~P.~Wigner, \emph{Characteristic vectors of bordered matrices with
    infinite dimension}, Ann.\ Math.\ \textbf{62} (1955) 548--564.
  
\bibitem{wig:otd}
  E.~P.~Wigner, \emph{On the distribution of roots of certain
    symmetric matrices}, Ann.\ Math.\ \textbf{67} (1957) 325--327.

\end{thebibliography}
\end{document}